\documentclass[11pt,3p]{elsarticle}
\usepackage{amssymb, amsthm}

\usepackage{tikz-cd}

\usepackage{mathrsfs}
\usepackage{mathtools} 
\usepackage{bigstrut}
\usepackage{tikz}
\usetikzlibrary{patterns}
\usetikzlibrary{matrix,arrows,decorations}
\usepackage{calc}

\usepackage{amsmath}
\usepackage{verbatim}
\usepackage{setspace}

\usepackage{hyperref}

\newtheorem{thm}{Theorem}[section]
\newtheorem{lem}[thm]{Lemma}
\newtheorem{coro}[thm]{Corollary}
\newtheorem{prop}[thm]{Proposition}

\theoremstyle{definition}
\newtheorem{df}[thm]{Definition}

\newtheorem{ex}[thm]{Example}

\newenvironment{newlist}
   {\begin{list}{}{\setlength{\labelsep}{0.25cm}
\setlength{\itemsep}{-0.1cm}
\setlength{\topsep}{0.1cm}
                   \setlength{\labelwidth}{0.65cm}
                      \setlength{\leftmargin}{0.9cm}}}
   {\end{list}}

\newcommand{\twiddle}[1]{\smash{\underset{\raise.375ex\hbox{$\smash\sim$}}
       {#1}}\vphantom{\underline{#1}}}

\newcommand{\twiddleodd}{
\smash{\underset{\lower.5ex\hbox{$\smash{\widetilde{\phantom{mm}}}
$}
}{\Zed_{2n+1}}}
}
\newcommand{\twiddleeven}{
\smash{\underset{\lower.65ex\hbox{$\smash{
\widetilde{\phantom{rr\!\!r}}}
$}
}{\Zed_{2n}}}}

\newcommand{\twiddlek}{\smash{\underset{\lower.4ex\hbox{$\smash{
\widetilde{\phantom{w\!\!u}}}
$}
}{\Zed_{k}}}}

\newcommand{\delp}{\alpha^+}  %{\delta^+}
\newcommand{\delm}{\alpha^-}     %    {\delta^-}

\newcommand{\class}[1]{\mathcal{#1}}
\newcommand{\cat}[1]{\boldsymbol{\mathscr{#1}}}  
\newcommand{\alg}[1]{\mathbf{#1}}
\newcommand{\fnt}[1]{\mathsf{#1}}
\newcommand{\E}{\fnt{E}}
\newcommand{\D}{\fnt{D}}
\newcommand{\ope}[1]{\mathbb{#1}}
\newcommand{\defn}[1]{\emph{#1}}
\newcommand{\A}{\alg{A}}
\newcommand{\B}{\alg{B}}
\newcommand{\C}{\alg{C}}
\newcommand{\M}{\alg{M}}
\newcommand{\X}{\alg{X}}
\newcommand{\Y}{\alg{Y}}
\newcommand{\Z}{\alg{Z}}

\newcommand{\K}{\class{K}}
\newcommand{\Tp}{\mathcal{T}}
\newcommand{\MT}{\twiddle{\M}}
\newcommand{\Zed}{\alg{Z}}
\newcommand{\two}{\boldsymbol 2}
\newcommand{\SA}{\cat{S\!A}}  
\newcommand{\CA}{\cat{A}}
\newcommand{\CX}{\cat{X}}
\newcommand{\CY}{\cat{Y}}  
\newcommand{\F}{\alg{F}}  
\newcommand{\w}{\omega}
\renewcommand{\epsilon}{\varepsilon}

\newcommand{\sub}[1]{_{_{\kern-.9pt{\scriptstyle #1}}}}
\newcommand{\medsub}[2]{#1\lower0.6ex\hbox{$\scriptstyle{#2}$}}
\newcommand{\eA}[1]{\medsub e {\kern-0.75pt\A\kern-0.75pt}(#1)}
\newcommand{\esub}[1]{\medsub e {\kern-0.75pt #1 \kern-0.75pt}}
\newcommand{\epsub}[1]{\medsub \varepsilon {\kern-1.25pt #1}}
\newcommand{\esubA}{\medsub e {\kern-0.75pt\A\kern-0.75pt}}

\DeclareMathOperator{\ISP}{\ope{ISP}}
\DeclareMathOperator{\IScP}{\ope{IS}_c\ope{P}^{+}}
 \DeclareMathOperator{\HSP}{\ope{HSP}}
 \DeclareMathOperator{\dom}{\rm dom}
\DeclareMathOperator{\graph}{\rm graph}
\DeclareMathOperator{\sgn}{\rm sgn}
  
\newcommand{\id}{\mathop{}\mathopen{}\mathrm{id}} 
\DeclareMathOperator{\img}{im}
\newcommand{\im}{\img}
\DeclareMathOperator{\End}{End}  
\newcommand{\PEZ}[1]{\End_{\text{p}} (\Zed_{#1})}

\newcommand{\proc}[1]{{\textsc{#1}}}
\newcommand{\tafa}{\emph{TAFA}}
\newcommand{\newcirc}{\,{\mathbin{\circ }\,\,}}
\newcommand{\xdoubleheadleftarrow}[1]{%
\leftarrow\mathrel{\mspace{-15mu}}{\xleftarrow{#1}}
}

\newcommand{\bvec}[1]{\text{\bf  {#1}}}
\newcommand{\eq}{\,{\approx} }
\renewcommand{\leq}{\leqslant}
\renewcommand{\geq}{\geqslant}

\begin{document}

\begin{frontmatter}

\title{Sugihara Algebras: Admissibility Algebras via the Test Spaces Method}

\author{\small \textit{Leonardo M. Cabrer}}
\address{}
\ead{lmcabrer@yahoo.com.ar}

\author[hap]{Hilary A. Priestley\corref{cor1}}
\ead{hap@maths.ox.ac.uk}
\address[hap]{Mathematical Institute, University of Oxford, Radcliffe Observatory Quarter,
Oxford, OX2 6GG, United Kingdom}

\cortext[cor1]{corresponding author}

%%%%%%%%%%%%%%%%%%%%%%%%%%%%%%%%%%%%%%%%%%%%%%

\begin{abstract}   
This paper studies finitely generated quasivarieties of Sugihara algebras.  These quasivarieties provide complete algebraic semantics for certain propositional logics associated with the relevant logic R-mingle. 
The motivation for the paper comes from the study of admissible rules.
Recent earlier work by the present authors, jointly with Freisberg and Metcalfe,
laid the  theoretical foundations for a feasible approach to this problem for a range of logics---the Test Spaces Method.  The method, based on natural duality theory, provides an algorithm to obtain 
 the
algebra of  minimum size on which admissibility of sets of rules can be tested.  
(In the most general case a set of such algebras may be needed rather than just one.)  
The method enables us to identify 
this  `admissibility algebra' 
for each quasivariety of Sugihara algebras which is  generated by an algebra whose underlying lattice is a finite chain.  
To achieve our goals, it was first necessary to develop a (strong)  duality for 
each of these quasivarieties.  
The dualities promise also to  
also provide a valuable
new tool for studying the structure of Sugihara algebras more widely.  
 \end{abstract}


\end{frontmatter}

%%%%%%%%%%%%%%%%%%%%%%%%%%%%%%%%%%%%%%%%%%%%%%%%%%%%%%%%%
\section{Introduction} \label{sec:intro}

This paper investigates finitely generated quasivarieties of Sugihara algebras.  
These algebras have attracted interest because they are associated with the
logic R-mingle, for which they provide complete algebraic semantics (see \cite{Du70,AB75,GM16}). 
The definition of these algebras is given in Section~\ref{Sec:Sugihara}. 

In a wider context, algebraic methods have been used for a long time to solve general  problems in logic for a variety of propositional logics, most famously classes of modal and intuitionistic logics.
However, often an algebraic semantics is too close to the
logic it models to provide a powerful tool.  
An exception is seen with the study of admissible rules when a suitable  algebraic semantics is available.  
Once it is observed that some  logical system has admissible non-derivable rules, 
it is natural to seek 
 a description of such admissible rules. 
Usually, this is done through an axiomatization.  That is, the provision of 
 a set of rules that, once these are
 added to those of the original system, its admissible rules become derivable; see \cite{Iem01,Roz92} (intuitionistic logic), \cite{Iem05,CM10} (intermediate logics),  \cite{Jer05} (transitive modal logics), \cite{Jer09a,Jer09b} (\L ukasiewicz many-valued logics).
In particular, Metcalfe~\cite{GM16} presented axiomatizations for the admissible rules of various fragments of the logic R-mingle. 
However, the problem of finding such axiomatizations for R-mingle remains open.  
This paper does not directly address this problem  but does develop tools whereby a rule's admissibility may be tested.

 Assume that we have a quasivariety of the form 
$\CA = \ISP(\M)$, where $\M$ is a finite algebra 
which has an $s$-element subset which generates it.
Then the free algebra $\F _{\CA}(s)$ on~$s$ generators 
can be used to test whether some given rule 
 for the associated logic is admissible (see for example \cite[Theorem~2]{CM15}).  
However, unless both~$\M$ and~$s$ are 
very small, this free algebra is likely to be extremely large:   no explicit
description may be available and often even the size
 cannot  be computed.  
Hence  this result on validating admissibility algebraically  is likely to be principally of theoretical interest.  It  does  prove that the admissibility problem is decidable, but it will not always be feasible to solve it  in practice.  

A breakthrough came with the demonstration 
by Metcalfe and R\"othlisberger  
\cite{MR13} and R\"othlis\-berger \cite{Roe13}
that validity of a rule 
 could be tested on a much smaller algebra 
(or set of algebras).
Proceeding 
syntactically, 
one searches for the minimal set of algebras, in a suitable multiset order, that generate the same quasivariety as does  
 $\F _{\CA}(s)$.
This \defn{admissibility set}   for~$\CA$  
has been proved to exist and to be unique up to isomorphism 
\cite[Theorem~4]{MR13}. 
 When this set is a singleton, we refer
to its unique element as the \defn{admissibility algebra}. 
In \cite{MR13} and \cite{Roe13}  an algorithm was developed to calculate the admissibility set. 
This algorithm is implemented in the \tafa\ package\footnote{Available at: \url{https://sites.google.com/site/admissibility/downloads}}.  
However, the success of 
\tafa\ in a given case hinges on the free algebra
$\F _{\CA}(s)$
not being too large  and  a search involving all subalgebras of that  free algebra
being computationally feasible.

A further advance in methods to  determine the admissibility set
came with the 
exploitation of duality theory.  It is very well known that the development of relational semantics revolutionised the investigation of modal and intuitionistic 
logics.  Relational semantics, and topological relational semantics, are especially 
powerful when the underlying algebraic semantics  are based on distributive 
lattices with additional operations which model, for example, a non-classical
implication or negation.   In this situation there will exist a dual equivalence 
between~$\CA$ and a category~$\CY$ of enriched Priestley spaces;  an associated 
relational semantics is obtained from~$\CY$ by suppressing the topology.  
What is significant is the computational advantages this passage to a dually equivalent category brings:  the functors  setting up the 
 dual adjunction between
$\CA$ and $\CY$ act like a `logarithm' from algebras to dual structures and an `exponential' in the other direction.  
However,  in general 
 we cannot expect  a duality based on enriched Priestley duality to yield
a smooth translation into an equivalent dual form of the strategy devised in
\cite{MR13}. 

 This problem can be overcome.  Instead of a hand-me-down duality based on Priestley duality, we need a duality more closely  tailored to the finitely generated
quasivariety $\ISP(\M)$.  The theory of natural dualities is available,  and supplies exactly the right machinery, provided we can set up a duality 
based on a strongly dualising alter ego~$\MT$. 
A preliminary exploration of the idea of using natural dualities to 
study admissible rules
was undertaken by Cabrer and Metcalfe~\cite{CM15}, drawing 
on well-known natural dualities  (for   De Morgan algebras, in particular).
Encouraged by the 
evidence in \cite{CM15},
Cabrer \textit{et al.}~\cite{CFMP} undertook 
 a more extensive study.
The setting is finitely generated quasivarieties $\ISP(\M)$ for which strong 
dualities
are available.  
 Under a strong duality,
 injective homomorphisms correspond to surjective morphisms on the dual side,
and surjective homomorphism correspond to embeddings.  In addition, the dual space of the free algebra on~$s$ generators is the $s^{\text{th}}$ power of the alter ego $\MT$.  These properties enable \proc{MinGenSet} and 
\proc{SubPreHom}, 
the component algorithms from {\tafa}, 
 to be recast 
in equivalent dual forms.
The Test Spaces Method (TSM for short)  can thereby  be formulated and validated.
 In outline TSM  is
an algorithm for determining an admissibility algebra (or admissibility set):
it tells us how to find a \defn{test space}  that  is the  dual space of the 
admissibility algebra we seek (or, likewise,  a set of test spaces).  Only at the final step 
does one  pass back from the dual category to the original category $\ISP(\M)$.

The theory in \cite{CFMP} is accompanied by a suite of case studies illustrating the Test Spaces Method in action, with the focus on computational
feasibility.
 The studies include  a reprise on  De Morgan algebras 
and 
progress to more complex examples of similar type. This confirmed 
that TSM,
 with or without computer assistance, can successfully find
admissibility algebras in examples  which have been shown to be beyond \tafa\'s reach. 
The largest of the admissibility algebras found in \cite{CFMP}, that for involutive Stone algebras, has 20 elements.  There the free algebra on two generators with which \tafa\ had to contend has $1\, 483\, 648$ elements.  
Small wonder that \tafa\  was not up to the task.

We shall reveal in this paper  
that  the variety $\SA$ of Sugihara algebras provides a powerful  demonstration of the capabilities of  the Test Spaces Method.  
We seek  to   describe the admissibility algebra for each  subquasivariety $\SA_k$ of~$\SA$ which is generated by a finite subdirectly
 irreducible algebra:  $\SA_k
= \ISP(\Zed_k)$, 
  where  the lattice reduct of  $\Zed_k$ is a $k$-element  chain.
The family of quasivarieties $\{\SA_k\}$ can be 
viewed (see Figure~\ref{fig:QuasiVar}) as forming  two interlocking chains
$\{\SA_{2n+1}\}_{n\geq 0}$ (the \emph{odd case}) and $\{ \SA_{2n}\}_{n\geq 1} $ (the \emph{even case}).  
   A minimal generating set for $\Zed_{2n+1}$ has  $(n+1)$ elements  
and   one for $\Zed_{2n}$ has~$n$ elements.
So both the size of  the generating algebra  $\Zed_k$ for $\SA_k$ and the size of 
a minimal generating subset  for this algebra 
  tend to infinity with~$k$. Nevertheless, we are able to achieve our goal of identifying the admissibility algebra for each $\SA_k$.
To accomplish this 
 we first
 have to develop the strong dualities we need.  
The dualities  we obtain are pleasingly simple and 
exhibit a uniform pattern in the odd case and in the even case.
This  uniformity works to our advantage 
in executing TSM.  Moreover, these dualities are  of potential value beyond the application that led us to derive~them.

The existing literature includes studies of Sugihara monoids, in which the language
contains a constant, as well as of Sugihara algebras.  Note in particular 
\cite{BP89,FG17}.
 Our techniques apply equally well to the monoid case.  
Only minor adaptations to the results and their proofs would be needed.

The paper is organised as follows.  
Section~\ref{Sec:Sugihara} assembles  definitions and algebraic facts about Sugihara algebras.   
Some basic results are well known, but the discussion of partial endomorphism monoids, 
which 
underpin our dualities,  is new. 
We  treat the theory of natural dualities in black-box fashion:
 Section~\ref{Sec:NatDual} summarises the bare essentials and supplies references.
This primer is followed by  a brief 
exposition of the Test Spaces Method as presented in~\cite{CFMP}.  Thereafter, we work exclusively with Sugihara algebras,  separating the odd and even cases
since the differences between them are great enough to make this advantageous. 
  Sections~\ref{Sec:Duality-SugOdd} (duality) and~\ref{Sec:TSM-SugOdd} (TSM)
cover the odd case and, likewise, Sections~\ref{Sec:Duality-SugEven} and~\ref{Sec:TSM-SugEven} the even case.
 Propositions~\ref{Prop:AnOdd}  and~\ref{Prop:AnEven} 
present our admissibility algebras. 
Section~\ref{sec:sumup}  treats the odd and even cases together. We take stock of what we have achieved by  employing the Test Spaces Method as opposed to  its algebraic counterpart. 
Table~\ref{table:Sugcasestud} 
compares the sizes of  our admissibility algebras with the sizes of the test spaces from which we derived them and with lower bounds for the sizes of the 
associated  free algebras.
Already  for small values of~$k$  
the data are very striking.  
We see here  the double benefit that the TSM approach has conferred:  working with  a `logarithmic' duality and the ability to test for admissibility on an algebra of minimum size.

Having, we believe, conclusively demonstrated the virtues of admissibility algebras 
we analyse their structure more closely.  We present a canonical generating set 
with $s=\left[\frac{k+1}{2}\right]$ elements for the admissibility algebra 
$\B_k$ for the quasivariety $\SA_k$.  This both allows us to view $\B_k$ in a standalone manner and also allows us to capitalise on the way that
$\B_k$ arises both as a quotient and as a subalgebra of the free algebra on $s$ generators in $\SA_k$. 
We conclude by offering  a  glimpse of the application of our results to admissible rules.
 
 We would like to thank Prof.~George Metcalfe twice over.
Firstly, he  introduced  
 us to the  problem of finding admissibility algebras for Sugihara algebras.    Secondly, we have benefitted from  his ongoing interest in our work as this has progressed
and from comments that have influenced %greatly help us 
 the presentation of our  results.   In particular he has  contributed to  our understanding  of the  consequences of these results  for 
 the study of admissible rules of R-mingle logics.

\section{Sugihara algebras: preliminaries}\label{Sec:Sugihara}

Here we introduce the quasivarieties $\SA_k$ of Sugihara algebras 
in which we are interested and present  algebraic facts about their generating algebras.

We assume familiarity  with the basic notions of universal algebra, 
for which we recommend~\cite{BS12} for reference.
Let $\CA$ be a non-trivial class of algebras over some common language.  
Then $\CA$  is a \defn{variety} if it is the class of models for a set  of equations and a 
\defn{quasivariety} if it is the class of models for a set  of quasi-identities.   
Our focus in this paper is on quasivarieties (for which \cite{Gor98} provides a 
background reference).  We are interested in 
the special case 
of a class obtained from a non-trivial 
\emph{finite} algebra~$\M$.  
Specifically,  the class $\ISP(\M)  $ is a quasivariety
 (where  the class operators $\mathbb I$, $\mathbb{S}$ and $\mathbb{P}$
have their usual meanings:  they denote respectively the formation of 
  isomorphic images, subalgebras and products over a non-empty index set).  
We note that natural duality, in its simplest form and as we employed it in
\cite{CFMP}, applies to quasivarieties $\ISP(\M)$.  
Free algebras play a central role in our investigations.  
A well-known  result  from universal algebra
tells us that free algebras exist in any class  $\ISP(\M)$ and  serve also as free algebras in the variety generated by~$\M$
\cite[Chapter~II]{BS12}.

We now proceed to definitions.
Let $\Zed=(Z;\wedge,\vee,\to,\neg)$ be the algebra
whose universe $Z$ is the set of  integers, $(Z;\wedge,\vee)$  
is the lattice derived from the natural order on~$Z$,
and~$\neg$ and~$\to$ are defined by $\neg a = -a$ and 
\[
a\to b = \begin{cases}
(-a) \vee b & \mbox{if }a\leq b,\\
(-a) \wedge b & \mbox{otherwise.}
\end{cases}
\]
Modulus, given by  $|a| := a \lor \neg a$, or alternatively by $a \to a$, defines a term function.  This elementary fact is important later.  
Usually a fusion operator $\cdot$ is also considered.   It can be defined   by 
 $a\cdot b = \neg (a \to \neg b)$.  We shall not  make use of it.  
%%%%%%%%%%%  

The variety of \defn{Sugihara algebras}  may be defined to be the variety generated by~$\Zed$ (see \cite{BD86}).  We denote it by~$\SA$.
We define a family  $\{ \Zed_k\}_{k\geq 1}$ of Sugihara subalgebras of $\Zed$ as follows:

%\begin{table}[h]
\begin{center}
\begin{tabular}{lll}
subalgebra  \qquad & universe  \quad &    \\[.5ex]\hline
   &&\\[-1.5ex]
$\Zed_{2n+1}$  & $%[-n,n]:=
\{\,a\in Z\mid -n \leq a \leq n\, \}$   & ($n = 0,1,2 , \ldots$) \\
 $\Zed_{2n}$  & $\Zed_{2n+1} \setminus \{ 0\} $  
&  ($n = 1,2, \ldots $)\\
\end{tabular}
\end{center}
We shall henceforth not distinguish in our notation between an algebra 
$\Zed_k$ and its universe.  Later we shall often encounter products
$(\Zed_k)^2$,  with   $k=2n$ or 
$k=2n+1$.  
To  avoid cumbersome notation we shall write $( \Zed_k)^2$ as 
$\Zed_k^2 $.  
A similar abuse of notation will be adopted for other powers.  

Let $\SA_{k}=\ISP(\Zed_{k})$ denote the quasivariety generated~$\Zed_{k}$. 
Here 
$\SA_1$ is the trivial quasivariety and $\SA_2$ is term-equivalent to Boolean algebras.  
The quasivariety $\SA_k$ is the algebraic counterpart of R-mingle's 
axiomatic extension $\text{RM}_k$;  see~\cite{Du70}.
Because of degeneracies,  the quasivarieties $\SA_k$ for $k < 6$ do not fully exhibit the
features seen with larger values of~$k$.  
For convenience we shall refer to consideration of the classes $\SA_{2n}$, 
for $n \geq 1$, as the even case and the study of $\SA_{2n+1}$, for $n \geq 0$, 
as the odd case.  
(We note that there are references in the literature to even (respectively, odd) Sugihara 
algebras for members of these classes.) 
We shall now proceed to a detailed discussion of basic algebraic properties,
separating the even and odd cases as necessary.   
 %%%%%%%%%%%%%%%%%%%%%%%

By the definition of $\to$,  a non-empty subset $A$ of $Z$ is the universe of a subalgebra 
of $\Z$ if and only if it is closed under $\wedge, \vee,$ and $\neg$. 
Hence, any union of sets of the form $\{ a, \neg a\} $,  for  $a \in \Zed_{2n}$, 
is the universe of a subalgebra of~$\Zed_{2n}$.
This observation leads to   the following proposition.

\begin{prop}[subalgebras] \label{prop:subalg}
Let $n\geq 1$.   
\begin{enumerate}[(i)] 
\item There is an isomorphism from the lattice of subalgebras of $\Zed_{2n}$, 
augmented with the empty set, to the powerset $\mathcal{P}(\{1,2,\ldots, n\})$. 
Moreover, each proper subalgebra of $\Zed_{2n}$ is isomorphic to $\Zed_{2m}$ 
for some $m< n$.
\item There is an isomorphism  from the lattice of subalgebras of $\Zed_{2n+1}$, 
augmented with the empty set, 
to $\mathcal{P}(\{0,1,2,\ldots, n\})$. Moreover, each proper
subalgebra of $\Zed_{2n+1}$ is isomorphic to $\Zed_{k}$ for some $k< 2n+1$.
\end{enumerate}
\end{prop}

\begin{prop}[homomorphisms, even case] \label{prop:phomEven} 
Let $m,n\geq 1$ and $h$ be a homomorphism from $\Zed_{2m}$  into $\Zed_{2n}$. 
Then $h$ is injective and $m\leq n$.

Furthermore, the only endomorphism of $\Zed_{2n}$ is the identity map.
\end{prop}
\begin{proof}
Assume there exist $k<\ell  \in A $ with $h(k)=h(\ell)=r\in \Zed_{2n}$. 
Then 
\[
\neg r\vee r=h(\neg k\vee \ell)=h(k\to \ell) =r\to r 
=h(\ell \to k)=h( \neg \ell \wedge k)= \neg r \wedge r,
\] 
that is, $\neg r= r $, which is impossible
in $\Zed_{2n}$.  
So $h$ is 
injective, and hence  an isomorphism from $\Zed_{2m}$ onto 
a subalgebra of $\Zed_{2n}$ isomorphic to $\Zed_{2m}$.  Necessarily $m\leq n$.

The final statement is immediate from what we have proved already.
\end{proof}
 
\begin{prop}[homomorphisms, odd case]  \label{prop:phomOdd} 
Let $m,k \geq 1$ and~$h$ be a homomorphism 
from 
$\Zed_{2m+1}$  into $\Zed_{k}$. Then $k=2n+1$ for some $n\geq 0$ and if $a,b\in \Zed_{2m+1}$ are distinct and 
such that $h(a)=h(b)$, then $h(a)=0=h(b)$.
%Therefore 
Moreover,
two endomorphisms $h_1,h_2\in\End(\Zed_{2m+1})$ are equal if and only if their images coincide and 
there is a bijection between $\End(\Zed_{2m+1})$ and subalgebras of $\Zed_{2m+1}$ that contain $0$.
\end{prop}
\begin{proof}  
Since $0\in\Zed_{2m+1}$ is the unique element $x$ such that $\neg\neg x= x$ 
necessarily $h(0)=0$, whence $k=2n+1$ for some $n\geq 0$. 

Now consider  
$\Zed_{2m+1}\setminus h^{-1}(0)$. 
This is a subalgebra $\A$ of $\Zed_{2m+1}$ with $0\notin \A$. 
By Proposition~\ref{prop:subalg},  $\A\cong \Zed_{2\ell}$ for some $\ell >0$. 
The remaining claims follow from this together with
 Proposition~\ref{prop:phomEven}. 
\end{proof}

%%%%%%%%%%%%%%%%%%%%%

We  now have enough information to describe 
the lattices of congruences of $\Zed_{2n}$ and $\Zed_{2n+1}$.  
For $m=0,\ldots ,n-1$, let $\eq_m\,\subseteq\Zed_{2n+1}^2$ 
be the equivalence relation
 defined by 
$a \eq_m b$ if and only if $a=b$ or $a,b\in \Zed_{2m+1}$;
here $\eq_0$ is just the diagonal relation.
Each $\eq_m$ is algebraic, 
that is, a subalgebra of $\Zed_{2n+1}^2$. Indeed, each $\eq_m$ is the kernel of a 
homomorphism from $\Zed_{2n+1}$ into $\Zed_{2m+1}$.  
To avoid overloading the notation,  $\eq_m$  will denote both the relation 
 on  $\Zed_{2n+1}$ defined above and also its 
 restriction to  $\Zed_{2n}$.

\begin{prop}[congruences]\label{Prop:Congruences}
For each $n\geq 1$ and ${k\in\{2n,2n+1\}}$ the lattice of congruences of $\Zed_{k}$ is the chain 
\[ 
\eq_0\subseteq\, \eq_1\,\subseteq \cdots\subseteq\, \eq_{n-1}\, \subseteq \eq_{n}=\Zed_{k}  ^2.
\]
\end{prop}

\begin{proof}  
For $k$  odd, the result follows straightforwardly from Proposition~\ref{prop:phomOdd}.

Now assume that $k=2n$ for some $n> 1$ (the case $n=1$ is trivial).
Let $\theta$ be a non-trivial congruence of $\Zed_{2n}$. 
Let $r=\max\{\,m\mid m\,\theta\,\neg m\,\}$.  
By construction 
$\Zed_{2r}^2\subseteq \theta$. 
If $r=n$, then $\theta=\Zed_{2n}^2$. If $r<n$, it follows that $\eq_r\subseteq \theta$. To complete the proof we derive the reverse inclusion.   Let $x\,\theta\, y$ be such that $x\neq y$.
Without loss of generality assume that $x>y\geq \neg x$.
Then
\[
x= \neg y \vee x = y \to x\ \theta\  x\to y = \neg x \wedge y = \neg x.
\]
Hence $ x\leq r$ and $x\eq_r \neg x$. Therefore $y= y\wedge x \eq_r y\wedge \neg x = \neg x$, and so $x\eq_r y$.
\end{proof}

Proposition~\ref{Prop:Congruences} leads to  the next 
result.  It complements \cite[Lemma~1.1]{BD86}
which asserts, \textit{inter alia},
that every finite subdirectly irreducible subalgebra of~$\SA$ is of the form 
$\Zed_k$ for some $k \geq 1$.  

\begin{prop}\label{IrrIdx}
For each $k\geq1$ every  subalgebra of $\Zed_k$  
is subdirectly irreducible.
\end{prop}
%%%%%%%%%%%%%%%%%%%%%%%%%%%%%%%%%%%%%%%%%%%%%%%%%%%%

We  now describe 
the 
lattice structure of the family of subquasivarieties $\{\SA_k\}$.  
From our results so far we may anticipate some connections
between the odd and even cases. 
Observe that 
$\Zed_{2n-1}\notin \ISP(\Zed_{2n}) =\SA_{2n}$.  
 It follows that the quasivarieties $\SA_k$ are ordered as indicated in 
 Figure~\ref{fig:QuasiVar}.     
We remark as an aside that $\SA_{2n+1}=\ISP(\Zed_{2n+1})$ coincides with $\HSP(\Zed_{2n+1})$ and 
$\HSP(\Zed_{2n})$ coincides with $\ISP(\Zed_{2n},\Zed_{2n-1})$.  
 We do not need these facts in the present paper.

\begin{figure}
\begin{center}
	\begin{tikzpicture}
	%odd
		\node [label=left:{$\SA_7$}] (S7) at (0,4) {};
		\node [label=left:{$\SA_5$}] (S5) at (0,3) {};
		\node [label=left:{$\SA_3$}] (S3) at (0,2) {};
		\node [label=left:{$\SA_1$}] (S1) at (0,1) {};

   		\draw [shorten <=-2pt, shorten >=-2pt] (S7) -- (S5);
		\draw [shorten <=-2pt, shorten >=-2pt] (S5) -- (S3);
	
		\draw (S1) circle [radius=2pt];
		\draw (S3) circle [radius=2pt];
		\draw (S5) circle [radius=2pt];
		\draw (S7) circle [radius=2pt];
	%% even
		\node [label=right:{$\SA_6$}] (S6) at (1,3.5) {};
		\node [label=right:{$\SA_4$}] (S4) at (1,2.5) {};
		\node [label=right:{$\SA_2$}] (S2) at (1,1.5) {};

   		\draw [shorten <=-2pt, shorten >=-2pt] (S2) -- (S4);
		\draw [shorten <=-2pt, shorten >=-2pt] (S6) -- (S4);

		\draw (S2) circle [radius=2pt];
		\draw (S4) circle [radius=2pt];
		\draw (S6) circle [radius=2pt];
		
%%cross connections
		\draw [shorten <=-2pt, shorten >=-2pt] (S2) -- (S3);
		\draw [shorten <=-2pt, shorten >=-2pt] (S1) -- (S2);
		\draw [shorten <=-2pt, shorten >=-2pt] (S4) -- (S5);
		\draw [shorten <=-2pt, shorten >=-2pt] (S6) -- (S7);
%% Top part
		\node [label=left:{$\SA$}] (S) at (0,5) {};
		\node  (S-) at (1,4.5) {};
		\draw (S) circle [radius=2pt];
   		\draw [shorten <=-2pt, shorten >=-2pt,dotted] (S7) -- (S);
  		\draw [shorten <=-2pt, shorten >=-2pt,dotted] (S-) -- (S);
 		\draw [shorten <=-2pt, shorten >=-2pt,dotted] (S-) -- (S6);
		
	\end{tikzpicture}
\end{center}
\caption{Quasivarieties $\SA_k$\label{fig:QuasiVar}}
\end{figure}

%%%%%%%%%%%%%%%%%%%%%%%%%%%%%%%%%%%%%%%%%%%%%%%%%%%%%%%

We shall now study partial endomorphisms.  
When we refer to a map~$e$ as a \defn{partial endomorphism} of $
Z\Zed_k$ we mean that $h \colon \dom h \to \img h$, where
$\dom h$ and $\img h$ are subalgebras of $
\Zed_k$ and thereby non-empty; total maps (the endomorphisms) are included. 
We shall denote by $\PEZ{k}$ the set of partial endomorphisms of
$\Zed_k$.
The composition of two  elements of $\PEZ{k}$ may have empty domain.
By a slight abuse of notation we shall when expedient  refer to
$\PEZ{k}$
as a monoid,  under the operation $\circ$ of composition.  In this
situation we tacitly
add in the empty map.  This allows us also to talk about the submonoid
of $\PEZ{k}$ generated by a given
subset.

We first record  an easy corollary of Proposition~\ref{prop:subalg}.

\begin{coro} \label{cor:k-trans}
 Let $1 \leq m \leq  n$.  Let
$0 < b_1 < b_2 < \cdots < b_m \leq n$ and 
$0 < c_1 < c_2 < \cdots < c_m\leq n$.
\begin{enumerate}
\item[{\rm (i)}] 
There exists an invertible partial endomorphism $e$ of~$\Zed_{2n}$
 such that $e(b_i) =  c_i$ for 
$1 \leq i\leq m$.  

\item[{\rm (ii)}]  There exists an invertible partial endomorphism $e$ of
$\Zed_{2n+1}$ such that $e(b_i) =  c_i$ for 
$1 \leq i\leq m$ and $e(0) = 0$.
\end{enumerate}
 \end{coro}

\begin{proof}  
By Proposition~\ref{prop:subalg}(i), the subalgebras 
$\B$ and $\C$
 of $\Zed_{2n}$ generated  by $\{b_1,b_2 , \ldots , b_m\}$ and $\{c_1,c_2 , \ldots , c_m\}$, respectively,  have $2m$ elements and each is  isomorphic to $\Zed_{2m}$. Hence they are isomorphic. 
Since morphisms are order preserving any isomorphism $e\colon \B\to \C$
 is such that $0<e(b_1)<\cdots<e(b_m)$. Therefore, $e(b_i)=c_i$ for %each
 $1\leq i\leq m$.   Hence (i) holds.

The partial endomorphism~$e$ of~$\Zed_{2n}$ in (i)  also belongs to $\PEZ{2n+1}$.
No conflict arises if we extend this map by defining $e(0) = 0$.  Hence (ii) holds.
\end{proof}

We now analyse partial endomorphisms more closely, 
with the objective of  identifying amenable generating sets for the monoids
$\PEZ {2n}$ and $\PEZ{2n+1}$.
For each 
$1< i\leq n$, we define  $f_i\colon \Zed_{2n}\setminus\{i,-i\} \to \Zed_{2n}$ 
(see Figure~\ref{fig:pe-Z6}) by 
\[
f_i(a) = 
\begin{cases}
i& \mbox{if } a = i-1,\\
-i& \mbox{if } a = -(i-1),\\
a & \mbox{otherwise}
\end{cases}
\]
and 
$g\colon \Zed_{2n}\setminus\{1,-1\} \to \Zed_{2n} $ by 
\[
g(a) = \begin{cases}
a-1 & \mbox{if } a>0, \\
a+1& \mbox{otherwise.}
\end{cases}
\]
 We have excluded  the case $n=1$ here;   the only 
element of
$\PEZ {2}$ is the identity endomorphism. 
 Henceforth we adopt a commonsense 
convention in relation to degeneracies of this sort, not explicitly excluding 
vacuous scenarios for example.   

Let $F_n$ denote  the submonoid of 
$\PEZ {2n}$
which is generated by $f_2,\ldots, f_{n},g$.

\begin{figure}[ht]
\begin{center}

	\begin{tikzpicture}[scale=.8]

%labels, once only
%vertical axis 
\begin{scope}[xshift=-4cm]
\begin{scope}[xshift=-.6cm]
%\node [label=left:{\small $3$\   -}]   (a3) at  (-8,4) {};
\node [label=left:{\small  $3$\  -}]   (a2) at  (-8,3) {};
\node [label=left:{\small  $2$\ -}]   (a1) at  (-8,2) {};
\node [label=left:{\small  $1$\   -}]   (a0) at  (-8,1) {};
\node [label=left:{\small  $-1$\  -}]   (a-1) at  (-8,0) {};
\node [label=left:{\small  $-2$\  -}]   (a-2) at  (-8,-1) {};
\node [label=left:{\small  $-3$\ -}]   (a-3) at  (-8,-2) {};

\draw[thin] (-8.35,3.2)--(-8.35,-2.2);
\end{scope}
\end{scope}

%f2 
\begin{scope}[xshift=-2cm]

	%%domain
		\node %[label=left:{$3$}] 
(2lf3) at (-7,3) {};
		\node %[label=left:{$2$}] 
(2lf2) at (-7,2) {};
		\node %[label=left:{$1$}] 
(2lf1) at (-7,1) {};
		\node %[label=left:{$-1$}] 
(2lf-1) at (-7,0) {};
		\node %[label=left:{$-2$}] 
(2lf-2) at (-7,-1) {};
		\node %[label=left:{$-3$}]
 (2lf-3) at (-7,-2) {};
\node  (lab-f2) at (-6.7,-2.5)  {$f_2$};

   		\draw [shorten <=-2pt, shorten >=-2pt] (2lf1) -- (2lf2);
		\draw [shorten <=-2pt, shorten >=-2pt] (2lf3) -- (2lf2);
		\draw [shorten <=-2pt, shorten >=-2pt] (2lf-1) -- (2lf1);
		\draw [shorten <=-2pt, shorten >=-2pt] (2lf-2) -- (2lf-1);
		\draw [shorten <=-2pt, shorten >=-2pt] (2lf-3) -- (2lf-2);

		\filldraw (2lf1) circle [radius=2pt];
		\draw (2lf2) circle [radius=2pt];
		\filldraw (2lf3) circle [radius=2pt];
		\filldraw (2lf-1) circle [radius=2pt];
		\draw (2lf-2) circle [radius=2pt];
		\filldraw (2lf-3) circle [radius=2pt];
	%% codomain
		\node %[label=right:{$3$}] 
(2rf3) at (-6,3) {};
		\node %[label=right:{$2$}]
 (2rf2) at (-6,2) {};
		\node %[label=right:{$1$}]
 (2rf1) at (-6,1) {};
		\node %[label=right:{$-1$}] 
(2rf-1) at (-6,0) {};
		\node %[label=right:{$-2$}] 
(2rf-2) at (-6,-1) {};
		\node %[label=right:{$-3$}] 
(2rf-3) at (-6,-2) {};

   		\draw [shorten <=-2pt, shorten >=-2pt] (2rf1) -- (2rf2);
		\draw [shorten <=-2pt, shorten >=-2pt] (2rf3) -- (2rf2);
		\draw [shorten <=-2pt, shorten >=-2pt] (2rf-1) -- (2rf1);
		\draw [shorten <=-2pt, shorten >=-2pt] (2rf-2) -- (2rf-1);
		\draw [shorten <=-2pt, shorten >=-2pt] (2rf-3) -- (2rf-2);

		\draw (2rf1) circle [radius=2pt];
		\filldraw (2rf2) circle [radius=2pt];
		\filldraw (2rf3) circle [radius=2pt];
		\draw (2rf-1) circle [radius=2pt];
		\filldraw (2rf-2) circle [radius=2pt];
		\filldraw (2rf-3) circle [radius=2pt];
	%%arrows
		\draw[->] (2lf1) -- (2rf2);
		\draw[->] (2lf-1) -- (2rf-2);
		\draw[->] (2lf3) -- (2rf3);
		\draw[->] (2lf-3) -- (2rf-3);
%%%%%%%%%%%%%%%%%%%%%%%%%%%%%%%%%%%%%%%%%%%%%%
\end{scope}

	%%%%% f3

\begin{scope}[xshift=-1cm]
	%%domain
		\node %[label=left:{$3$}] 
(lf3) at (-4,3) {};
		\node %[label=left:{$2$}] 
(lf2) at (-4,2) {};
		\node %[label=left:{$1$}] 
(lf1) at (-4,1) {};
		\node %[label=left:{$-1$}] 
(lf-1) at (-4,0) {};
		\node %[label=left:{$-2$}] 
(lf-2) at (-4,-1) {};
		\node %[label=left:{$-3$}]
 (lf-3) at (-4,-2) {};

   		\draw [shorten <=-2pt, shorten >=-2pt] (lf2) -- (lf3);
		\draw [shorten <=-2pt, shorten >=-2pt] (lf1) -- (lf2);
		\draw [shorten <=-2pt, shorten >=-2pt] (lf-1) -- (lf1);
		\draw [shorten <=-2pt, shorten >=-2pt] (lf-2) -- (lf-1);
		\draw [shorten <=-2pt, shorten >=-2pt] (lf-3) -- (lf-2);

		\filldraw (lf1) circle [radius=2pt];
		\filldraw (lf2) circle [radius=2pt];
		\draw (lf3) circle [radius=2pt];
		\filldraw (lf-1) circle [radius=2pt];
		\filldraw (lf-2) circle [radius=2pt];
		\draw (lf-3) circle [radius=2pt];
	%% codomain
		\node %[label=right:{$3$}] 
(rf3) at (-3,3) {};
		\node %[label=right:{$2$}]
 (rf2) at (-3,2) {};
		\node %[label=right:{$1$}]
 (rf1) at (-3,1) {};
		\node %[label=right:{$-1$}] 
(rf-1) at (-3,0) {};
		\node %[label=right:{$-2$}] 
(rf-2) at (-3,-1) {};
		\node %[label=right:{$-3$}] 
(rf-3) at (-3,-2) {};
\node (lab-f3) at (-3.7, -2.5) {$f_3$};

   		\draw [shorten <=-2pt, shorten >=-2pt] (rf1) -- (rf2);
		\draw [shorten <=-2pt, shorten >=-2pt] (rf3) -- (rf2);
		\draw [shorten <=-2pt, shorten >=-2pt] (rf-1) -- (rf1);
		\draw [shorten <=-2pt, shorten >=-2pt] (rf-2) -- (rf-1);
		\draw [shorten <=-2pt, shorten >=-2pt] (rf-3) -- (rf-2);

		\filldraw (rf1) circle [radius=2pt];
		\draw (rf2) circle [radius=2pt];
		\filldraw (rf3) circle [radius=2pt];
		\filldraw (rf-1) circle [radius=2pt];
		\draw (rf-2) circle [radius=2pt];
		\filldraw (rf-3) circle [radius=2pt];
	%%arrows
		\draw[->] (lf1) -- (rf1);
		\draw[->] (lf-1) -- (rf-1);
		\draw[->] (lf2) -- (rf3);
		\draw[->] (lf-2) -- (rf-3);
%%%%%%%%%%%%%%%%%%%%%%%%%%%%%%%%%%%%%%%%%%%%%%%%%%%%%%%%			
	\end{scope}

\begin{scope}[xshift=-0.5cm]
	%%%%% g
	%%domain
		\node %[label=left:{$3$}] 
(lg3) at (-1,3) {};
		\node %[label=left:{$2$}] 
(lg2) at (-1,2) {};
		\node %[label=left:{$1$}] 
(lg1) at (-1,1) {};
		\node %[label=left:{$-1$}] 
(lg-1) at (-1,0) {};
		\node %[label=left:{$-2$}] 
(lg-2) at (-1,-1) {};
		\node %[label=left:{$-3$}]
 (lg-3) at (-1,-2) {};
\node (lab-g)  at  (-.7,-2.5) {$g$};

   		\draw [shorten <=-2pt, shorten >=-2pt] (lg1) -- (lg2);
		\draw [shorten <=-2pt, shorten >=-2pt] (lg3) -- (lg2);
		\draw [shorten <=-2pt, shorten >=-2pt] (lg-1) -- (lg1);
		\draw [shorten <=-2pt, shorten >=-2pt] (lg-2) -- (lg-1);
		\draw [shorten <=-2pt, shorten >=-2pt] (lg-3) -- (lg-2);

		\draw (lg1) circle [radius=2pt];
		\filldraw (lg2) circle [radius=2pt];
		\filldraw (lg3) circle [radius=2pt];
		\draw (lg-1) circle [radius=2pt];
		\filldraw (lg-2) circle [radius=2pt];
		\filldraw (lg-3) circle [radius=2pt];
	%% codomain
		\node %[label=right:{$3$}] 
(rg3) at (0,3) {};
		\node %[label=right:{$2$}]
 (rg2) at (0,2) {};
		\node %[label=right:{$1$}]
 (rg1) at (0,1) {};
		\node %[label=right:{$-1$}] 
(rg-1) at (0,0) {};
		\node %[label=right:{$-2$}] 
(rg-2) at (0,-1) {};
		\node %[label=right:{$-3$}] 
(rg-3) at (0,-2) {};

   		\draw [shorten <=-2pt, shorten >=-2pt] (rg1) -- (rg2);
		\draw [shorten <=-2pt, shorten >=-2pt] (rg3) -- (rg2);
		\draw [shorten <=-2pt, shorten >=-2pt] (rg-1) -- (rg1);
		\draw [shorten <=-2pt, shorten >=-2pt] (rg-2) -- (rg-1);
		\draw [shorten <=-2pt, shorten >=-2pt] (rg-3) -- (rg-2);

		\filldraw (rg1) circle [radius=2pt];
		\filldraw (rg2) circle [radius=2pt];
		\draw (rg3) circle [radius=2pt];
		\filldraw (rg-1) circle [radius=2pt];
		\filldraw (rg-2) circle [radius=2pt];
		\draw (rg-3) circle [radius=2pt];
	%%arrows
		\draw[->] (lg2) -- (rg1);
		\draw[->] (lg-3) -- (rg-2);
		\draw[->] (lg3) -- (rg2);
		\draw[->] (lg-2) -- (rg-1);
\end{scope} 
		
	\end{tikzpicture}
\end{center}
\caption{A generating set for $\PEZ{6}$ \label{fig:pe-Z6}}
\end{figure}

%%%%%%%%%%%%%%%%%%%%%%%%%%%%%%%%%%%%%%%%%%%%%%%%%%%%%%%%

\begin{prop}
\label{lem:peclaim}  Let $n \geq 1$.  Then for each  $m\leq n$ and for any
 $2m$-element subalgebra 
 $\A$ of $\Zed_{2n}$. 
there  exists
$\varphi _{\A} \in F_n$ such that  $\varphi _{\A}{\restriction}_{\A}$ is an isomorphism
from~$\A$ onto~$\Zed_{2m}$.  
Moreover, $\varphi _{\A}^{-1} $ also belongs to~$F_n$.
\end{prop}

\begin{proof}  
If $n=1$  the only 
 subalgebra  $\A$   of $\Zed_{2}$ is $\Zed_{2}$ itself,  and then the  result is trivial.  We henceforth assume that $n \geq 2$.

The existence of 
$\varphi _{\A}\in\PEZ{2n}$ that maps $\A$
 isomorphically  onto $\Zed_{2m}$ is ensured by Corollary~\ref{cor:k-trans}. 
Hence the point at issue is that $\varphi_{\A} $ and its inverse 
belong to~$F_n$.
We shall prove this by induction on~$n$.
Assume that  the claim is valid when $n$ is replaced by  $n-1$. 
Our first task is one of reconciliation: 
 in order to apply our  inductive hypothesis we shall need to relate elements of $\PEZ{2n}$ 
to elements of $\PEZ{2(n-1)}$.
We  keep the notation $f_2, \ldots, f_{n},g$ for the partial endomorphisms 
of $\Zed_{2n}$ defined already and adopt the notation $h_2, \ldots , h_{n-1},p$
for the corresponding maps obtained when~$n$ is replaced by $n-1$.
Observe that for $1<i \leq n-1$ we have the following compatibilities:
$f_i{\restriction}_{\Zed_{2n-2}}=h_i$ and  $g{\restriction}_{\Zed_{2n-2}} =p$. 
Consequently, $\eta_n\colon F_{n-1} \hookrightarrow F_n$ defined by
\begin{align*}
\eta_n(h_i) &=f_i  \text{ if } 1 \leq i <  n-1,\\
\eta_n(p) &= g,\\
\eta_n(\id_{\Zed_{2n-2}}) &=g\newcirc  f_{2}\newcirc  \cdots\newcirc  f_{n-2}\newcirc  f_{n-1}\newcirc  f_n
\end{align*}
determines an embedding from $F_{n-1}$ into $F_n$.  
 
 Let~$\A$ be a $2m$-element
subalgebra of $\Zed_{2n}$, where $m \leq n$.
If $m=n$ then $\A = \Zed_{2n} = \Zed_{2m}$.  Then   
$\varphi _{\A} $ is the identity map on  $\Zed_{2n}$ and necessarily  belongs to~$F_n$. 
Now assume $m < n$.

\noindent \textit{Case 1}:  Assume 
$n \not\in \A$.   
Then $\A \subseteq\Zed_{2n-2}$ and, 
by the inductive hypothesis, we can find an isomorphism $\iota$ from 
$\A$ onto $\Zed_{2m}$ such that both $\iota$ and $\iota^{-1}$  belong to   $F_{n-1}$.
Then $\eta_{n}(\iota)\in F_n$ is an isomorphism from $\A$ onto $\Zed_{2m}$ and also $\eta_{n}(\iota^{-1})=(\eta_{n}(\iota))^{-1}\in F_n$.
  
\noindent\textit{Case 2}:  Assume $n \in \A$. 
Since $m<n$, there exists a largest $k > 0$ such that $k \notin \A$.  
Thus all of $k+1, \ldots,n$ belong to $\A$.  
The composite map $g{\newcirc } f_{2}{\newcirc } \cdots {\newcirc } f_{k}$ 
is an isomorphism from $\Zed_{2n}\setminus\{k,-k\}$ onto $\Zed_{2n-2}$; 
both this map and its inverse, \textit{viz.}  $f_{k}\newcirc  \cdots \newcirc  f_{n-1}$, 
belong to~$F_n$.   
This allows us to reduce the problem to that considered in Case~1.
\end{proof}

\begin{prop}[generation of  partial endomorphisms, even case] \label{prop:genpe}
For $n =2, 3, \ldots$  the monoid 
$\PEZ {2n}$ of partial endomorphisms of~$\Zed_{2n}$ is generated by the maps $f_2,\ldots,f_n,g$.
\end{prop}

\begin{proof}  
 Let $f\in \PEZ {2n}$. %
 By Proposition~\ref{prop:subalg} and Lemma~\ref{lem:peclaim},
$f$ is an isomorphism from its domain $\dom f$ onto its image $\img f$.
and  $m=|\dom f|=|\img f|$. By Proposition~\ref{lem:peclaim},
the  maps $\varphi _{\dom f}$ and $\varphi _{\img f}$ constructed there 
restrict to  isomorphisms onto $\Zed_{2m}$ from $\dom f$ and $\img f$ 
respectively, and they and their inverses lie in $F_n$.  
It follows immediately that $f=\varphi _{\img f}\newcirc  (\varphi _{\dom f})^{-1}\in F_n$. 
\end{proof}

We now consider partial endomorphisms in the odd case.   
This time we are able to take advantage of the existence of non-identity endomorphisms.  
We make the following definitions:
\begin{alignat*}{3}
 &\text{partial endomorphisms:} \quad 
&& f_0 
\colon \Zed_{2n+1} \setminus \{ 0\}  \to \Zed_{2n+1} , 
%\\    
\qquad &&    f_0 
(a) = a \mbox{ for  $a\ne 0 $;}
\\[1ex]
&&&  f_1  
\colon \Zed_{2n+1}\setminus \{1,-1\} \to \Zed_{2n+1},
%\\&
\qquad &&       f_1 
(a) = a \mbox{ for $a \ne \pm 1$};\\[1.5ex]
 &&& \text{and for } 1< i \leq  n,\\[-.5ex]
&&&f_i
\colon \Zed_{2n+1} \setminus\{i,-i\} \to \Zed_{2n+1} , 
%\\&
\qquad &&f_i(a) = 
\begin{cases}
i& \mbox{if } a = i-1,\\
-i& \mbox{if } a = -(i-1),\\
a & \mbox{otherwise};
\end{cases} 
\\[2ex]
&\text{endomorphism:}
&& g\colon \Zed_{2n+1} \to \Zed_{2n+1} , 
\qquad && g(a) = \begin{cases}
a-1 & \mbox{if } a>0, \\
a+1& \mbox{if } a<0,\\
0& \mbox{if } x=0.
\end{cases}
\end{alignat*}

%%%%%%%%%%%%%%%%%%%%%%%%%%%%%%%
%%%%%%%%%%%%%%%%%%%%%%%%%%%%%%%%%%%%%%%%%%%%%%%%%%%%%%%%5 

\begin{figure}
\begin{center}
	\begin{tikzpicture}[scale=.8]
\begin{scope}[xshift=-6cm]
%vertical axis 
\begin{scope}[xshift=-.6cm]
\node [label=left:{\small $3$\   -}]   (a3) at  (-8,4) {};
\node [label=left:{\small  $2$\  -}]   (a2) at  (-8,3) {};
\node [label=left:{\small  $1$\ -}]   (a1) at  (-8,2) {};
\node [label=left:{\small  $0$\   -}]   (a0) at  (-8,1) {};
\node [label=left:{\small  $-1$\  -}]   (a-1) at  (-8,0) {};
\node [label=left:{\small  $-2$\  -}]   (a-2) at  (-8,-1) {};
\node [label=left:{\small  $-3$\ -}]   (a-3) at  (-8,-2) {};

\draw[thin] (-8.35,4.2)--(-8.35,-2.2);
\end{scope}
\end{scope} 

%%%%%%%%%%%%%%%%%%%%%%%%%%%%%%%%%%%%%%%%%%%%	

	%%%%% f_0
\begin{scope}[xshift=-14.7cm]
	%%domain
		\node [label=left:{}] (lg3) at (2,4) {};
		\node [label=left:{}] (lg2) at (2,3) {};
		\node [label=left:{}] (lg1) at (2,2) {};
		\node [label=left:{}] (lg0) at (2,1) {};
		\node [label=left:{}] (lg-1) at (2,0) {};
		\node [label=left:{}] (lg-2) at (2,-1) {};
		\node [label=left:{}] (lg-3) at (2,-2) {};

   		\draw [shorten <=-2pt, shorten >=-2pt] (lg1) -- (lg2);
		\draw [shorten <=-2pt, shorten >=-2pt] (lg3) -- (lg2);
		\draw [shorten <=-2pt, shorten >=-2pt] (lg0) -- (lg1);
		\draw [shorten <=-2pt, shorten >=-2pt] (lg-1) -- (lg0);
		\draw [shorten <=-2pt, shorten >=-2pt] (lg-2) -- (lg-1);
		\draw [shorten <=-2pt, shorten >=-2pt] (lg-3) -- (lg-2);

		\draw (lg0) circle [radius=2pt];
		\filldraw (lg1) circle [radius=2pt];
		\filldraw (lg2) circle [radius=2pt];
		\filldraw (lg3) circle [radius=2pt];
		\filldraw (lg-1) circle [radius=2pt];
		\filldraw (lg-2) circle [radius=2pt];
		\filldraw (lg-3) circle [radius=2pt];
	%% codomain
		\node [label=right:{}] (rg3) at (3,4) {};
		\node [label=right:{}] (rg2) at (3,3) {};
		\node [label=right:{}] (rg1) at (3,2) {};
		\node [label=right:{}] (rg0) at (3,1) {};
		\node [label=right:{}] (rg-1) at (3,0) {};
		\node [label=right:{}] (rg-2) at (3,-1) {};
		\node [label=right:{}] (rg-3) at (3,-2) {};

   		\draw [shorten <=-2pt, shorten >=-2pt] (rg1) -- (rg2);
		\draw [shorten <=-2pt, shorten >=-2pt] (rg3) -- (rg2);
		\draw [shorten <=-2pt, shorten >=-2pt] (rg-1) -- (rg0);
		\draw [shorten <=-2pt, shorten >=-2pt] (rg0) -- (rg1);
		\draw [shorten <=-2pt, shorten >=-2pt] (rg-2) -- (rg-1);
		\draw [shorten <=-2pt, shorten >=-2pt] (rg-3) -- (rg-2);

		\draw (rg0) circle [radius=2pt];
		\filldraw (rg1) circle [radius=2pt];
		\filldraw (rg2) circle [radius=2pt];
		\filldraw (rg3) circle [radius=2pt];
		\filldraw (rg-1) circle [radius=2pt];
		\filldraw (rg-2) circle [radius=2pt];
		\filldraw (rg-3) circle [radius=2pt];
	%%arrows
		\draw[->] (lg1) -- (rg1);
		\draw[->] (lg-1) -- (rg-1);
%		\draw[->] (lg0) -- (rg0);
		\draw[->] (lg3) -- (rg3);
		\draw[->] (lg-3) -- (rg-3);
		\draw[->] (lg2) -- (rg2);
		\draw[->] (lg-2) -- (rg-2);
\node (lab-f0)  at  (2.3,-2.5) {$f_0$};
\end{scope}

%%%%%%%%%%%%%%%%%%%%%%%%%%%%%%%%%%%%%%%%%		

			%%%%%  f_1
\begin{scope}[xshift=-14.56cm] 
	%%domain
		\node [label=left:{}] (lc13) at (5,4) {};
		\node [label=left:{}] (lc12) at (5,3) {};
		\node [label=left:{}] (lc11) at (5,2) {};
		\node [label=left:{}] (lc10) at (5,1) {};
		\node [label=left:{}] (lc1-1) at (5,0) {};
		\node [label=left:{}] (lc1-2) at (5,-1) {};
		\node [label=left:{}] (lc1-3) at (5,-2) {};

   		\draw [shorten <=-2pt, shorten >=-2pt] (lc11) -- (lc12);
		\draw [shorten <=-2pt, shorten >=-2pt] (lc13) -- (lc12);
		\draw [shorten <=-2pt, shorten >=-2pt] (lc10) -- (lc11);
		\draw [shorten <=-2pt, shorten >=-2pt] (lc1-1) -- (lc10);
		\draw [shorten <=-2pt, shorten >=-2pt] (lc1-2) -- (lc1-1);
		\draw [shorten <=-2pt, shorten >=-2pt] (lc1-3) -- (lc1-2);

		\filldraw (lc10) circle [radius=2pt];
		\draw (lc11) circle [radius=2pt];
		\filldraw (lc12) circle [radius=2pt];
		\filldraw (lc13) circle [radius=2pt];
		\draw (lc1-1) circle [radius=2pt];
		\filldraw (lc1-2) circle [radius=2pt];
		\filldraw (lc1-3) circle [radius=2pt];
	%% codomain
		\node [label=right:{}] (rc13) at (6,4) {};
		\node [label=right:{}] (rc12) at (6,3) {};
		\node [label=right:{}] (rc11) at (6,2) {};
		\node [label=right:{}] (rc10) at (6,1) {};
		\node [label=right:{}] (rc1-1) at (6,0) {};
		\node [label=right:{}] (rc1-2) at (6,-1) {};
		\node [label=right:{}] (rc1-3) at (6,-2) {};

   		\draw [shorten <=-2pt, shorten >=-2pt] (rc11) -- (rc12);
		\draw [shorten <=-2pt, shorten >=-2pt] (rc13) -- (rc12);
		\draw [shorten <=-2pt, shorten >=-2pt] (rc1-1) -- (rc10);
		\draw [shorten <=-2pt, shorten >=-2pt] (rc10) -- (rc11);
		\draw [shorten <=-2pt, shorten >=-2pt] (rc1-2) -- (rc1-1);
		\draw [shorten <=-2pt, shorten >=-2pt] (rc1-3) -- (rc1-2);

		\filldraw (rc10) circle [radius=2pt];
		\draw (rc11) circle [radius=2pt];
		\filldraw (rc12) circle [radius=2pt];

		\filldraw (rc13) circle [radius=2pt];
		\draw (rc1-1) circle [radius=2pt];
		\filldraw (rc1-2) circle [radius=2pt];
		\filldraw (rc1-3) circle [radius=2pt];
	%%arrows
%		\draw[->] (lc11) -- (rg1);
%		\draw[->] (lg-1) -- (rg-1);
		\draw[->] (lc10) -- (rc10);
		\draw[->] (lc13) -- (rc13);
		\draw[->] (lc1-3) -- (rc1-3);
		\draw[->] (lc12) -- (rc12);
		\draw[->] (lc1-2) -- (rc1-2);
\node (lab-f1)  at  (5.3,-2.5) {$f_1$};
\end{scope}

%%%%%%%%%%%%%%%%%%%%%%%%%%%%%%%%%%%%%%%%%%%%%
	%%%%% f2
\begin{scope}[xshift=-5.5cm]
	%%domain
		\node [label=left:{}] (lf3) at (-1,4) {};
		\node [label=left:{}] (lf2) at (-1,3) {};
		\node [label=left:{}] (lf1) at (-1,2) {};
		\node [label=left:{}] (lf0) at (-1,1) {};
		\node [label=left:{}] (lf-1) at (-1,0) {};
		\node [label=left:{}] (lf-2) at (-1,-1) {};
		\node [label=left:{}] (lf-3) at (-1,-2) {};

   		\draw [shorten <=-2pt, shorten >=-2pt] (lf1) -- (lf2);
		\draw [shorten <=-2pt, shorten >=-2pt] (lf3) -- (lf2);
		\draw [shorten <=-2pt, shorten >=-2pt] (lf1) -- (lf0);
		\draw [shorten <=-2pt, shorten >=-2pt] (lf-1) -- (lf0);
		\draw [shorten <=-2pt, shorten >=-2pt] (lf-2) -- (lf-1);
		\draw [shorten <=-2pt, shorten >=-2pt] (lf-3) -- (lf-2);

		\filldraw (lf1) circle [radius=2pt];
		\draw (lf2) circle [radius=2pt];
		\filldraw (lf3) circle [radius=2pt];
	\filldraw (lf0) circle [radius=2pt];
			\filldraw (lf-1) circle [radius=2pt];
		\draw (lf-2) circle [radius=2pt];
		\filldraw (lf-3) circle [radius=2pt];
	%% codomain
		\node [label=right:{}] (rf3) at (0,4) {};
		\node [label=right:{}] (rf2) at (0,3) {};
		\node [label=right:{}] (rf1) at (0,2) {};
		\node [label=right:{}] (rf0) at (0,1) {};
		\node [label=right:{}] (rf-1) at (0,0) {};
		\node [label=right:{}] (rf-2) at (0,-1) {};
		\node [label=right:{}] (rf-3) at (0,-2) {};

   		\draw [shorten <=-2pt, shorten >=-2pt] (rf1) -- (rf2);
		\draw [shorten <=-2pt, shorten >=-2pt] (rf3) -- (rf2);
			\draw [shorten <=-2pt, shorten >=-2pt] (rf0) -- (rf1);
\draw [shorten <=-2pt, shorten >=-2pt] (rf-1) -- (rf0);

		\draw [shorten <=-2pt, shorten >=-2pt] (rf-2) -- (rf-1);
		\draw [shorten <=-2pt, shorten >=-2pt] (rf-3) -- (rf-2);

		\draw (rf1) circle [radius=2pt];
		\filldraw (rf2) circle [radius=2pt];
		\filldraw (rf3) circle [radius=2pt];
		\draw (rf-1) circle [radius=2pt];
		\filldraw (rf-2) circle [radius=2pt];
		\filldraw (rf-3) circle [radius=2pt];
             	\filldraw (rf0) circle [radius=2pt];
	%%arrows
		\draw[->] (lf1) -- (rf2);
		\draw[->] (lf0) -- (rf0);
		\draw[->] (lf-1) -- (rf-2);
		\draw[->] (lf3) -- (rf3);
		\draw[->] (lf-3) -- (rf-3);
\node (lab-f2)  at  (-.5,-2.5) {$f_2$};
\end{scope}

	%%%%% f3  
\begin{scope}[xshift=-2.5cm]
	%%domain
		\node [label=left:{}] (lf3) at (-1,4) {};
		\node [label=left:{}] (lf2) at (-1,3) {};
		\node [label=left:{}] (lf1) at (-1,2) {};
		\node [label=left:{}] (lf0) at (-1,1) {};
		\node [label=left:{}] (lf-1) at (-1,0) {};
		\node [label=left:{}] (lf-2) at (-1,-1) {};
		\node [label=left:{}] (lf-3) at (-1,-2) {};

   		\draw [shorten <=-2pt, shorten >=-2pt] (lf1) -- (lf2);
		\draw [shorten <=-2pt, shorten >=-2pt] (lf3) -- (lf2);
		\draw [shorten <=-2pt, shorten >=-2pt] (lf1) -- (lf0);
		\draw [shorten <=-2pt, shorten >=-2pt] (lf-1) -- (lf0);
		\draw [shorten <=-2pt, shorten >=-2pt] (lf-2) -- (lf-1);
		\draw [shorten <=-2pt, shorten >=-2pt] (lf-3) -- (lf-2);

		\filldraw (lf1) circle [radius=2pt];
		\filldraw (lf2) circle [radius=2pt];
		\draw (lf3) circle [radius=2pt];
	\filldraw (lf0) circle [radius=2pt];
			\filldraw (lf-1) circle [radius=2pt];
		\filldraw (lf-2) circle [radius=2pt];
		\draw (lf-3) circle [radius=2pt];
	%% codomain
		\node [label=right:{}] (rf3) at (0,4) {};
		\node [label=right:{}] (rf2) at (0,3) {};
		\node [label=right:{}] (rf1) at (0,2) {};
		\node [label=right:{}] (rf0) at (0,1) {};
		\node [label=right:{}] (rf-1) at (0,0) {};
		\node [label=right:{}] (rf-2) at (0,-1) {};
		\node [label=right:{}] (rf-3) at (0,-2) {};

   		\draw [shorten <=-2pt, shorten >=-2pt] (rf1) -- (rf2);
		\draw [shorten <=-2pt, shorten >=-2pt] (rf3) -- (rf2);
\draw [shorten <=-2pt, shorten >=-2pt] (rf0) -- (rf1);		
\draw [shorten <=-2pt, shorten >=-2pt] (rf0) -- (rf-1);

		\draw [shorten <=-2pt, shorten >=-2pt] (rf-2) -- (rf-1);
		\draw [shorten <=-2pt, shorten >=-2pt] (rf-3) -- (rf-2);

		\filldraw (rf1) circle [radius=2pt];
		\draw (rf2) circle [radius=2pt];
		\filldraw (rf3) circle [radius=2pt];
             \filldraw (rf0)  circle [radius =2pt];
		\filldraw (rf-1) circle [radius=2pt];
		\draw (rf-2) circle [radius=2pt];
		\filldraw (rf-3) circle [radius=2pt];
	%%arrows
		\draw[->] (lf1) -- (rf1);
		\draw[->] (lf0) -- (rf0);
		\draw[->] (lf-1) -- (rf-1);
		\draw[->] (lf2) -- (rf3);
		\draw[->] (lf-2) -- (rf-3);
\node (lab-f3)  at  (-.5,-2.5) {$f_3$};
\end{scope}

	%%%%% g
\begin{scope}[xshift=-8.5cm]
	%%domain
		\node [label=left:{}] (lt3) at (8,4) {};
		\node [label=left:{}] (lt2) at (8,3) {};
		\node [label=left:{}] (lt1) at (8,2) {};
		\node [label=left:{}] (lt0) at (8,1) {};
		\node [label=left:{}] (lt-1) at (8,0) {};
		\node [label=left:{}] (lt-2) at (8,-1) {};
		\node [label=left:{}] (lt-3) at (8,-2) {};

   		\draw [shorten <=-2pt, shorten >=-2pt] (lt1) -- (lt2);
		\draw [shorten <=-2pt, shorten >=-2pt] (lt3) -- (lt2);
		\draw [shorten <=-2pt, shorten >=-2pt] (lt0) -- (lt1);
		\draw [shorten <=-2pt, shorten >=-2pt] (lt-1) -- (lt0);
		\draw [shorten <=-2pt, shorten >=-2pt] (lt-2) -- (lt-1);
		\draw [shorten <=-2pt, shorten >=-2pt] (lt-3) -- (lt-2);

		\filldraw (lt0) circle [radius=2pt];
		\filldraw (lt1) circle [radius=2pt];
		\filldraw (lt2) circle [radius=2pt];
		\filldraw (lt3) circle [radius=2pt];
		\filldraw (lt-1) circle [radius=2pt];
		\filldraw (lt-2) circle [radius=2pt];
		\filldraw (lt-3) circle [radius=2pt];
	%% codomain
		\node [label=right:{}] (rt3) at (9,4) {};
		\node [label=right:{}] (rt2) at (9,3) {};
		\node [label=right:{}] (rt1) at (9,2) {};
		\node [label=right:{}] (rt0) at (9,1) {};
		\node [label=right:{}] (rt-1) at (9,0) {};
		\node [label=right:{}] (rt-2) at (9,-1) {};
		\node [label=right:{}] (rt-3) at (9,-2) {};

   		\draw [shorten <=-2pt, shorten >=-2pt] (rt1) -- (rt2);
		\draw [shorten <=-2pt, shorten >=-2pt] (rt3) -- (rt2);
		\draw [shorten <=-2pt, shorten >=-2pt] (rt-1) -- (rt0);
		\draw [shorten <=-2pt, shorten >=-2pt] (rt0) -- (rt1);
		\draw [shorten <=-2pt, shorten >=-2pt] (rt-2) -- (rt-1);
		\draw [shorten <=-2pt, shorten >=-2pt] (rt-3) -- (rt-2);

		\filldraw (rt0) circle [radius=2pt];
		\filldraw (rt1) circle [radius=2pt];
		\filldraw (rt2) circle [radius=2pt];
		\draw (rt3) circle [radius=2pt];
		\filldraw (rt-1) circle [radius=2pt];
		\filldraw (rt-2) circle [radius=2pt];
		\draw (rt-3) circle [radius=2pt];
	%%arrows
		\draw[->] (lt1) -- (rt0);
		\draw[->] (lt-1) -- (rt0);
		\draw[->] (lt0) -- (rt0);
		\draw[->] (lt3) -- (rt2);
		\draw[->] (lt-3) -- (rt-2);
		\draw[->] (lt2) -- (rt1);
		\draw[->] (lt-2) -- (rt-1);
\node (lab-g)  at  (8.4,-2.5) {$g$};
\end{scope}
		\end{tikzpicture}
\end{center}
\caption{A set of generators for  $\PEZ{7}$ \label{fig:PEodd}}
\end{figure}

%%%%%%%%%%%%%%%%%%%%%%%%%%%%%%%%%%%%%%%%%%%%%%%%%%%%%%%
%%%%%%%%%%%%%%%%%%%%%%%%%%%%%%%%%%%%%%%%%%%%%%%%%%%%%%
%

\begin{prop}[generation of  partial endomorphisms, odd case] \label{prop:genpeOdd}
For $n =1, 2, \ldots$, the partial endomorphism monoid $\PEZ {2n+1}$ 
is  generated by the maps $f_0,  \ldots, f_n, %and
 g$.
\end{prop}
\begin{proof}
If $n=1$, it is easy to see that $\PEZ {3} = \{\id_{\Zed_{3}},f_0,f_1,g, g\newcirc  f_0\}$.

Assume now that $n>1$.
We provide only a sketch of the proof.
We denote by $F_{2n+1}$ the subset of $\PEZ {2n+1}$ 
generated by $f_0,f_1,f_2, \ldots, f_n,g$.
Given $p\in \PEZ {2n+1}$ the first step is to find a $p_0$ in $F_{2n+1}$ 
that has the same domain as  $p$. 
This is easily achieved using only $f_0,f_1,f_2,\ldots, f_n$. 
Secondly one applies $g$ sufficiently many times to $p_0$  
to obtain $p_1$ such that $p_1(a)=0$ if and only if $p(a)=0$. 
Observe that now the sizes of the images of $p$ and $p_1$ are the same.
Hence these images are isomorphic as subalgebras of $\Zed_{2n+1}$.
Finally, an adaptation of Lemma~\ref{lem:peclaim} 
(now using $g, f_2,\ldots,f_n$ as defined 
for $\Zed_{2n+1}$)  
will provide a partial endomorphism in $F_{2n+1}$ that is defined at~$0$ 
and that maps the image of $p_1$ to the image of $p$.
\end{proof}
%%%%%%%%%%%%%%%%%%%%%%%%%%%%%%%%%%%%%%%%%%%%%%%%%%%%%%%%%

\section{Natural dualities and the Test Spaces Method}
\label{Sec:NatDual}  
%%%%%%%%%%%%%%%%%%%%%%%%%%%%%%%%%%%%%%%%%%%%%

We first 
lay the groundwork for the presentation of our Test Spaces Method by recalling very briefly the theory of natural dualities as we shall use it. 
A textbook treatment can be found  in~\cite{CD98}.  
Alternative sources 
which jointly cover the material on which  we shall draw in black-box fashion are: 
~\cite{tendoll} (an introductory survey of natural duality theory in general, aimed at  novices)  
and \cite{GLT,HPbirk} 
(survey articles  focusing on the theory as it applies to  classes of algebras with distributive lattice reducts).  In addition 
\cite{CDquest} provides  a detailed contextualised  account of strong dualities.

We shall tailor our exposition to our intended applications.  In particular we shall
restrict arities of operations and relations to those we shall require.  
   Let $\CA$ be the quasivariety generated by a finite algebra $\M$, that is, $\CA = \ISP(\M)$, regarded as a category by taking the morphisms to be all homomorphisms.
Our aim is to find a second category~$\CX $ 
whose objects are topological structures of common type  and 
which is dually equivalent to~$\CA$ via 
functors $\D \colon \CA \to \CX $ and $\E \colon \CX  \to \CA$.  

\begin{df}[alter ego]  \label{def:alterego}  
{\rm 
We consider a topological structure 
$%\[
%lc added:
\MT = (M; G,H,K,R,\Tp),
$ % \]
 where 
\begin{newlist}
\item[$\bullet$] $\Tp$ is the discrete topology on~$M$;
\item[$\bullet$] $G$ is a set of endomorphisms of~$\M$;
\item[$\bullet$] $H$ is a set of (non-total) partial endomorphisms on $\M$;
%lc added:
\item[$\bullet$] $K$ is a set of one element subalgebras of $\M$ (considered as constants);
\item[$\bullet$]  $R$ is a set of  binary relations 
on~$M$ such that each 
$r\in R$ is the universe of a 
subalgebra~$\mathbf r$ of~$\M^2$. 
\end{newlist} 
We refer to such a topological structure $\MT$ as an \defn{alter ego} for~$\M$ 
and say that $\MT$ and $\M$ are \defn{compatible}.   
}
\end{df}

Using an alter ego $\MT$ we build the desired category~$\CX$ of 
structured topological spaces. 
  We define the topological quasivariety generated by~$\MT$ to be
$\CX  := \IScP(\MT) $,  
the class of isomorphic copies of closed substructures of 
non-empty powers of $\MT$,
with $^+$ indicating that the empty structure is included.  Here a non-empty
power $\MT^S$ of  $\MT$ carries the product topology and is equipped with the pointwise
liftings of the members of 
%lc K added:
$G\cup H \cup K\cup R$.  Closed substructures and isomorphic 
copies are defined in the expected  way;   Thus  a member  
$\X$ of $\CX$ is a structure
%lc K^{\Y} added
 $(X; G^\X, H^\X, K^{\X}, R^\X,\Tp^\X)$ of the same type as
$\MT$.   Details are given in
\cite[Section 1.4]{CD98}.  
 We make~$\CX $ into  a category by taking all continuous 
structure-preserving maps as the morphisms.

The best-known natural dualities---and in particular those employed hitherto in the study of admissible rules---have alter egos which contain no partial endomorphisms which are not total, that is, $H=\emptyset$.  Except in 
very special cases, partial operations play a crucial role in the dualities we construct for Sugihara algebras;  the exceptions are the quasivarieties $
\SA_k$ for $k \leq 3$. 
 A full discussion of the technical niceties that arise when an alter ego
contains 
partial operations  is given  in \cite[Chapter~2]{CD98}.
We draw attention here to  the  constraints on the domain and range of 
a morphism
$\varphi  \colon \Y \to \Z$ in $\CX$ 
 when partial operations are present, that is, $H \ne \emptyset$.  
Given  any $h \in H$ and $y \in \dom h^\Y$,  we must  restrict  $\varphi(y)$ to lie in  
$ \dom h^\Z$; the  preservation condition becomes $h^\Z(\varphi (y)) = \varphi (h^\Y(y))$, for all $y \in \dom h^Y$.    
The morphism $\varphi$ is said to be 
an \defn{embedding}
if  
$\varphi(\Y)$ is a substructure of~$\Z$ and $\varphi\colon \Y \to \varphi(\Y)$ 
is an isomorphism.   This implies  in particular that
$y \in \dom h^\Y$ whenever $\varphi(y) \in h^\Z$. %, for any $h \in H$. 
 We henceforth use the same symbol for an operation $g \in G$ and for its pointwise lifting to a power of~$\MT$ and more generally for its interpretation 
on any $\X \in \IScP(\MT)$.  We do likewise for members of~$H$ and~$R$.
 This will cause no confusion in practice, since the  meanings
attributed to the various 
 symbols 
  will  be dictated by the context.   

Assume, as above,  that  $\MT$ is an alter ego for $\M$ and that $\CA= \ISP(\M)$ and $\CX = \IScP(\MT)$. 
Given this scenario we can set up a dual adjunction between $\CA$ and~$\CX$,
based on hom-functors into $\M$ and into~$\MT$.  
Let $\A \in \CA$ and $\X \in \CX$. 
Then  $\CA(\A,\M)$  is the universe of a closed substructure of $\MT^{A} $ and 
$\CX (\X,\MT)$  is the universe of a subalgebra of~$\M^{X}$.
As a consequence of compatibility, there exist well-defined
contravariant  hom-functors $\D \colon \CA \to \CX $
and $\E \colon \CX \to \CA$:
\begin{alignat*}{3}
&\text{on objects:} & \hspace*{3.5cm}  &\D \colon  \A \mapsto  \CA(\A,\M), 
  \hspace*{3.5cm}  \phantom{\text{on objects:}}&&\\
&\text{on morphisms:}  & &  \D \colon  x \mapsto - \newcirc  x&& \\
\shortintertext{and }
&\text{on objects:} & & \E  \colon  \X \mapsto  \CX (\X,\MT),
\phantom{\text{on objects:}}&&\\
&\text{on morphisms:}  & &\E  \colon  \varphi  \mapsto - \newcirc  \varphi . %,
\phantom{\text{on morphisms:}}&&
\end{alignat*}
Given  $\A \in \CA$ we refer to $\D(\A)$ as the 
\defn{{\rm (}natural{\,\rm)} dual 
space of} $\A$. 

Let  $\A\in \CA$ and $\X \in \CX$.    There exist natural  evaluation maps
 $\esubA 
\colon \A \to \E\D(\A)$ and $\epsub{\X}\colon \X \to \D\E(\X)$,
with $ \esubA (a)\colon f \mapsto f(a)$ and $\epsub{\X}(x) \colon g \mapsto g(x)$.
Moreover,  $(\D,\E,e,\epsilon)$ is a dual adjunction (see~\cite;Chapter~2]{CD98}). 
Each of the maps 
$\esubA$ and $\epsub{\X} $ is an embedding.  
These claims hinge on the compatibility of~$\MT$ and~$\M$.  
We say that~$\MT$ \defn{yields a duality on}~$\CA$
if each $\esubA$ 
is also surjective.   
If in addition each $\epsub{\X} $ 
is  surjective and so an isomorphism, 
we say that the duality yielded by~$\MT$ is \defn{full}.  
In this case $\CA$ and~$\CX$ are dually equivalent.

With the basic categorical framework now in place, we  interpose an example  to illustrate the calculation of dual spaces
when partial operations are present.

\begin{ex}\label{Ex:Sug46}   We work with the quasivariety  $\SA_6= \ISP(\Zed_{6})$   
and take  our alter ego  $\twiddle{\Zed_{6}}$  to include the generating set 
$\{f_2,f_3,g\} $ for $\PEZ{6}$ 
(as shown in Figure~\ref{fig:pe-Z6})   and the relations  $\eq_1$ and $\eq_2$ (as in Proposition~\ref{Prop:Congruences}).   
 (This will turn out to be a dualising alter ego,
 but this fact is not needed  here.)    Denote by  $\D_{6}\colon \SA_{6} \to \IScP(\twiddle{\Zed_{6}})$  the functor determined by $\twiddle{\Zed_{6}}$.

Note that $\Zed_4\in \SA_{4}\subset \SA_{6}$.   We shall calculate 
$\D_{6}(\Zed_{4})$ as a member of $\CX = \IScP(\twiddle{\Zed_{6}})$.  Its universe 
 consists of all the homomorphisms  from $\Zed_4$ into $\Zed_6$.
By \ref{prop:phomEven} every 
 such map is injective.   
From this we easily see that  $\SA_{6}(\Zed_{4},\Zed_6)$ consists of three maps $e_1,e_2$, and $e_3$.  These are given by
 $e_1=(-3,-2,2,3)$, $e_2= (-3,-1,  1,3)$ and $e_3= (-2,-1,1,2)$.
Here we  
use   the tuple $(\varphi (-2),\varphi (-1),\varphi (1),\varphi(2))$ to depict  $\varphi \in \SA_{6}(\Zed_{4},\Zed_6) = \D_{6}(\Zed_{4})$.

Consider 
the action  on  $\D_6(\Zed_4)$ of  $f_2, f_3$ and~$g$. 
   For each 
 $\varphi \in \{ e_1,e_2,e_3\}$ and any $y \in \dom {\varphi} $,  we require  $\varphi(y) \in \dom{f_2 } \cap \dom{f_3} \cap \dom{g}$.
We deduce that 
$\dom{f_2}=\{e_1\}$,  $\dom{f_3}=\{e_2\}$  and  $\dom{g}=\{e_3\}$, and $f_2(e_1)=e_2$, $f_3(e_2)=e_3$ and $g(e_3)=e_1$.

Similarly, the binary relations are lifted pointwise. The equivalence classes of $\eq_1$ consist only of singletons and 
those  of $\eq_2$ are $\{e_1,e_2\}$ and $\{e_3\}$.
\begin{figure}[h]
\begin{center}
	\begin{tikzpicture}
			%%domain f1
		\node [label=below:{$e_1$}] (e1) at (-2,0) {};
		\node [label=below:{$e_2$}] (e2) at (0,0) {};
		\node [label=below:{$e_3$}] (e3) at (2,0) {};

		\draw (e1) circle [radius=2pt];
		\draw (e2) circle [radius=2pt];
		\draw (e3) circle [radius=2pt];
%	
%	%%arrows
		\draw[->] (e1) to  (e2);
		\node[below,font=\footnotesize] at (-1,0) {$f_2$};
		\draw[->] (e2) -- (e3);
		\node[below,font=\footnotesize] at (1,0) {$f_3$};
		\draw[->] (e3) to [out = 135, in = 45](e1);
		\node[above,font=\footnotesize] at (0,1) {$g$};
%
	
	%%% Eq_1
			\draw (-2.25,.25) -- (-1.75,.25) -- (-1.75,-.6) -- (-2.25, -.6) -- cycle;
			\draw (-.25,.25) -- (.25,.25) -- (.25,-.6) -- (-.25, -.6) -- cycle;
			\draw (2.25,.25) -- (1.75,.25) -- (1.75,-.6) -- (2.25, -.6) -- cycle;
	%%% Eq_2
			\draw[dashed] (2.35,.35) -- (1.65,.35) -- (1.65,-.7) -- (2.35, -.7) -- cycle;
			\draw[dashed] (-2.35,.35) -- (.35,.35) -- (.35,-.7) -- (-2.35, -.7) -- cycle;

	\end{tikzpicture}
\end{center}
\caption{$\D_{6}(\Zed_{4})$.}\label{Sug64}
\end{figure}
\end{ex}

We now return to general theory.
  Our objective in setting up a natural duality for a given quasivariety~$\CA$ is 
thereby to transfer algebraic problems about~$\CA$ into problems 
about the dual category~$\CX$ using the hom-functors~$\D$ and~$\E$ 
to toggle backwards and forwards.

Our needs in this paper are very specific,  tailored as they are to the Test Spaces Method (which we shall recall shortly).
First of all, we  need a dual description of finitely generated free algebras.
Here we can call on a fundamental fact.
A dualising alter ego~$\MT$ plays a special role in the duality it sets up:
it is the dual space of  $\F_{\CA}(1)$.
More generally, the free algebra $\F_{\CA}(\kappa )$ generated by a non-empty set~$S$ of cardinality~$\kappa$ 
has dual space~$\MT^S$  \cite[Section~2.2, Lemma~2.1]{CD98}.  
Assuming the duality is full the free algebra $\F_{\CA}(\kappa)$ is  concretely realised as $\E(\MT^\kappa)$, with the coordinate projections as the free generators.
%%%%

We also require  a  duality for $\CA$ for which there 
are dual characterisations  of homomorphisms 
which are injective and of those which are surjective.  
This is not a categorical triviality since, for  morphisms in~$\CX$,  epi (mono) 
may not equate to surjective (injective).   
We need the notion of a duality which is \defn{strong}, 
which may most concisely be  defined as one in which $\twiddle{\M}$ 
is injective in the topological quasivariety $\CX$ that it generates.
The technical details, and equivalent definitions, need not concern us here 
(they can be found in \cite[Chapter~3]{CD98}
or \cite{CDquest}).  
We shall exploit without  proof two key facts.  
Firstly,
any strong duality is full. 
Secondly, in a strong  duality,  each of~$\D$ and~$\E$ has the property that 
it converts embeddings to surjections and surjections to embeddings, and vice versa;  see \cite[Chapter~3, Lemma 2.4]{CD98} 
or \cite[Section~3]{CDquest}.

The Test Spaces Method (TSM) 
is an algorithm that takes as input a finite algebra $\M$ and returns a  set $\mathcal{K} $ of algebras in $\ISP(\M)$ such that $\ISP(\M)=\ISP(\mathcal{K})$ and 
$\mathcal{K}$ is a minimal set of finite algebras each of minimum size.  
So, according to the definitions in Section~\ref{sec:intro},  $\mathcal{K}$ is the admissibility set and,  if $\mathcal{K} = \{\B\}$, as occurs in our applications in this paper,  then~$\B$ is the admissibility 
algebra. 
The method relies  on the availability of a strong duality for $\CA:= \ISP(\M)$.

To  describe TSM,  as it is given in \cite[Section~4]{CFMP}
and as we shall apply 
it, we first
recall some definitions. Assume that
$\D$ and $\E$  
set up a strong duality between $\CA:= \ISP(\M)$ and $\CX$
and that~$\M $ is a finite $s$-generated algebra in~$\CA$. 
We refer to a triple $(\X,\gamma,\eta)$ as a \emph{Test Space configuration}, or \emph{{\rm TS}-configuration} for short, if 
 \[ 
 \D(\M) \xhookrightarrow{\ \eta \ }\alg{X} \xdoubleheadleftarrow{\ \gamma\ } \D(\F_{\CA}(s)) = \twiddle{\M}^s,
\]
where the morphism $\eta\colon\D(\M)\hookrightarrow \X$ is an embedding 
and the morphism
$\gamma\colon \twiddle{\M}^s \twoheadrightarrow \X$ is surjective.
Since we are assuming the duality is strong, fullness then ensures that any
TS-configuration gives rise to an algebra $\A :=\E(\X)$ in $\CA$
which is a both a subalgebra of $\F_{\CA}(s)$ and such that $\M$ is a homomorphic image of~$\A$.   (See \cite[Section~2]{CFMP}
for a contextual
discussion.)  

Let $\X$ be  a finite structure in $\IScP (\twiddle{\M})$. 
We denote by $\mathcal{S}_\X$ the set of substructures  $\Z$  
of~$\X$ 
with the property that 
 $\Z$ is generated (as a substructure of $\X$) by $\img \varphi _1\cup \cdots\cup\img\varphi _m$, where $\varphi _1, \ldots, \varphi _m \colon \X \to \X$ are morphisms. 
It was shown in 
\cite[Proposition~5.3]{CFMP} 
that $\mathcal{S}_\X$ is a lattice for the inclusion order.

\begin{table}[ht]
\begin{center}
\begin{tabular}{lp{15.3cm}}
\hline
\multicolumn{2}{c}{Test Spaces Method}\bigstrut\\
\hline 
\\[-2.5ex] 
0.&{\tt Find $\twiddle{\M}$ that yields a strong duality for $\CA=\ISP(\M)$.}\\
1. & {\tt Compute $\D(\M)$.}\\
2. &{\tt Find a TS-configuration $(\X,\gamma,\eta)$ with~$X$ of minimum size.}  \\
3. &{\tt Determine the set $\mathcal{M}$ of maximal join-irreducible elements of $\mathcal{S}_{\X}$}.\\
4. & {\tt Construct a set $\mathcal{V}$ by repeatedly 
removing from $\mathcal{M}$ any structure   
that is a morphic image of some other structure in the set.}\\
5. & {\tt Compute  $\K=\{\,E(\X)\mid\X\in \mathcal{V}\,\}$.}
\\[-2.5ex]
\end{tabular}
\end{center}
\caption{Test Spaces Method}\label{table:TSM}
\end{table}

Suppose we have a candidate TS-configuration $(\X,\gamma,\eta)$ and that 
$\X$  is join-irreducible in the lattice $\mathcal{S}_\X$.  Then 
$\X$ itself is the unique maximal join-irreducible.  
Observe that this happens if, for some specified element of~$\X$, 
any morphism $\varphi  \colon \X \to \X$ whose image contains that element
is such that  the substructure generated by $\img \varphi $ is~$\X$.  
In this situation, Step~3 is accomplished without the need to check in Step~2 that $\X$ is of minimum size.
Moreover, Step~4 is bypassed, and we proceed straight to Step~5.  
Under this scenario, the outcome of the Test Spaces Method is a TS-configuration
$(\X, \gamma,\eta)$   such that
 $\E(\X)$
is the admissibility algebra we seek, with the property that $\X$ is both a substructure and a quotient in $\CX$ of $\MT^s$.   Dually (because the duality
is strong),  $\E(\X)$ is both a quotient and a subalgebra of $\F_{\CA}(s)$.  
See Corollary~\ref{cor:gens} below.

%%%%%%%%%%%%%%%%%%%%%%%%%%%%%%%%%%%%%%%%%%%%%%

Step 0 demands a strong duality for the quasivariety under consideration. 
 The class of  Sugihara algebras 
 is  lattice-based.  
In this setting  the \textit{existence} of some strongly dualising alter ego  for any finitely generated subquasivariety is not at issue:   general theory ensures this
(see the discussion in \cite{CD98}
that we be able to  \emph{exhibit} an alter ego for each of the quasivarieties
$\SA_{2n}$ and $\SA_{2n+1}$. 
We also wish to do this in an efficient manner, discarding from the set 
$G \cup H \cup R$ any elements which  are not needed for the duality to work.
This will make the duality easier to understand and so to apply.
The result we call on here is the  specialisation to distributive-lattice-based algebras of the Piggyback Duality Theorem,  as given in 
\cite[Chapter 7, Theorem~2.1]{CD98},
 combined,
for the strongness assertion, with the NU Strong Duality Theorem Corollary \cite[Chapter~3, Corollary 3.9]{CD98}.
In the form in which we apply piggybacking, the statement in  
Theorem~\ref{Thm:Piggyback}  will precisely meet our needs.

Let $\cat D$ denote  the category of distributive lattices with lattice homomorphisms as the morphisms and let $\two \in \cat D$
have universe $\{0,1\}$, where $0 < 1$.

\begin{thm}[Piggyback Strong Duality Theorem] \label{Thm:Piggyback}
Let $\M$ be a finite algebra having a reduct
 $\fnt{U}(\M)$ in $\cat D$.  Assume further that~$\M$ is such that every non-trivial subalgebra of~$\M$ is subdirectly irreducible.
%lc K Added
Let $\MT = (M,G,H,K,R,\Tp)$   where $\Tp$ is the discrete topology and 
$G, H$ and $R$ are chosen to satisfy the following conditions, 
 for some subset $\Omega$ of 
$\cat D(\fnt{U}(\M),\two)$:
\begin{newlist}

\item[{\rm (i)}]  given $a \ne b$ in $M$ there exists $\omega \in \Omega$ and 
an endomorphism~$h$ which is a composite of finitely many maps in~$G$ such that $\omega(h(a)) \ne \omega (h(b))$;   
\item[{\rm (ii)}]  $R$ is the set of   subalgebras which are maximal, with respect to 
inclusion, in sublattices of 
$\fnt{U}(\M)^2$
of the form
$
(\omega,\omega')^{-1}(\leq) := 
   \{\, (a,b) \in M^2 \mid \omega(a) \leq \omega'(b)\, \}$,
where $\omega,\omega' $ range over $\Omega$;
\item[{\rm (iii)}]  $G \cup H $ is
 the 
monoid 
of partial endomorphisms of~$\M$;
%lc Added:
\item[{\rm (iv)}] $K$ is the set of one-element subalgebras of $\M$ {\rm (}if any{\rm)}.
\end{newlist}
Then $\MT$ is an alter ego for $\M$ which strongly dualises $\ISP(\M)$.
\end{thm}

Conditions (i) and (ii) in Theorem~\ref{Thm:Piggyback} suffice to yield a duality.  Including conditions (iii) and (iv)  
is a brute force way 
to ensure that the duality is in fact strong.

In Section~\ref{Sec:TSM-SugOdd}  we shall  apply Theorem~\ref{Thm:Piggyback}  with $\M $ as $\Zed_{2n+1}$
($n \geq 0$) 
and in Section~\ref{Sec:TSM-SugEven} we apply the theorem with $\M$ as $\Zed_{2n}$ ($n \geq 1$).  
This enables us to identify a strongly dualising alter ego for any $\Zed_k$, 
in a uniform manner for the odd case and for the even case.
Before we embark on identifying the alter egos for these two families of dualities
 some comments about the conditions in the piggyback theorem should be made.  
In the even case, in which  $\End (\M) $ contains only the identity map, we need the elements of~$\Omega$ to separate the points of~$\M$.  
We shall see that for the odd case 
 a more economical choice  is available, courtesy of   
the endomorphism~$g$.

We shall reveal  a close connection between the graphs of the 
maps in condition (iii) 
 and the relations 
 in condition~(ii). 
This works to our advantage,  
 enabling us to 
streamline our alter egos.   
 To strip down an alter ego 
by removing superfluous  elements of $G \cup H \cup R$ 
 we rely on the notion of entailment,
as discussed in \cite[Chapter~2, Section~3]{CD98},  
together with an important technical result (see \cite[Chapter~3, Subsection~2.3]{CD98}
or \cite[Lemma 3.1]{CDquest}).
 We record 
a  simplified version.

\begin{lem}[$\MT$-Shift Strong Duality Lemma] \label{Lem:Shift} Assume that a finitely generated  quasivariety ${\CA =\ISP(\M)}$ is strongly dualised by  
%lc K added 
$\MT = (M; G,H,K,R,\Tp)$.
Then %an alter ego 
${\MT' = (M; G', H',K,R',\Tp)}$ will also yield a strong duality if it is obtained from~$\MT$ by 
\begin{newlist}
\item[{\rm (a)}] enlarging $R$, and/or 
\item[{\rm (b)}]  deleting from $G\cup H$ any elements
 expressible  as compositions of   the elements that remain.
\end{newlist}
Moreover, if $\MT'$ yields a duality on $\CA$ and is obtained from $\MT$ by deleting members of~$R$ then $\MT'$ yields a strong duality on~$\CA$.
\end{lem}
%%%%%%%%%%%%%%%%%%%

In the statement of Theorem~\ref{Thm:Piggyback} the restriction in (ii) to subalgebras which are \emph{maximal} is customary, and this provides a device
for simplifying an alter ego from the outset.  However this restriction is optional.
This is shown by examination of the proof of the theorem  (or more circuitously
by an appeal to Lemma~\ref{Lem:Shift}).   
(Any duality has an alternative alter ego in which operations and partial operations
are replaced by their graphs.  But we warn that this process may destroy strongness, 
so it is  not relevant to our study.)

%%%%%%%%%%%%%%%%%%%%%%%%%%%%%%%%%%%%%%%%%%%%%%%%%%%%%%%%%%%%
\section{Strong duality:  odd case}\label{Sec:Duality-SugOdd}
%%%%%%%%%%%%%%%%%%%%%%%%%%%%%%%%%%%%%%%%%%%%%%%%%%%%%%%%%%%%

In this section we  establish the  strong duality for $\SA_{2n+1}$  which we shall use in our application of the Test Spaces Method.  
We wish to apply Theorem~\ref{Thm:Piggyback} with~$\M$ as $\Zed_{2n+1}$.
We still have some work to do 
to identify the alter ego we want.

Define lattice homomorphisms  
$\delp 
$ and~$\delm 
$ from 
$ \fnt{U} (\Zed_{2n+1})$ to $\two$ by
\[
\delp 
(a) = 1 \Leftrightarrow  a \geq   1  \quad  \text{and} \quad 
\delm 
 (a) = 1 \Leftrightarrow a \geq 0.
\]
Note that  $\delm (a) = 1- \delp(\neg a)$, so that the apparent 
asymmetry in the definition, caused by the presence of~$0$, is illusory

\begin{lem}\label{lem:SepOdd}
Given $a \ne b$ in $\Zed_{2n+1}$ there exists 
an endomorphism~$h$ such that either 
$\delp 
(h(a)) \ne \delp 
(h(b)) $ or 
$\delm 
(h(a)) \ne 
\delm 
(h(b))$. 
\end{lem}
\begin{proof}
If $a< 0\leq b$, then $\delm 
(a)\neq \delm 
(b)$. Similarly, if $a\leq 0< b$, then $\delp
(a)\neq \delp 
(b)$.
Assume now that $0\leq a< b$. Then $\delp  
(g^{a}(a))\neq\delp 
(g^{a}(b))$.  Likewise,  if $a<b\leq 0$,  then $\delm 
(g^{-b}(a))\neq \delm 
(g^{-b}(b))$.
\end{proof}

We next 
investigate the relations in Theorem~\ref{Thm:Piggyback}(ii)
associated with the choice $\Omega = \{\delp 
,\delm 
\}$.  
 Let   $\w,\w' \in \Omega$.  
For each choice of $\w,\w'$ 
 we wish  to describe  the subalgebras~$S$ of $\Zed_{2n+1}^2 $
for which
\[ 
S \subseteq (\w,\w')^{-1} (\leq) = \{\, (a,b) \in \Zed_{2n+1}^2 \mid  \w(a)
\leq \w'(b)\, \}.
\]   
Necessarily, $S$ is disjoint from the rectangle $\w^{-1} (1) \times
\w'^{-1}(0)$.    Further constraints arise  because $S$ has to be closed
under~$\neg$ and also  under the implication $\to$. 
Proposition~\ref{prop:PrelOdd} identifies the permitted subalgebras for the four
possible cases.  
We illustrate the proofs in   
Figures
\ref{Fig:Beta}  and~\ref{Fig:AlphaBeta-etc}.
In the diagrams 
 the shaded regions indicate exclusion zones. For $Q \subseteq \Zed_{2n+1}^2$ we write 
$\{  \, \neg q \mid q \in Q\,\}$ as $\neg Q$.  
Our starting  point is Figure~\ref{Fig:Beta}(a), in which we include labels for  points of reference on the axes.
These labels are suppressed on subsequent diagrams.  
 We denote the converse of a binary relation~$r$  by $r^\smallsmile$.

\begin{prop}\label{prop:PrelOdd}
Let 
$S$ be a subalgebra of  $\Zed_{2n+1}^{2}$. 
\begin{newlist}
\item[{\rm (i)}] If $S\subseteq (\delm,\delm) 
^{-1}(\leq)$
 then there exists  $e\in \PEZ {2n+1} $ such that $S= \graph e$.
\item[{\rm (ii)}] If $S\subseteq     (\delp, \delp)
^{-1}(\leq)$  
then there exists  $e\in \PEZ {2n+1}$ such that $S^{\smallsmile}= \graph e$.
 \item[{\rm (iii)}] If $S\subseteq (\delm 
,\delp 
)^{-1}(\leq)$  then there exists  $e\in \PEZ {2n+1}$ such that $0\notin \im e \cup \dom e$ and  $S= \graph e$.
\item[{\rm (iv)}]  If $S \subseteq (\delp 
,\delm 
)^{-1}(\leq)$  then
$S$ or $S^\smallsmile$ is the graph of a partial endomorphism.
\end{newlist}
\end{prop}
\begin{proof} 
We adopt  the notation 
$[i,j]$ for  the set of elements~$k$ for which $i \leq k \leq j$  ($i,j,k 
\in \Zed_{2n+1}$).

Consider (i).
 Since $S$ is closed under $\neg$ and $\neg$ is an involution,  
\[
S\subseteq (\delm,\delm) 
^{-1}(\leq) \cap \neg  
  (\delm,\delm) 
^{-1}(\leq)
=[-n,-1]\times [-n,0]\cup\{(0,0)\}\cup [1,n]\times [0,n]
\] 
(see Figure~\ref{Fig:Beta}, (a) and~(b)).

\begin{figure}[h]
\begin{center}

\begin{tikzpicture}[scale=0.76]
\begin{scope}[xshift=-9.5cm]
\node [label=right:{\small  $(\delm,\delm)  
^{-1}(\leq)$}]   (b) at  (-3,2.7) {};
\node (ba) at  (-2.7,2.7) {};
\node  (bb) at  (-4.2,1) {};
\draw[thin] (bb) .. controls +(right:1cm) and +(left:1cm) .. (ba);

\node (a2) at  (-5.5,2) {\scriptsize $-$};
\node (a1) at  (-5.5,0.23) {\scriptsize  $-$};
\node  (a0) at  (-5.5,0) {\scriptsize $-$};
\node  (a-1) at  (-5.5,-0.23) {\scriptsize  $-$};
\node  (a-2) at  (-5.5,-2) {\scriptsize $-$};

\draw[thin] (-5.5,2.2)--(-5.5,-2.5);

%% horiz axis   
\node  (b2) at  (-1,-2.5) {\scriptsize $|$};
\node  (b1) at  (-2.77,-2.5) {\scriptsize $|$}; 
\node  (b0) at  (-3,-2.5) {\scriptsize $|$};
\node  (b-2) at  (-3.23,-2.5) {\scriptsize$|$};  
\node  (b-3) at  (-5,-2.5) {\scriptsize $|$};

\draw[thin] (-5.5,-2.5)--(-0.7,-2.5);

%%%%%%%%%%%%%%%%%%%%  midway axes, rectangle and shading 
\draw[thin, dashed] (-5.2,0)--(-.7,0);     
\draw[thin, dashed] (-3,-2.4)--(-3,2.2);   
  \draw[pattern color=gray,pattern=north east lines] (-3.13,-2.2) rectangle (-0.8,-0.13);

\draw (-5.2,-2.2) rectangle (-0.8,2.2);  

\node [label=below:{\small  (a)}]   (capa) at  (-3,-3.5) {};

\begin{scope}[xshift=1.5mm]
\node [label=left:{\scriptsize  $n$}]   (la2) at  (-5.5,2) {}; 
\node [yshift=.6mm, label=left:{\scriptsize $1$}]   (la1) at  (-5.5,0.26) {};  
\node [label=left:{\scriptsize  $0$}]   (la0) at  (-5.5,0) {}; 
\node [yshift=-.5mm, label=left:{\scriptsize  $-1$}]   (la-1) at  (-5.5,-0.26) {}; 
\node [label=left:{\scriptsize  $-n$}]   (la-2) at  (-5.5,-2) {}; 
\end{scope}

\begin{scope}[yshift=1mm]%

\node [label=below:{\scriptsize  $n$}]   (a2) at  (-1,-2.5) {}; 
\node [label=below:{\scriptsize  $1$}]   (a2) at  (-2.8,-2.5){}; 
\node [label=below:{\scriptsize  $0$\ }]   (a0) at  (-3,-2.5) {}; 
\node [label=below:{\scriptsize  $-1$\ \ \ \ }]   (a2) at  (-3.2,-2.5) {}; 
\node [label=below:{\scriptsize  $-n$}]   (a-3) at  (-5,-2.5) {};  
\end{scope}

\end{scope}

%%%%%%%%%%%%%%%%%%%%%%%%%%%%%%%%%%%%%%%%%%%%%% (b) 
%%%%%%%%%%%%%%%%%%%%%

\begin{scope}[xshift=-8cm]
\node [label=left:{\small 
\begin{tabular}{lr}
 $(\delm,\delm)
^{-1}(\leq)\cap$\\[-.5ex] 
\ \ \ \ $ \neg (\delm,\delm) 
^{-1}(\leq)$
\end{tabular}
}]   (lb1) at  (4.8,2.9) {};
\node  (bs1) at  (4.5,2.7) {};
\node   (bb1) at  (4.5,1) {};
\draw[thin] (bb1) .. controls +(left:1cm) and +(right:1cm) .. (bs1);

\node  (a2) at  (0.5,2) {\scriptsize $-$};
\node   (a1) at  (0.5,0.23) {\scriptsize $-$};
\node    (a0) at  (0.5,0) {\scriptsize$-$};
\node  (a-1) at  (0.5,-0.23) {\scriptsize$-$};
\node   (a-2) at  (0.5,-2) {\scriptsize$-$};

\draw[thin] (0.5,2.2)--(0.5,-2.5);

\node (b2) at  (1,-2.5) {\scriptsize $|$};
\node (b1) at  (3.23,-2.5) {\scriptsize $|$};
\node (b0) at  (3,-2.5) {\scriptsize $|$};
\node (b-1) at  (2.77,-2.5) {\scriptsize $|$};
\node (b-2) at  (5,-2.5) {\scriptsize$|$};

\draw[thin] (0.5,-2.5)--(5.2,-2.5);
%%%%%%%%%%%%%%%%%%%%
\draw[thin, dashed] (0.7,0)--(5.2,0);
\draw[thin, dashed] (3,2.2)--(3,-2.5);

\draw[pattern color=gray,pattern=north east lines] (2.87,-2.2) rectangle (5.2,-0.13);   %br  lower
\draw[pattern color=gray,pattern=north east lines] (0.8,0.13) rectangle (3.13,2.2);
\draw (0.8,-2.2) rectangle (5.2,2.2);
\node [label=below:{\small  (b)}]   (capb) at  (3,-3.5) {};
\end{scope}

%%%%%%%%%%%%%%%%%%%%%%%%%%%%%%%%%%%
%ImpVert  Beta    (c)  

\begin{scope} [xshift=5cm]

\begin{scope}[xshift = -6cm]
\draw[pattern color=gray,pattern=north east lines] (2.87,-2.2) rectangle (5.2,-0.13);   %br  lower
\draw[pattern color=gray,pattern=north east lines] (0.8,0.13) rectangle (3.13,2.2);
\draw (0.8,-2.2) rectangle (5.2,2.2);
\end{scope} 
\node (a2) at  (-5.5,2) {\scriptsize$-$};
\node (a0) at  (-5.5,0) {\scriptsize$-$};
\node (a-3) at  (-5.5,-2) {\scriptsize$-$};

\draw[thin] (-5.5,2.2)--(-5.5,-2.5);

\node (a2) at  (-1,-2.5) {\scriptsize$|$};
\node (a0) at  (-3,-2.5) {\scriptsize$|$};
\node (a-3) at  (-5,-2.5) {\scriptsize$|$};
\node (x) at  (-1.5,-2.5) {\scriptsize$|$};

\draw[thin] (-5.5,-2.5)--(-0.7,-2.5);

\draw[thin, dashed] (-5.2,0)--(-.7,0);
\draw[thin, dashed] (-3,-2.5)--(-3,2.2);

\node  (a) at  (-1.5,1) {};
\node  (b) at  (-1.5,1.5) {};
\node  (a-b) at  (-1.5,-1.5) {};

\draw[ fill=white]  (a) circle [radius=2pt];
\draw[fill=white]  (b) circle [radius=2pt];
\draw[fill=white] (a-b) circle [radius=2pt];
\draw [dotted, shorten <=-2pt, shorten >=-2pt] (a) -- (b);
\draw [dotted, shorten <=-2pt, shorten >=-2pt] (a-b) -- (a);
\draw [dotted, shorten <=-2pt, shorten >=-2pt] (a-b) -- (x);

\node (y) at  (-5.5,1.5) {$-$};
\node  (-y) at  (-5.5,-1.5) {$-$};
\draw [shorten <=-2pt, shorten >=-2pt,dotted] (y) -- (b);
\draw [shorten <=-2pt, shorten >=-2pt,dotted] (-y) -- (a-b);

\begin{scope}[xshift=1mm]
\node [label=left:{\small  $y$}]  
 (y) at  (-5.5,1.5) {}; 
\node [label=left:{\small  $-y$}]  
 (-y) at  (-5.5,-1.5) {};

\end{scope}

\begin{scope}[xshift=-1mm, yshift=6.5mm] 
\node [label=below left: {\small  $b$}]  
 (a) at  (-1.3,1) {};
\node [label= below  left %135
: {\small  $a$}] 
  (b) at  (-1.3,1.5) {};
\node [label= below left 
:{\small  $a\to b$}]  
 (a-b) at  (-1.3,-1.3) {};
\end{scope}

\begin{scope}[yshift=1mm] 
\node [label=below:{\small  $x$}]
   (x) at  (-1.5,-2.5) {}; 
\end{scope}

\node [label=below:{\small  (c)}]  (capc) at  (-3,-3.5) {};

\end{scope} 

\end{tikzpicture}
\end{center}
\caption{Illustration of proof of Proposition~\ref{prop:PrelOdd}(i)
}\label{Fig:Beta}
\end{figure}
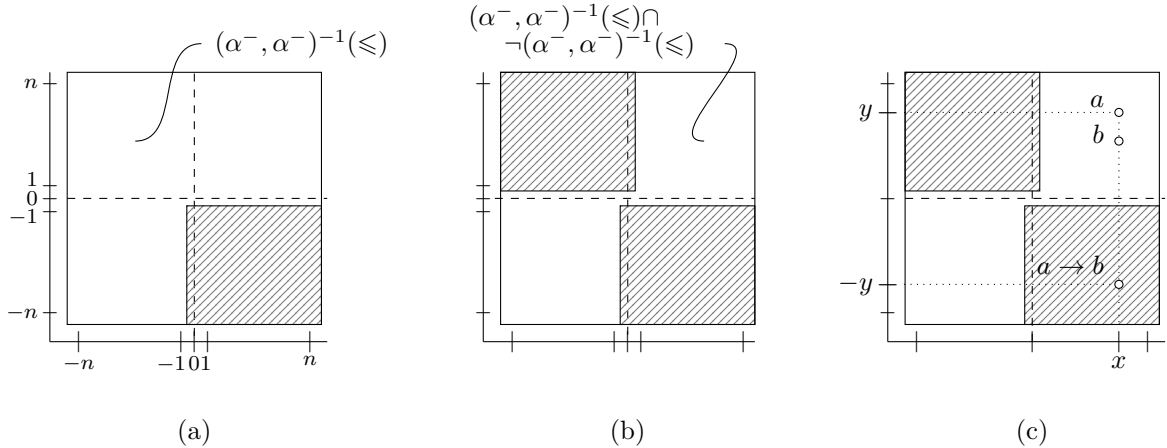

%%%%%%%%%%%%%%%%%%%%%%%%%%%%%%%%%%%%%%%%%%%%%%%%%%%%
%%%%%%%%%%%%%%%%%%%%%%%%%%%%%%%%%%%%%%%%%%%%%%%%%%%%

 Let  $a = (x,y) \in S$ where~$x$ is fixed.  We wish to show that $y$ is 
unique. 
If $x=0$  both  $(0,y)$ and $(0,-y)$ belong to $S\subseteq 
(\delm,\delm) 
^{-1}(\leq)$. 
Then $\delm 
(y) = \delm 
(-y) = 0$ and hence $y=0$.
We may now assume without loss of generality that $x > 0$
(otherwise we can consider instead $\neg a = (-x, -y)$, which also belongs to~$S$). 
 Suppose for contradiction that there exists  $y'\neq y$  such that $b=(x,y')\in S$. Since $\delm 
(x)=1$,  we know that$y,y'\in {\delm}
^{-1}(1)=[0,n]$.
If $0\leq y'< y$  then
\[
a\to b = (x\to x, y\to y')=(x,-y\wedge y')=(x,-y).
\]
Since $-y<0$ we must have   $\delm 
(-y)=0\ngeqslant \delm 
(x)$.  But 
this is incompatible with closure of $S$ under implication and $S$ being 
contained in $  (\delm,\delm) 
^{-1}(\leq)$ (see Figure~\ref{Fig:Beta}(c)).
If $y<y$', a similar argument considering   $b\to a$
leads to a contradiction.  
Therefore $S$ is the graph of a (possibly partial) map $e$ from $\Z_{2n+1}$ to $\Z_{2n+1}$. Since~$S$ is a subalgebra of $\Z_{2n+1}^2$, it follows that $e$
is a partial endomorphism.

We now prove (ii).
First observe that 
\[
\bigl((\delp,\delp)   
^{-1}(\leq)\cap\neg 
(\delp,\delp 
)^{-1}(\leq) 
\bigr) ^{\smallsmile} =(\delm,\delm) 
^{-1}(\leq)\cap\neg 
 (\delm,\delm) 
^{-1}(\leq) 
 \]
 (see Figures~\ref{Fig:Beta}(b) and~\ref{Fig:AlphaBeta-etc}(a)).
Hence,
$S\subseteq (\delp,\delp) 
^{-1}(\leq)$  implies $S ^{\smallsmile}
\subseteq (\delm,\delm)  
^{-1}(\leq)$. Now (ii) follows from (i).

\begin{figure}[h]
\begin{center}
	\begin{tikzpicture}[scale=.76]
%alpha-1 cap alpha -1
\begin{scope}[xshift=-12cm]
\node [label=left:{\small 
\begin{tabular}{lr}
$(\delp 
,\delp   
)^{-1}(\leq) \ \cap$\\[-.5ex]
\ \  $ \neg(\delp
,\delp
)^{-1}(\leq)$
\end{tabular}
}]   (lab) at  (-1.3,3.1) {};
\node []   (base) at  (-1.7,3.1) {};
\node []   (bs) at  (-1.7,1) {};

\node (a2) at  (-5.5,2) {\scriptsize $-$};
\node (a2) at  (-5.5,0.23) {\scriptsize $-$};
\node (a0) at  (-5.5,0) {\scriptsize $-$};
\node (a2) at  (-5.5,-0.23) {\scriptsize $-$};
\node (a-3) at  (-5.5,-2) {\scriptsize $-$};

\draw[thin] (-5.5,2.2)--(-5.5,-2.5);

\node (a2) at  (-1,-2.5) {\scriptsize$|$};
\node (a2) at  (-2.77,-2.5) {\scriptsize$|$};
\node (a0) at  (-3,-2.5) {\scriptsize$|$};
\node (a2) at  (-3.23,-2.5) {\scriptsize$|$};
\node (a-3) at  (-5,-2.5) {\scriptsize$|$};

\draw[thin] (-5.5,-2.5)--(-0.7,-2.5);

\draw[thin] (bs) .. controls +(left:1cm) and +(right:1cm) .. (base);

\draw[thin, dashed] (-5.2,0)--(-.7,0);
\draw[thin, dashed] (-3,-2.4)--(-3,2.2);
\draw[pattern color=gray,pattern=north east lines] (-2.87,-2.2) rectangle (-0.8,0.13);  %*
\draw[pattern color=gray,pattern=north east lines] (-5.2,2.2) rectangle (-3.13,-0.13);
\draw (-5.2,-2.2) rectangle (-0.8,2.2);

\node [label=below:{\small  (a)}]   (capa) at  (-3,-3) {};    %-3.5  --raised
\end{scope}

%%%%%%%%%%%%%%%%%%%%%%%%%%%%%%%%%%%%%%%%%%%%%%%%%%%%%%%
\begin{scope}[xshift=-10cm]
%Betaalpha-1\cap negbetaalpha-1
\node [label=left:{\small  
\begin{tabular}{lr}
$(\delm
,\delp 
)^{-1}(\leq) \ \cap$\\[-.5ex]
\ \  $ \neg(\delm
,\delp
)^{-1}(\leq)$
\end{tabular}
}]   (b1) at  (5.2, 3.1) {};  
\node []   (bs1) at  (4.8, 3.1) {};
\node []   (bb1) at  (4.5,1) {};
\node []   (b2) at  (1.2,3.1) {};
\node []   (bb2) at  (1.5,-1.2) {};
\draw[thin] (bb1) .. controls +(left:1cm) and +(right:1cm) .. (bs1);
\draw[thin] (bb2) .. controls +(right:1cm) and +(left:1cm) .. (b2);

\node (a2) at  (0.5,2) {\scriptsize $-$};
\node  (a2) at  (0.5,0.23) {\scriptsize$-$};
\node  (a0) at  (0.5,0) {\scriptsize$-$};
\node  (a2) at  (0.5,-0.23) {\scriptsize$-$};
\node  (a-3) at  (0.5,-2) {\scriptsize$-$};

\draw[thin] (0.5,2.2)--(0.5,-2.5);

\node %[label=below:{\small  $-n$}]  
 (a2) at  (1,-2.5) {\scriptsize$|$};
\node % [label=below:{\small  $1$}] 
  (a2) at  (3.2,-2.5) {\scriptsize$|$};
\node % [label=below:{\small  $0$}] 
  (a0) at  (3,-2.5) {\scriptsize$|$};
\node % [label=below:{\small  $-1$\ \ \ \ }]
   (a2) at  (2.8,-2.5) {\scriptsize$|$};
\node % [label=below:{\small  $n$}] 
  (a-3) at  (5,-2.5) {\scriptsize$|$};

\draw[thin] (0.5,-2.5)--(5.2,-2.5);

\draw[thin, dashed] (0.7,0)--(5.2,0);
\draw[thin, dashed] (3,2.2)--(3,-2.5);
\draw (0.8,-2.2) rectangle (5.2,2.2);
\node [label=below:{\small  (b)}]   (capb) at  (3,-3) {};

\draw[pattern color=gray,pattern=north east lines] (0.8,-0.13) -- (2.87,-0.13) -- (2.87,-2.2) -- 
(5.2,-2.2) -- (5.2,0.13) -- (3.13,0.13) -- (3.13,2.2) -- (0.8,2.2) -- (0.8,-0.13) ;
\end{scope} 

%%%%%%%%%%%%%%%%%%%%%%%%%%%%%%%%%%%%%
%AlphaBeta-etc   (c)  

\begin{scope}[xshift= -2cm]
%alphaBeta-1\cap neg alphabeta-1
\node [label=left:{\small  
\begin{tabular}{lr}
$(\delp
,\delm ) 
^{-1}(\leq) \ \cap$\\[-.5ex]
\ \  $ \neg(\delp 
,\delm 
)^{-1}(\leq)$
\end{tabular}
}]
(b1) at  (5.2,3.1) {};  
 %%%
\node []   (bs1) at  (1.2,3.1) {};
\node []   (bb1) at  (1.5,-0.5) {};
\draw[thin] (bb1) .. controls +(right:1cm) and +(left:1cm) .. (bs1);

\node %[label=left:{\small  $n$}]  
 (a2) at  (0.5,2) {\scriptsize$-$};
\node %[label=left:{\small  $1$}] 
  (a2) at  (0.5,0.23) {\scriptsize$-$};
\node %[]  
 (a0) at  (0.5,0) {\scriptsize$-$};
\node % [label=left:{\small  $-1$}]
   (a2) at  (0.5,-0.23) {$-$};
\node %[label=left:{\small  $-n$}]
   (a-3) at  (0.5,-2) {\scriptsize$-$};

\draw[thin] (0.5,2.2)--(0.5,-2.5);

\node %[label=below:{\small  $-n$}] 
  (a2) at  (1,-2.5) {\scriptsize$|$};
\node  %[label=below:{\small  $1$}]  
 (a2) at  (3.2,-2.5) {\scriptsize$|$};
\node % [label=below:{\small  $0$}] 
  (a0) at  (3,-2.5) {\scriptsize$|$};
\node % [label=below:{\small  $-1$\ \ \ \ }] 
  (a2) at  (2.8,-2.5) {\scriptsize$|$};
\node % [label=below:{\small  $n$}]
   (a-3) at  (5,-2.5) {\scriptsize$|$};

\node [label=below:{\small  (c)}]   (capc) at  (3,-3) {};

\draw[thin] (0.5,-2.5)--(5.2,-2.5);

\draw[thin, dashed] (0.7,0)--(5.2,0);
\draw[thin, dashed] (3,2.2)--(3,-2.5);
\draw[pattern color=gray,pattern=north east lines] (3.13,-2.2) rectangle (5.2,-0.13);
\draw[pattern color=gray,pattern=north east lines] (0.83,0.13) rectangle (2.9,2.2);
\draw (0.83,-2.2) rectangle (5.2,2.2);

\node %[label=135:{\small  $a$}]  
 (a) at  (3,1.5) {};
\node %[label=-45:{\small  $b$}]  
 (b) at  (4,0) {};
\node %[label=45:{\small  $a\vee b$}]  
 (ab) at  (4,1.5) {};

\node %[label=below:{\small  $x$}]  
 (x) at  (4,-2.5) {\scriptsize$|$};
\node %[label=left:{\small  $y$}] 
  (y) at  (0.5,1.5) {\scriptsize$-$};
\node %[label=left:{\small  $y$}]  
 (-y) at  (0.5,-1.5) {\scriptsize$-$};

\node    
(al) at  (3.23,1.76) {\small $a$ };
\node  
 (abl) at  (4.55,1.82) {\small  $a\vee b$};

\node %[label=135:{\small  $(a\vee b)\to b$}] 
  (a-b) at  (4,-1.5) {};
\node  (a-bl) at (2.8,-1.2)  {\small  $(a\vee b)\to b$};

%\node %[label=-45:{\small  $b$}]  
% (b) at  (4,0) {};
\node  (bl)  at   (4.2,.3) {\small $b$};

\begin{scope}[yshift=1mm]
\node [label=below:{\small  $x$}]   (x) at  (4,-2.5) {};
\end{scope} 
\begin{scope}[xshift=1mm]
\node [label=left:{\small  $y$}]   (y) at  (0.5,1.5) {};
\node [label=left:{\small  $-y$}]   (-y) at  (0.5,-1.5) {};
\end{scope}

\draw[thin, dotted] (y)--(ab);
\draw[thin, dotted] (x)--(ab);
\draw[thin, dotted] (-y)--(a-b);

\draw[fill=white] (a) circle [radius=2pt];
\draw[fill=white] (b) circle [radius=2pt];
\draw[fill=white] (ab) circle [radius=2pt];
\draw[fill=white] (a-b) circle [radius=2pt];
\end{scope}

	\end{tikzpicture}
\end{center}
\caption{Proofs of Proposition~\ref{prop:PrelOdd}(ii)--(iv)}\label{Fig:AlphaBeta-etc}
\end{figure}
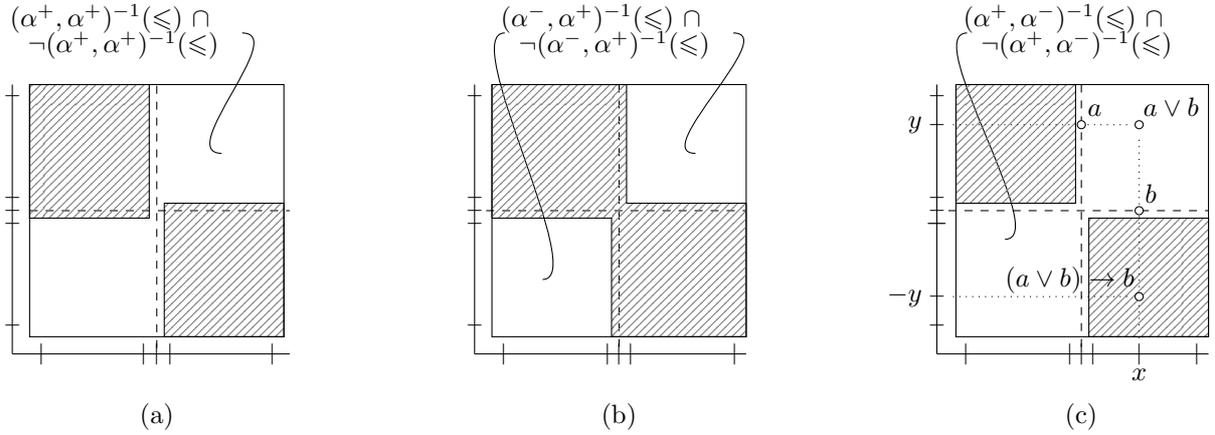

For (iii), 
observe that
\[
(\delm ,
\delp 
)^{-1}(\leq)\cap\neg(\delm 
,\delp 
)^{-1}(\leq) =
(\delm,\delm )
^{-1}(\leq)\cap\neg 
 (\delm,\delm) 
^{-1}(\leq)
\setminus \bigl(\Zed_{2n}\times \{0\}\bigr) 
 \]
(see Figure~\ref{Fig:AlphaBeta-etc}(b)).
Hence, by (i), $S$ is the graph of some partial endomorphism $e$. Since  
\[(\delm 
,\delp 
)^{-1}(\leq)\cap\neg 
(\delm 
,\delp 
)^{-1}(\leq)
\cap \Bigl (\bigl( \Zed_{2n}\times \{0\}\bigr) \cup \bigl(\{0\}\times \Zed_{2n}\bigr)\Bigr) = \emptyset, 
\]
 it follows that  $0\notin \im e \cup \dom e$, which concludes the proof of (iii).

Finally, to prove (iv) we will prove that  $S\subseteq (\delp 
,\delm 
)^{-1}(\leq)$ implies 
$S\subseteq (\delm,\delm) 
^{-1}(\leq)$ or $S\subseteq  (\delp,\delp) 
^{-1}(\leq)$ and the result
 will then  follow from (i) and (ii).
Note  that  
\begin{multline*}
(\delp 
,\delm )^{-1}(\leq)\cap \neg 
 (\delp 
,\delm 
)^{-1}(\leq) 
= 
\bigl( (\delp,\delp) 
^{-1}(\leq)\cap \neg 
(\delp\delp) 
^{-1}(\leq)
\bigr) \cup \bigl(\{0\} \times \Zed_{2n}\bigr) 
%\\&
\\
=\bigl(  (\delm,\delm) 
^{-1}(\leq)\cap \neg 
 (\delm,\delm) 
^{-1}(\leq) 
\bigr) \cup ( \Zed_{2n} \times  \{0\})
\end{multline*}  
(see
 Figure~\ref{Fig:AlphaBeta-etc}(c)).  
Suppose there exist $a=(x,0), b=(0,y)\in S$ such that 
$x\neq0\neq y$. 
Assume, without loss of generality, that $0< x,y$. Then 
$ 
{(a\vee b)\to a=(-x\vee x, -y\wedge 0) = (x, -y)}
$ 
(again see Figure~\ref{Fig:AlphaBeta-etc}(c)).
However $\delp 
(x) = 1 \nleqslant 0=\delm 
(-y)$. 
Hence, $S\cap \bigl((\{0\} \times \Zed_{2n})\cup (\Zed_{2n}\times \{0\}) \bigr)$ is contained in either $\{0\} \times \Zed_{2n}$ or in $\{0\} \times \Zed_{2n}$. 
Therefore $S\subseteq (\delp,\delp) 
^{-1}(\leq)$ or ${S\subseteq (\delm,\delm) 
^{-1}(\leq)}$. 
\end{proof}

We are ready to present 
our duality theorem for the odd case.

\begin{thm}\label{Thm:SugDualityOdd}
For $n= 
1,2,\ldots
$,
 the topological structure 
\[ 
\twiddleodd 
=  
(\Zed_{2n+1}; f_0,f_1, f_2, \ldots, f_n, g, {\bf 0},\Tp)
\]
 is an alter ego for 
$\Zed_{2n+1}$ which 
 yields a strong 
duality on $\SA_{2n+1}$.  
\end{thm}

\begin{proof} 
Fix $n$ and let $\M =\Zed_{2n+1}$.   
Proposition~\ref{IrrIdx}  tells us that  every non-trivial subalgebra of~$\Zed_{2n+1}$ is subdirectly irreducible.
So Theorem~\ref{Thm:Piggyback} is applicable provided we satisfy its conditions
(i)--(iv).

By Lemma \ref{lem:SepOdd}  we
satisfy~(i) be taking  $\Omega =\{ \delp 
,\delm 
\}\subseteq\cat D(\fnt{U}(\Zed_{2n+1}),\two)$.
By Proposition~\ref{prop:PrelOdd}, every relation as in 
Theorem~\ref{Thm:Piggyback}(ii)   is the graph of a member of $G \cup H=
\PEZ {2n+1}$ or is the converse of such a relation.  (For the present application
of the theorem the restriction in item (ii)  to subalgebras
which are maximal could be dispensed with.)  
%lc added.
The only one-element subalgebra of $\Zed_{2n+1}$ is $\{0\}$. %Hence $K$ only contains the constant ${\bf 0}$. 
It follows from \cite[Chapter~2, Section~3]{CD98}  
that after discarding all the relations in (ii) from the alter ego  we have  a new dualising alter ego.  
Lemma~\ref{Lem:Shift}  tells us
that this  still yields a strong duality.

Finally, we 
apply Lemma~\ref{Lem:Shift}
again, now with 
$\MT$ having 
$R = \emptyset$:  we may replace 
$G \cup H = \PEZ {2n+1}$
by the  generating set $\{f_0,f_1, 
f_2,\ldots,f_n,g\}$  given in Proposition~\ref{prop:genpeOdd}.   
\end{proof}

We remark for subsequent use that any  morphism in the dual category
$\IScP(\twiddleodd 
)$ 
 preserves all elements of 
$\PEZ{2n+1}$ and not just those present in the alter ego.
  
%%%%%%%%%%%%%%%%%%%%%%%%%%%%%%%%%%%%%%%%%%%%%%%%%%%%%%%%%%
\section{Test Spaces Method applied to Sugihara algebras: odd case}
\label{Sec:TSM-SugOdd}
%%%%%%%%%%%%%%%%%%%%%%%%%%%%%%%%%%%%%%%%%%%%%%%%%%%%%%%%%%%

In this section we work with $\SA_{2n+1}$, treating $n$ as fixed. 
Step 0 of the Test Spaces Method for $\SA_{2n+1}$  is covered by  
Theorem~\ref{Thm:SugDualityOdd}.
 To accomplish Step 1 of the Test Spaces Method we need to calculate the
dual space $\D(\Zed_{2n+1})$ for the choice $\twiddleodd
$ of alter ego
given in the theorem.

\begin{prop}[{\bf Step 1}] \label{Prop:dualZ2n1}  
Up to an isomorphism of structures, 
the dual space $\D(\Zed_{2n+1})$  has universe 
$\{\, U\in \mathcal{P} (\{ 0,\ldots ,n\} ) \mid 0\in U\,\}$. On this set,  an element  
$e \in G \cup H$ acts by restriction on those sets $U$ for which  $U \subseteq \dom e$  and is undefined otherwise; also the interpretation of ${\bf 0}\in K$  is the set $\{0\}$.
\end{prop}

\begin{proof}

Together, 
Propositions~\ref{prop:phomOdd} and~\ref{prop:subalg}
 supply a bijection 
from $\D(\Zed_{2n+1})$  onto the family  
of subsets 
 of $\{ 0,1, \dots, n\}$ which contain~$0$: 
this assigns  to 
an endomorphism
$h$ the non-negative elements
of its image $\img h$.  
The lifting of $e \in \PEZ {2n+1}$ to the dual space $\D(\Zed_{2n+1})$  is the map which is defined on those endomorphisms~$h$ for which $\img h \subseteq \dom e$
and which sends any such~$h$ to $e\circ h$. Since $\img( e\circ h)=e(\img h)$ the bijection obtained above is an isomorphism of structures.
\end{proof}

We now put  
forward a candidate 
$(\Y, \mu, \nu)$
 for the TS-configuration  that we require for Step~2.
The underlying set of any TS-configuration
for $\SA_{2n+1}$ is  a set of
$(n+1)$-tuples
 of elements of $\twiddleodd 
$, equipped with  the structure it inherits
from $\twiddleodd 
^{\!\!n+1}
$.  
Let $U = \{0, b_1,\ldots,b_k\} $, where  $\{b_1,\ldots,b_k\} $ is a (possibly empty) subset of $\{1,\ldots,n\}$ whose elements are 
 listed without repetitions and in increasing order.  With this convention,  $U$ is uniquely determined by the $(n+1)$-tuple
$(0, \ldots, 0 ,b_1,\ldots ,b_k)$, where there are $n+1-k$ zeros.

Define 
   \begin{align*}
   \Y 
 = \{\, \bvec{a}=(a_1,\ldots,a_{n+1})
\mid \exists\,  j \geq 1
 [\forall i\leq j (0\leq a_i= a_j)  
\mbox{ and }\forall k\geq j (a_k < a_{k+1})]\,\}.
\end{align*} 
The value of~$j$ depends on $
\bvec{a}$.  When convenient we write $j$ as  $j_{\bvec{a}}$.
In what follows, 
we shall  make use several times of the fact  (the \emph{uniqueness property}) 
that, for each choice of~$j$, there is one and only one element  
$(a_1,\ldots,a_{n+1})$ of~$\Y$ whose set of coordinates is  
$\{a_j,\ldots,a_{n+1}\}$.   

Define $\nu \colon  \D(\Zed_{2n+1})\to \Y$ 
by 
\begin{multline*}  
\nu 
( \{0, b_1, \ldots,  b_k\}) =(0,\ldots, 0,
 b_1, \ldots,  b_k) ,
\text{where the first 
$(n+1-k)$ 
 coordinates are zero} \\
\text{
and }
    0<b_p<b_q\text{ when }1 \leq p<q \leq k.   
\end{multline*}
Given $\bvec{a}=(a_1,\ldots,a_{n+1}) \in \twiddleodd ^{n+1}
$, we
 define $\mu$ to be the map  that assigns to $\bvec{a}$
the unique element of $\Y$ whose set of coordinates is $\{|a_1|,\ldots,|a_{n+1}|\}$. Thus 
if $\{|a_1|,\ldots,|a_{n+1}|\}=\{c_1,\ldots,c_r\}$ with $0\leq c_1<c_2<\cdots<c_r$, 
then $\mu(a_1,\ldots,a_{n+1})=(c_1,\ldots,c_1,c_2,\ldots,c_r)$.

\begin{prop}[{\bf Step 2}] \label{Prop:X2n1-as-TSOdd}
  Define $\Y$, $\mu$ and $\nu$ as above.  Then 
$ (\Y ,\mu,\nu) $ is a $\text{\rm TS}$-configuration.
\end{prop}  

\begin{proof} In our alter ego for $\Zed_{2n+1}$,  we have  $G = \{g\}$, $H = \{f_0,f_1,f_2, \ldots, f_{n}\}$, $K=\{{\bf 0}\}$ and $R =\emptyset$.
Since all the maps in $G\cup H$ are order preserving, it is straightforward to check  
 that $\Y$ is closed under the action of~$G\cup H$
and that $\nu$ is an embedding.
%lc added
 Moreover, ${\bf 0}=(0,\ldots,0)\in \Y$.
Since modulus is a term function on $\Zed_{2n+1}$ it is preserved by any element of $\PEZ{2n+1}$, and so $\mu $ is a morphism.  Also $\mu$ is clearly 
surjective.
\end{proof}

In preparation for  Steps 3 and  5 we  introduce some notation. 
 For each $k$ with  $1 \leq k\leq n$ 
 there is a unique element in~$\Y$,  \textit{viz.}
$\bvec{k} := (1, \ldots, 1, 2,\ldots,k)$, whose set of coordinates is $\{1,2,\ldots,k\}$;
 it has  
$n-k+2$  coordinates equal to~$1$. 
Also define $\bvec{n+1}:=(0,1,\ldots,n)$.

Let  $\sigma\colon \Y\to \{1,\ldots,n+1\}$ be the map sending  $(u_1,\ldots,u_{n+1})\in\Y$
to the cardinality of the set
$\{u_{1},\ldots, u_{n+1}\}$.
Observe that  
$\sigma(\bvec{k})=k$  and that $j_{\bvec{k}} = n-k+2$,  for $1 \leq k \leq n+1$.
   
\begin{lem}\label{lem:genclaimOdd}
Let $\Y$ be the structure defined above.  Let $\bvec{u} = (u_1,\ldots, u_{n+1})$ belong to~$\Y$.  Then 
\begin{enumerate}
\item[{\rm (i)}] 
if $u_1>0$ then $\bvec{u}$ lies in the substructure generated by $\{ \bvec{1}. \ldots, \bvec{n}\}$;  
\item[{\rm (ii)}] if 
$u_1 = 0$ then $\bvec{u}$ lies in the substructure generated by $\bvec{n+1}$.
\end{enumerate}
\noindent 
Consequently, $\{\bvec{1},\ldots,\bvec{n+1}\}$ generates~$\Y$.
\end{lem}

\begin{proof} 
Consider (i).  Here $\sigma(\bvec{u})=k$,  where $k \leq n$ and 
$j_{\bvec{u}} = j_{\bvec{k}}$.  
Corollary~\ref{cor:k-trans}   supplies 
$e\in \PEZ{2n+1}$ such that  
$u_i = e(\bvec{k}_i)$
 for each~$i$.   Since $\bvec{u}$ preserves~$e$ (coordinatewise) (i)
 follows.

Now consider (ii).  
If $\sigma(\mathbf{u})= n+1$ then $\mathbf{u}= \bvec{n+1}$ and there is nothing to prove.   
Assume $\sigma(\mathbf{u}) = k \leq n$.   Then
$j_{\bvec{u} } = n-k+2$.
 Also
$\bvec{v}:= g^{n-k+1}(\bvec{n+1}) = (0,\ldots ,0,1, \ldots , k-1)$
and $j_{\bvec{v}}=j_{\bvec{u}}$.   By Corollary~\ref{cor:k-trans}(ii)
there exists $e \in \PEZ{2n+1}$ such that $e(\bvec{v}) = \bvec{u}$.  We conclude that  $\bvec{u}$ belongs to the substructure of~$\Y$ generated by $\bvec{n+1}$. 
\end{proof}

\begin{prop}[{\bf Step 3}] \label{Prop:X_SX}
  $\Y$ is join-irreducible in $\mathcal{S}_\Y$. 
\end{prop}  
\begin{proof} 
Let  $\varphi \colon \Y \to \Y$ be a morphism such that 
%lc modified:
%$(0,1,2,\ldots,n)\in \img\varphi $. 
$ \bvec{n+1}\in \img\varphi $. 
We claim that $\varphi $ is the identity map.
Let $\bvec{x}         
$ be such that 
%$\varphi (\bvec{x}  )=(0,1,2,\ldots,n)$. 
$\varphi (\bvec{x}  )= \bvec{n+1}$. 
Since $(0,1,\ldots, n) \notin \dom e$ for any $e \in H$ and it
is the only element of $\Y$ with this property, $ \bvec{x} 
$ is not in the domain of any element of $H$, so 
%$\bvec{x} =(0,1,2,\ldots,n)$. 
$\bvec{x} =\bvec{n+1}$. 
That is, $\bvec{n+1}$ %$(0,1,2,\ldots,n)$
 is fixed by $\varphi $. By Lemma~\ref{lem:genclaimOdd},
the morphism  $\varphi $ fixes any element in~$\Y$ having first coordinate zero.

Now let $\bvec{y} 
=(b_1,\dots,b_{n+1})$ be any element of~$\Y$.
By the uniqueness property, $\bvec{y}
$ is the only element of~$\Y$ with set of coordinates $\{b_j,\ldots,b_{n+1}\}$,
where $j = j_{\bvec{y}}$. 
Therefore $\bvec{y}  
$ is the only element of $\Y$ that is in  $\dom f_i$ 
if and only if $i\notin \{b_j,\ldots,b_{n+1}\}$. 
Hence $\{c_1,\ldots,c_{n+1}\}\subseteq\{b_j,\ldots,b_{n+1}\}$,  where $\varphi (\bvec{y})=
(c_1, \ldots, c_{n+1})$. 
Observe that 
$
g^{b_j}(\bvec{y})=(0,\ldots,0,b_{j+1}-b_{j},\ldots,b_{n+1}-b_j)$,
where $j = j_{\bvec{y}}$, and  that 
this element is fixed by $\varphi $.  
Since $g^{b_j}(b)=b-b_j$ for any  $b\in  \{b_j,\ldots,b_{n+1}\}=\{b_1,\ldots,b_{n+1}\}$, and in particular when $b = c_i$ ($1 \leq i \leq n+1$), 
it follows that
\[
(c_1-b_j,\ldots,c_{n+1}-b_j)=g^{b_j}(\varphi (\bvec{y}  
))=
\varphi (g^{b_j}(\bvec{y}  
))
=g^{b_j}(\bvec{y} 
)=(b_{1}-b_{j},\ldots,b_{n+1}-b_j).
\]
Therefore, $c_i=b_i$ for each $i\in\{0,1,\ldots,n\}$.
\end{proof}

\begin{prop}[{\bf Step 5:  admissibility algebra}]\label{Prop:AnOdd}
Let $\alg{B}\subseteq\Zed_{2n+1}^{\!\!n+1}$ be the subalgebra whose elements $(a_1,\ldots,a_n,a_{n+1})$ satisfy the following conditions:
\begin{itemize}
\item[{\rm (i)}]
$a_1\in\{1,-1\}$;
\item[{\rm (ii)}]
$a_k\neq 0$ for each $k\in\{1,\ldots,n\}$;
\item[{\rm (iii)}] there exists a value  of $ j$, necessarily unique,  with $j \leq n$ such that 
\begin{itemize}
\item[{\rm (a)}]
$|a_k| = 1 $ if $ k\leq j$  and   
{\rm (b)} $g(a_{k+1})=a_{k}$ if $n>k\geq j$;
\end{itemize}
\item[{\rm (iv)}] 
$g(a_{n+1})=g(a_n)$.

\end{itemize}
Then $\alg{B}$ is a subalgebra of $\Zed_2\times \Zed_4\times\cdots\times \Zed_{2n}\times \Zed_{2n+1}$ and it is isomorphic to $\E (\Y)$.
\end{prop}

\begin{proof}
The case $n=1$ follows from a straightforward calculation. 

Now fix $n\geq2$. 
We 
define a map $t$ 
which we shall show is an isomorphism
from $\E(\Y)$ to  $\alg{B}$. For 
$x\in \E(\Y)$, let
\[
t(x):=(x(\bvec{1}),x(\bvec{2}),\ldots, x(\bvec{n}),x(\bvec{n+1})).
\]

\noindent {\bf Claim 1:} $t(x)\in \B$.
\begin{proof}
If $|x(\bvec{k})|\leq 1$ for all $k\leq n$, then the claim is true.

 Assume now that there exists $k\leq n$ with $|x(\bvec{k})|\ne 1$. 
The tuple  
$\bvec{1}=(1,1,\ldots,1)$ is in the domain of every $f_i$ except $f_1$.  
Since~$x$ preserves all the partial operations, 
this  constrains  
 $x(\bvec{1})$ to lie in $\{-1,1\}$ and so  (i) holds.
 There exists a maximum $j \leq n$ such that 
$|x(\bvec{i})|=1$ for $i \leq j$  and  $|x(\bvec{j+1})|\ne 1$.
Then 
$ %\[
g^2(x(\bvec{j+1}))= g(x(\bvec{j}))\in \{g(1),g(-1)\}= \{0\}$.
%\]
 Since $|x(\bvec{j+1})|\neq 1$, it follows that $x(\bvec{j+1})\in\{-2,2\}$.
If $k$ is such that $j<k< n$  then
$ %\[
 g^{2}(x(\bvec{k+1}))= g(x(g(\bvec{k+1}))) =g(x(\bvec{k}))$.
%\]
Therefore,  by induction,  we can see that $|x(\bvec{k+1})|>|x(\bvec{k})|\geq|x(\bvec{j+1})|=2$. Since 
$g$ is injective when restricted to $\{-n,\ldots, -2,2,\ldots,n\}$  
and it sends  positive (resp.~negative) elements to positive (resp.~negative)
elements
it also follows that  $g(x(\bvec{k+1}))=x(\bvec{k})$.   
We have shown that (ii) and (iii) hold, with~$j$ as above.  Finally, 
 since $g(\bvec{n+1})=g(\bvec{n})$, it follows that $g(x(\bvec{n}))=x(g(\bvec{n}))=x(g(\bvec{n+1}))=g(x(\bvec{n+1}))$. Thus, $t(x)$ satisfies (iv).
 \end{proof}

\noindent  {\bf  Claim 2}:  $t\colon \E(\X) \to \alg{B} $ is an injective homomorphism.

\begin{proof}
By Lemma~\ref{lem:genclaimOdd}, $\{\bvec{1},\ldots,\bvec{n+1}\}$ generates $\Y$, hence any 
$x\in\E(\Y)$ is uniquely determined by $x(\bvec{1}),\ldots,x(\bvec{n+1})$. Therefore $t$ is injective.
 Moreover, $t\colon \E(\Y)\to \B$ is coordinatewise an evaluation map, hence a homomorphism. 
\end{proof}

\noindent  {\bf  Claim 3}: The homomorphism $t$ maps $\E(\Y)$  onto $\B$.  
\begin{proof}

The following function enables  us manipulate signs to  create elements of~$\Y$ from suitable $(n+1)$-tuples: 
 For a real number $r$,  we let   
\[
\sgn r=
\begin{cases}
0 &\mbox{if } r=0,\\
\dfrac{r}{|r|} &\mbox{otherwise}.
\end{cases}
\]
On $\Zed_{2n+1}$, partial endomorphisms commute with $\sgn$.

Let $\bvec{a}=(a_1,\ldots, a_{n+1})\in \B$. 
Here  $\bvec{a}  = (a_1, \ldots, a_{n+1}  ) $  can take  two possible  forms. The first form is 
$
(\pm 1, \ldots, \pm 1,  2, \ldots, n-j, a_{n+1})$.
 Here $j$ can take any value 
for which $1 \leq  j \leq n$, and the choices of sign in the first $j$ coordinates are  arbitrary.
The last coordinate, $a_{n+1}$, lies between~$0$ and~$n-1$,  its value linked to that of $n-j+1$ by condition~(iv).   
The second form is the same, except that
the signs of the last $n+1-j$ coordinates are reversed.
We need to define a morphism $x_{\mathbf{a}}\colon\Y\to \twiddleodd
$ such that $x_{\mathbf{a}}(\bvec{i})=a_{i}$ for $i\in\{1,\ldots, n+1\}$. 
Assume without loss of generality that~$\bvec{a}$ takes the first form; the other case is  handled likewise.
We define
\[
x_{\bvec{a}} (\bvec{u}) : = \begin{cases} 
                        \sgn{\bvec{a}_{\sigma(\bvec{u})} } \cdot \bvec{u}_{n-j+2} & \text{if } \bvec{u} \ne \bvec{n+1}, \\
                                a_{n+1}+ 1  &\text{if } \bvec{u} = \bvec{n+1}.
                        \end{cases}
\]

We first confirm that $x_{\bvec{a}}(\bvec{k}) = a_k$ for each~$k$.
If $k \leq j$ then  
 \[
x_{\bvec{a}} (\bvec{k}) = 
\sgn a_k\cdot
\bvec{k}_{n+2-j}=\sgn a_k \cdot 1= \sgn a_k\cdot |a_k| = a_k,
\]
as required.
If $k>j$ then 
$
\bvec{k}_{n+2-j}=k-j+1$.
Also, $a_k-(k-j-1)=g^{k-j-1}(a_k)= 
a_{j+1}=2$, by  (iii)(b). Hence, $a_k=k-j+1$.
 Therefore $x_{\mathbf{a}}(\bvec{k})=a_k$.
Moreover, 
$
 x_{\mathbf{a}}(\bvec{n+1})=
a_{n+1}$.

It  remains to prove  that $x_{\mathbf{a}}\in \E(\Y)$.
%lc added:
Note that $x_{\mathbf{a}}(\bvec{0})=\sgn{\bvec{a}_{\sigma(\bvec{0})} } \cdot \bvec{0}_{n-j+2} =0$. Therefore $x_{\mathbf{a}}$ preserves $\bvec{0}$.
Now, let $i\in \{0,\ldots,n\}$ and $\mathbf{u}=(u_1,\ldots,u_{n+1})\in \dom f_i$. 
Since $\bvec{n+1}$ is not in the domain of $f_i$, necessarily ${ \mathbf u}\neq \bvec{n+1}$. 
Since $f_i$ is injective, $\sigma({\mathbf{u}})=\sigma({f_i(\mathbf{u})})$ 
and $f_i(\bvec{u}) \neq \bvec{n+1}$.
Hence 
\[
x_{\mathbf{a}}(f_i(\mathbf{u}))
=
\sgn  a_{\sigma({\mathbf u})}\cdot
f_i( u_{n+j-2})   
= f_{i}(\sgn  a_{k_\mathbf{u}}\cdot
 u_{n+j-2})
=f_{i}(x_{\mathbf{a}}(\mathbf{u})).
\]
We now prove that $x_{\bvec{a}}$ preserves~$g$.  
Assume first that $\bvec{u} = \bvec{n+1}$.  Note  that $g(\bvec{n+1})=g(\bvec{n}) $ and that $g(a_n) = g(a_{n+1})$ by (iv).  
 Therefore 
\begin{multline*}
x_{\mathbf a}(g(\mathbf{n+1}))=x_{\mathbf a}(g(\mathbf{n}))
=\sgn a_{n}\cdot g(\bvec{n})_{n+2-j}\\
=x_{\mathbf a}(g(\mathbf{n}))
=g(x_{\mathbf a}(\bvec{n}))=g(a_n)=g(a_{n+1})=g(x_{\mathbf a}(\mathbf{n+1})).
\end{multline*}
Assume now that $\bvec{u} \in \Y \setminus \{\bvec{n+1}\}$.  If $\sigma(\bvec{u}) =  \sigma(g(\bvec{u}))$ then 
  $x_{\bvec{a}}
(g(\bvec{u})) = g(x_{\bvec{a}}(\bvec{u}))$. Suppose $\sigma(\bvec{u}) \ne 
\sigma(g(\bvec{u}))$.  We have three cases to consider.

\noindent \textit{Case 1}:  $\sigma(\bvec{u}) \leq j$.  Since $\sigma(g(\bvec{u}))\leq \sigma (\bvec{u}) \leq j$, we have  
$\sgn \bvec{a}_{\sigma (\bvec{u})} = 0 = \sgn \bvec{a}_{\sigma (g(\bvec{u}))}$.  
Then 
\[
x_{\bvec{a}}(g(\bvec{u}) )= \sgn \bvec{a} _{\sigma g(\bvec{u})} \cdot 
g(u)_{n+2-j}   = 0 = g(0 \cdot \bvec{u}_{n+2-j}) = g(\sigma(\bvec{u} )\cdot
\bvec{u}_{n+2-j}) = g(x_{\bvec{a}} (\bvec{u})).
\]

\noindent \textit{Case 2}:  $\sigma (\bvec{u}) > j+1$.  Then $
\sgn \bvec{a}_{\sigma(\bvec{u})} =   
\sgn  \bvec{a}_{\sigma(g(\bvec{u}))}$.    Again 
$x_{\bvec{a}}(g(\bvec{u})) = g(x_{\bvec{a}}(\bvec{u}))$.

\noindent \textit{Case 3}:  $\sigma(\bvec{u}) = j+1$.  Then $\sigma(g(\bvec{u}))=j$.  Observe  that $\sigma (\bvec{u}) = 1 + \sigma(g(\bvec{u}))$,  which  implies that $0,1 \in \{u_1, \ldots, u_j\}$.   
Note that  
$\bvec{u}_{n+2-j} = \bvec{u}_{n+1 -(j-1)}$.  
Then $0 = \bvec{u}_i$ for $i \leq n+1 - \sigma(\bvec{u}) + 1$ and
$1= \bvec{u}_{n+1 - \sigma(\bvec{u}) + 2}  = 
\bvec{u}_{n+1 - (j+1) + 2} = \bvec{u}_{n+2-j}$.  It follows that $
g(\bvec{u})_{n+2-j} =0$.  
Hence
\[
x_{\bvec{a}}(g(\bvec{u})) = \sgn \bvec{a}_{\sigma(g(\bvec{u}))} \cdot 
g(\bvec{u})_{n+2-j} = 0= 
g(\sgn \bvec{a}_{\sigma(\bvec{u})} \cdot 1) = 
g(\sgn \bvec{a}_{\sigma(\bvec{u})} \cdot \bvec{u}_{n+2-j}) = g(x_{\bvec{a}}(\bvec{u})).\qedhere
\]
\end{proof}

Claims 1--3.~establish that  $\E(\Y) \cong \B$.
\end{proof}

We refer to Section~\ref{sec:sumup}
for more information about~$\B$ and its relationship to 
 $\F_{\SA_{2n+1}}
(n+1)$.

%%%%%%%%%%%%%%%%%%%%%%%%%%%%%%%%%%%%%%%%%%%%%%%%%%%%%%%%
\section{Strong duality: even case} \label{Sec:Duality-SugEven}
%%%%%%%%%%%%%%%%%%%%%%%%%%%%%%%%%%%%%%%%%%%%%%%%%%%%%%%

For the odd Sugihara quasivarieties  $\SA_{2n+1}$ we were able
to base our piggyback dualities  on just two
`carrier maps', $\delp
$ and $\delm
$, from $\Zed_{2n+1}$ into $\{ 0,1\}$, thanks to the presence of non-trivial
endomorphisms; recall Lemma~\ref{lem:SepOdd}.   
In the even case we have the opposite extreme:  $\End \Zed_{2n} = \{\id_{\Zed_{2n}} \}$.  
 To satisfy the separation condition (i) in Theorem~\ref{Thm:Piggyback} we    
take 
$\Omega = \cat D(\fnt{U}(\Zed_{2n}),\two)\setminus \{\boldsymbol 0, \boldsymbol 1\}$,
where the excluded maps are those taking constant value~$0$ or $1$.  

We  label the elements of $\Omega$ as 
$\beta^-_{n-1}, \ldots, \beta^-_{1};\ \beta 
;\  \beta^+_1, \ldots, \beta_{n-1}^+$, where 
for $a \in \Zed_{2n}$,
\[
\beta^-_i(a) =  1  \Leftrightarrow a \geq -i; 
\qquad  
\beta %\alpha 
(a) = 1 \Leftrightarrow a > 0;
\qquad 
\beta_i^+(a) = 1 \Leftrightarrow a > i.
\]

As compared with the odd case  we have a proliferation of piggyback relations to 
describe. 
Moreover, the analogues of our exclusion diagrams in Figures~\ref{Fig:Beta} and~\ref{Fig:AlphaBeta-etc}  involve additional possible scenarios;
 see Figure~\ref{Fig:BetaBeta}.
  Lemma~\ref{lem:omega-facts} sets out elementary facts which  simplify our analysis.

\

%%%%%%break 
    
\begin{figure}[h]
\begin{center}
	\begin{tikzpicture}[scale=.76]

%Beta-1
\begin{scope}[xshift=-6.5cm]

\node [label=left:{\small  
$(\beta_{i}^{+},\beta_{j}^{+})^{-1}(\leq)$}]   (b) at  (-4,2.7) {};
\node []   (bb) at  (-3.2,1) {};
\draw[thin] (bb) .. controls +(left:1cm) and +(right:1cm) .. (b);

%Y axis
\node %[label=left:{\small  $n$}]  
 (yn) at  (-5.5,2) {\scriptsize$-$};
\node   (ai1) at  (-5.5,1.5) {\scriptsize$-$};
\node %[label=left:{\small  $j$}]   
(yi) at  (-5.5,1.2) {\scriptsize$-$};
\node []   (a0) at  (-5.5,0) {\scriptsize$-$};
\node %[label=left:{\small  $-n$}]  
 (y-n) at  (-5.5,-2) {\scriptsize$-$};

\begin{scope}[xshift=1.1mm]
\node  [label=left:{\scriptsize$j\!+\!1$}] (ai1) at  (-5.5,1.5) {};
\node [label=left:{\scriptsize  $j$}]   
(yi) at  (-5.5,1.15) {};
\end{scope}

\draw[thin, dashed] (-5.2,0)--(-.7,0);

\draw[thin] (-5.5,2.2)--(-5.5,-2.5);

%X axis
\node %[label=below:{\small  $n$}]  
 (xn) at  (-1,-2.5) {\scriptsize$|$};
\node    (x2) at  (-2,-2.5) {\scriptsize$|$};
\node %[label=below:{\small  $i$}]
   (x2) at  (-2.3,-2.5) {\scriptsize$|$};
\node %[label=below:{\small  $0$}] 
  (x0) at  (-3,-2.5) {\scriptsize$|$};
\node %[label=below:{\small  $-n$}]  
 (x-n) at  (-5,-2.5) {\scriptsize$|$};

\begin{scope}[yshift=.7mm]
\node  [label=below:{\scriptsize $i\!+\!1$}]   (x2) at  (-1.85,-2.5) {};
\node [label=below:{\scriptsize  $i$}]
   (x2) at  (-2.4,-2.5) {};
\end{scope}

\draw[thin] (-5.5,-2.5)--(-0.7,-2.5);

\draw[thin, dashed] (-3,-2.4)--(-3,2.2);

\draw[pattern color=gray,pattern=north east lines] (-2.2,-2.2) rectangle (-0.8,1.3);
\draw (-5.2,-2.2) rectangle (-0.8,2.2);

\node [label=below:{\small  (a)}]   (capa) at  (-3,-3) {};
\end{scope}

%%%%%%%%%%%%%%%%%%%%%%%%%%%%%%%%%%%%%%%%%%%%%%5
\begin{scope}[xshift=-5.5cm]
%Beta-1\cap negbeta-1
\node [label=left:{\small  
\begin{tabular}{lr}
$(\beta_{i}^{+},\beta_{j}^{+})^{-1}(\leq)\, \cap$\\[-.5ex]
$ 
\ \  \neg
(\beta_{i}^{+},\beta_{j}^{+})^{-1}(\leq)$
\end{tabular}
}] 
  (b1) at  (4,3.1) {};
\node []   (bs1) at  (3.5,3.1) {};
\node []   (bb1)  at  (3.1,-.5) {};   
\draw[thin] (bb1) .. controls +(left:1cm) and +(right:1cm) .. (bs1);

%X axis
\node %[label=left:{\small  $n$}] 
  (a2) at  (0.5,2) {\scriptsize$-$};
\node   (a2) at  (0.5,1.5) {\scriptsize$-$};
\node %[label=left:{\small  $j$}] 
  (a2) at  (0.5,1.2) {\scriptsize$-$};
\node []   (a0) at  (0.5,0) {\scriptsize$-$};
\node   (a2) at  (0.5,-1.5) {\scriptsize$-$};
\node %[label=left:{\small  $-j$}]  
 (a2) at  (0.5,-1.2) {\scriptsize$-$};
\node %[label=left:{\small  $-n$}] 
  (a-3) at  (0.5,-2) {\scriptsize$-$};

\draw[thin] (0.5,2.2)--(0.5,-2.5);

\begin{scope}[xshift=1.1mm]
\node [label=left:{\scriptsize  $j$}] 
  (a2) at  (0.5,1.2) {};
\node [label=left:{\scriptsize  $-j$}]  
 (a2) at  (0.5,-1.2) {};
\end{scope}

%Y axis
\node % [label=below:{\small  $-n$}] 
  (a2) at  (1,-2.5) {\scriptsize$|$};
\node    (a2) at  (4,-2.5) {\scriptsize$|$};
\node% [label=below:{\small  $i$ }]
   (a2) at  (3.7,-2.5) {\scriptsize$|$};
\node %[label=below:{\small  $0$}] 
  (a0) at  (3,-2.5) {\scriptsize$|$};
\node    (a2) at  (2,-2.5) {\scriptsize$|$};
\node %[label=below:{\small  $-i$ }]
   (a2) at  (2.3,-2.5) {\scriptsize$|$};
\node %[label=below:{\small  $n$}] 
  (a-3) at  (5,-2.5) {\scriptsize$|$};

\draw[thin] (0.5,-2.5)--(5.2,-2.5);

\begin{scope}[yshift=.7mm]
\node  [label=below:{\scriptsize  $i$ }]
   (a2) at  (3.7,-2.5) {};
\node [label=below:{\scriptsize  $-i$}]
   (a2) at  (2.3,-2.5) {};
\end{scope}

\draw[thin, dashed] (0.7,0)--(5.2,0);
\draw[thin, dashed] (3,2.2)--(3,-2.5);
\draw (0.8,-2.2) rectangle (5.2,2.2);

\draw[pattern color=gray,pattern=north east lines] (3.8,-2.2) rectangle (5.2,1.3);
\draw[pattern color=gray,pattern=north east lines] (0.8,-1.3) rectangle (2.2,2.2);

\node [label=below:{\small  (b)}]   (capb) at  (3,-3) {};
\end{scope} 

%%%%%%%%%%%%%%%%%%%%%%%%%%%%%%%%%%%%%%%%%%%%%%%%%%%%%%

%%%%%%%%%%%%%%%%%%%%%%%%%%%%%%%%%%%%%%%%%%%%%%%%%%%%%%%%%
\begin{scope}[xshift=1.5cm]
%Beta-1\cap negbeta-1
\node [label=left:{\small 
\begin{tabular}{lr}
 $(\beta_{j}^{-},\beta_{i}^{-})^{-1}(\leq)\, \cap$\\[-.5ex] 
$\ \ \ \neg(\beta_{j}^{-},\beta_{i}^{-})^{-1}(\leq)$
\end{tabular}
}]  
 (b1) at  (4.7,3.1) {};
\node (bs1)  at  (4.3,3.1) {}; 
\node []   (bb1) at  (4.9,-.5) {};
\draw[thin] (bb1) .. controls +(left:1cm) and +(right:1cm) .. (bs1);

%X axis
\node %[label=left:{\small  $n$}]  
 (a2) at  (0.5,2) {\scriptsize$-$};
\node   (a2) at  (0.5,1) {\scriptsize$-$};
\node %[label=left:{\small  $i$}] 
  (a2) at  (0.5,0.7) {\scriptsize$-$};
\node []   (a0) at  (0.5,0) {\scriptsize$-$};
\node   (a2) at  (0.5,-1) {\scriptsize$-$};
\node %[label=left:{\small  $-i$}]
   (a2) at  (0.5,-0.7) {\scriptsize$-$};
\node %[label=left:{\small  $-n$}] 
  (a-3) at  (0.5,-2) {\scriptsize$-$};

\draw[thin] (0.5,2.2)--(0.5,-2.5);

\begin{scope}[xshift=1.1mm]
\node [label=left:{\scriptsize  $i$}] 
  (a2) at  (0.5,0.7) {};
\node [label=left:{\scriptsize $-i$}]
  (a2) at  (0.5,-0.7) {};
\end{scope}

%Y axis
\node %[label=below:{\small  $-n$}]  
 (a2) at  (1,-2.5) {\scriptsize$|$};
\node    (a2) at  (4.5,-2.5) {\scriptsize$|$};
\node %[label=below:{\small  $j$ }] 
  (a2) at  (4.2,-2.5) {\scriptsize$|$};
\node %[label=below:{\small  $0$}]  
 (a0) at  (3,-2.5) {\scriptsize$|$};
\node    (a2) at  (1.5,-2.5) {\scriptsize$|$};
\node %[label=below:{\small  $-j$ }] 
  (a2) at  (1.8,-2.5) {\scriptsize$|$};
\node %[label=below:{\small  $n$}]  
 (a-3) at  (5,-2.5) {\scriptsize$|$};

\draw[thin] (0.5,-2.5)--(5.2,-2.5);

\begin{scope}[yshift=.7mm]
\node [label=below:{\scriptsize  $j$ }] 
  (a2) at  (4.2,-2.5) {};
\node [label=below:{\scriptsize  $-j$ }] 
  (a2) at  (1.8,-2.5) {};
\end{scope}

\draw[thin, dashed] (0.7,0)--(5.2,0);
\draw[thin, dashed] (3,2.2)--(3,-2.5);
\draw (0.8,-2.2) rectangle (5.2,2.2);

\draw[pattern color=gray,pattern=north east lines] (1.7,-2.2) rectangle (5.2,-0.8);
\draw[pattern color=gray,pattern=north east lines] (0.8,0.8) rectangle (4.3,2.2);

\node [label=below:{\small  (c)}]  
 (capb) at  (3,-3) {};

\end{scope}

	\end{tikzpicture}
\end{center}
\caption{Illustration of proof of Proposition~\ref{prop:newAemr}}\label{Fig:BetaBeta}
\end{figure}
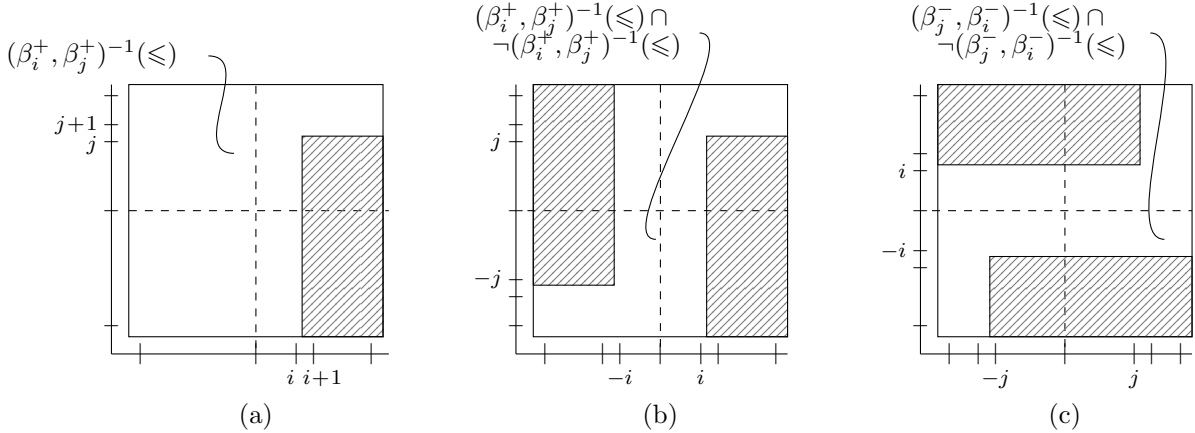

\begin{lem}  \label{lem:omega-facts}  Let 
$S$ be a subalgebra of $\Zed_{2n}^2$.
Let  $i,j\in \{ 1,\ldots , n-1\}$.
 Then the following  hold:
 \begin{itemize}
 \item[\rm (i)] $S \subseteq 
(\beta_i^+,\beta_j^+) ^{-1}(\leq )$
 if and only if 
$S^\smallsmile \subseteq 
(\beta_j^{-},\beta _i^{-})^{-1} (\leq)$. 
\item[\rm (ii)] $S 
\subseteq 
(\beta_i^{\pm},\beta 
) ^{-1}(\leq )$
 if and only if 
$S^\smallsmile \subseteq 
(\beta 
,\beta _i^{\mp})^{-1} (\leq)$. 
\item[\rm (iii)]$S 
\subseteq (\beta_i^+,\beta_j^-) ^{-1}(\leq )$
 if and only if 
$S\subseteq(\beta_i^{+},\beta _j^{+})^{-1} (\leq)$ or $S\subseteq(\beta_i^{-},\beta _j^{-})^{-1} (\leq)$. 
\end{itemize}
\end{lem} 

\begin{proof} 
Consider (i).  Let $(a,b) \in \Zed_{2n}^2$.
 Then 
\begin{align*}
(a,b) \in (\beta_i^+,\beta_j^+) ^{-1}(\leq )& \ \cap \ \neg  (\beta_i^+,\beta_j^+) ^{-1}(\leq )\\
    &\Longleftrightarrow \beta_i^+(a)\leq\beta_j^+( b)\mbox{ and }\beta_i^+(\neg a)\leq\beta_j^+(\neg b) \\
    &\Longleftrightarrow  1-\beta_i^-(\neg a)\leq1-\beta_j^-(\neg b)\mbox{ and } 
1-\beta_i^{-}( a)\leq 1-\beta_j^{-}(b)\\
    &\Longleftrightarrow \beta_j^{-}(\neg b) \leq \beta_i^{-}(\neg a) \mbox{ and }\beta_j^{-}(b) \leq \beta_i^{-}( a)\\
    & \Longleftrightarrow (b,a) \in (\beta_j^{-\ell},\beta_i^{-k}) ^{-1}(\leq ) \cap \neg( (\beta_j^{-\ell},\beta_i^{-k}) ^{-1}(\leq ) ).
\end{align*}

Item (ii) follows by the same argument but replacing $\beta_i^{+}$ and $\beta_i^{-}$ ($\beta_j^{+}$ and $\beta_j^{-}$) with $\beta $.

To prove (iii) assume 
that  
there exist $(a,b)$ and $(c,d)$ in~$S$ such that
$(a,b)\notin
(\beta_i^{+},\beta _j^{+})^{-1} (\leq)$ and $(c,d)\notin
(\beta_i^{-},\beta _j^{-})^{-1} (\leq)$. Since $S\subseteq(\beta_i^+,\beta_j^-) ^{-1}(\leq )$, we have  $i<a$, $-j<b\leq j$, $-i<c\leq i$, and $d<-j$. Hence 
\[  
(\neg(a,b)\wedge (c,d)) \to (c,d) = (-a\wedge c, -b\wedge d) \to (c,d)
%\\
%&
=(-a,d)\to (c,d) = (a\vee c,d)=(a,d).
\]
However, $(a,d)\notin (\beta_i^+,\beta_j^-) ^{-1}(\leq )$ and this  contradicts the fact that $ S$ is a subalgebra of $\Zed_{2n}^2$.
\end{proof}

Compositions of partial endomorphisms $h$ and relations~$\eq_m$ are central to our characterisation of piggyback relations.  
We define  $h \newcirc \eq_m$ by composition in the expected way.  Lemma~\ref{lem:relprod}(i) relies on the last statement in Proposition~\ref{lem:peclaim}.
In (ii) and (iii)  the relations are those we would get if we replaced~$h$ by $\graph h $ and formed the relational product. 
 However, as we indicated
in Section~\ref{Sec:NatDual}, 
we    
cannot blithely replace a partial operation by its graph.

\begin{lem} \label{lem:relprod}
Let  $h \in \PEZ{2n}$  and let $\eq_m$ 
$(m\geq 1)$
be as defined earlier.
\begin{enumerate}
\item[{\rm (i)}]   $h$ is invertible and $(\graph h)^\smallsmile = \graph h^{-1}$.
\item[{\rm (ii)}]  $h \newcirc \eq_m :=  \{\,  (x,y) \in  \Zed_{2n}^2 \mid 
x \in \dom h \ \& \  h(x) \eq_m y\,\}$ is a subalgebra of $\Zed_{2n}^2$.
\item[{\rm (iii)}]  
  $
\eq_m \newcirc  h:= \{\,  (x,y) \in \Zed_{2n}^2 \mid 
\exists z \bigl(  x \eq_m z  \ \& \ (z \in \dom h \   \& \ y = h(z))
\bigr)\}
$ and  
$\eq_m \newcirc h = h^{-1} \newcirc \eq_m$.
\end{enumerate}
\end{lem}

%%%%%%%%%%%%%%%%%%%%%%%%%%%%%%%%%%%%%%%%%%%%%%%%%%%

\begin{prop}  \label{prop:newAemr}  
  Let  $S$ be  a given subalgebra of  $\Zed_{2n}$ which is maximal in $(\w,\w')^{-1}(\leq)$,
where $\w,\w' \in \Omega$.  
 The possible forms for $S$ are indicated below.

\newcommand{\eqhm}{ h   \newcirc  \eq_m}
\newcommand{\eqmh} {\eq_m \newcirc  h}
\newcommand{\grh}{\graph h}

\begin{center}
\begin{tabular} {l|ccc}
                      &  $ \phantom{mm}\beta_k^-$ \phantom{mm}&  $\phantom{mm}\beta 
\phantom{mmm}$   & \ $\phantom{mm}\beta_\ell^+ $
\phantom{mm}
 \bigstrut[b]  \\
\hline
$\beta_i^-  $  &    $\eqmh$  &     $\grh$       &   $\grh$  \bigstrut[t] \\ 
$\beta 
$  &      $\eqmh$    &     $\grh$   & $ \grh$\\ 
$\beta_j^+$   &  
  $\eqmh$ or  
 $\eqhm$ 
    & $\eqhm$  &   $\eqhm$
  \end{tabular}
\end{center}
 \noindent Here
 $i,j,k,\ell \in \{1,\ldots,n-1\}$.  
The row  label specifies the choice of $\omega$, the column  label specifies $\omega'$;   $h$ denotes some element of $\PEZ{2n} $ and 
$m \in \{1,\ldots,n-1\}$, 
where  both~$h$ and~$m$ will depend on~$S$ and on $\omega$ and $\omega'$.  
\end{prop}

\begin{proof}

\noindent\textit{Case  1:}\ 
Consider the choices  
 $\omega = \beta_l^- $ and $\omega' =
\beta_k^-$.  
Let  
\[
s=\max\{\, r \mid (r,b)\in S \implies -r\leq b\leq r\, \}.
\] 
Observe that necessarily $s\geq i$.

Suppose first that  $s=n$.  Then $S\subseteq \Zed_{2n}\times \Zed_{2k}\subseteq (\beta^-_{i},\beta^-_k)^{-1}(\leq)$.
 In this case let $m=n$ and $h=\id _{\Zed_{2k}}$, and then 
\[ 
\eq_m \newcirc  h= (\Zed_{2n}\times \Zed_{2n})\newcirc  {\id }_{\Zed_{2k}}
 = \Zed_{2n}\times \Zed_{2k}.
\]
Since $\Zed_{2n}\times \Zed_{2k}\subseteq (\beta^-_{i},\beta^-_k)^{-1}(\leq)$ and $S$ is maximal we deduce  that  $S=\Zed_{2n}\times \Zed_{2k}$.

Now suppose $s<n$.  Let  $(a,b)\in S$ be such that $b\notin \Zed_{2k}$. Assume without loss of generality that $b<-k$.  Then $\beta^{-}_i(a)\leq \beta^{-}_k(b)=0$, that is, $a<-i$. 

If there exists $a'\in\Zed_{2n}$ such that $(a',b)\in S$ and $a<a'$, then 
\[
(a,b)\to (a',b)= (a\to a', b\to b))= (-a \vee a',b).
\]
    However $\beta^-_i(-a\vee a')=1$ and $\beta^{-}_k(b) = 0$, that is, $(a,b)\to (a',b)\notin S$. We arrive at a similar contradiction if we assume $a'<a$. Hence $a$ is the unique element such that $(a,b)\in S$.
Therefore,  for each $b\notin \Zed_{2k}$,  if there exists $a\in \Zed_{2n}$ for which  $(a,b)\in S$, then such an $a$ is unique and $a\notin \Zed_{2k}$. 
Now let $h$ be the partial map defined as follows:
\[
h(c) = \begin{cases}
c & \mbox{if }c\in \Zed_{2k},\\
a & \mbox{if }c\notin \Zed_{2k} \mbox{ and }(a,c)\in S,  \\
\mbox{undefined}& \mbox{otherwise}.
\end{cases}
\]
Since $S$ is a subalgebra, 
$h$ is indeed a partial endomorphism.
Moreover 
\begin{align*}
\eq_s \newcirc  h
&= 
\{\,(a,c)\in\Zed_{2n}^2\mid \exists b\in \Zed_{2k}
\left( ( a\eq_s b\ \& \ 
c=b )   
\text{ or } \exists 
b\notin \Zed_{2k} (
a \eq_s b \ \& \
c = h(b)
\right)
\, \}\\
 &=\bigl (\Zed_{2s}\times \Zed_{2k} \bigr)\cup  \{\, (a,c)\in\Zed_{2n}^2\mid \exists b\notin \Zed_{2k}
\left( a=b\ \& \ 
(b,c)\in S\right) \, \}\\
&= \bigl(\Zed_{2s}\times \Zed_{2k} \bigr)\cup  \{\,(a,c)\in S\mid a\notin \Zed_{2s}\, \}.  
\end{align*}
Then
$S\subseteq \eq_s \newcirc  h\subseteq   (\beta^-_{i},\beta^-_k)^{-1}(\leq)$. 
Maximality of $S$ now implies that  $S=\eq_s\newcirc  h$.

%%%%%%%%%%%%%%%%%%%%%%%%%%%%%%%%%%%%%%%%%%

\medskip

\noindent\textit{Case  2:} \ 
Here 
we treat the choices
$\w = \beta_j^+$ and $\w' = \beta_\ell^+$.
 By Lemma~\ref{lem:omega-facts}, $S^\smallsmile\subseteq  
 (\beta^-_{\ell^-},\beta^-_j)^{-1}(\leq)$. 
We now apply Case~1 to $S^\smallsmile$ and make use of 
Lemma~\ref{lem:relprod}.

\medskip

\noindent\textit{Case 3.}\ 
The case in which $\w= \beta_j^+$ and $\w'=\beta_k^-$ can now be handled by
appealing to 
 Lemma~\ref{lem:omega-facts}(iii). 
 
%%%%%%%%%%%%%%%%%%%%%%%%%%%%%%%%%%%%%%%%%%%%%%

\medskip

\noindent\textit{Case 4.}
The proof  for the  case
$\w = \beta_i^-$ and $\w' = \beta_\ell^+$
follows the same 
lines as that for  Case 2, but is simpler:  consider the converse  relation $S^\smallsmile$ and note that $(\graph h)^\smallsmile = \graph h^{-1}$ for any
$h \in \PEZ{2n}$.

\medskip

\noindent\textit{ Residual cases:} \   It remains to consider the cases in which one or both 
of $\w$ and  $\w'$ is $\beta
$.

\begin{enumerate}[(a)]

\item 
$\w = \beta 
$ and $\w'= \beta_i^-$:  proceed, {\it mutatis mutandis},  
 as in Case 1.  
\item 
 $\w = \beta _j^+$ and $\w'= \beta % \alpha
$: make use of Case (a) and 
Lemma~\ref{lem:omega-facts}(ii).    

\item 
 $\w=\w'=\beta 
$: proceed, \textit{mutatis mutandis}, as in Case~3.

\item
$\w = \beta_i^- $ and $\w' = \beta 
$:  
argue as  
for Case 3.

\item $\w = \beta 
$ and $\w'=\beta_\ell^+$: 
appeal to 
the result from Case (c) and Lemma~\ref{lem:omega-facts}(ii).
\end{enumerate}

This completes our characterisation of the piggyback relations.  
\end{proof}

We now present our strong duality theorem for the even case.
Our strategy, as in  
the proof of  Theorem~\ref{Thm:SugDualityOdd}, is to start from the alter ego that the piggyback theorem supplies and then to
adjust that alter ego to arrive at the one 
we want. 
This time 
we need to take advantage of the restriction in 
Theorem~\ref{Thm:Piggyback}(ii) to \emph{maximal} subalgebras of $(\w,\w')^{-1}(\leq)$.

\begin{thm}\label{Thm:SugDualityEven}
For $n= 
1,2,\ldots$,
 the topological structure 
\[
\twiddleeven 
=  
(\Zed_{2n}
;  f_2, \ldots, f_n, g,   \eq_1,   
\ldots, \eq_{n-1}, \Tp)
\]
 is an alter ego for the algebra
$\Zed_{2n}$ which 
 yields a strong 
duality on $\SA_{2n}=\ISP(\Zed_{2n})$.
\end{thm}

\begin{proof} 
In applying the Piggyback Strong 
Duality Theorem we take $\Omega =
\{\beta %\alpha
,  \beta_1^\pm,  \ldots, \beta_{n-1}^\pm\}$,  
take $G\cup H$ to be the 
entire monoid of partial endomorphisms and let $R$ be the set of  piggyback relations 
specified in condition (ii) of the theorem.
%lc added:
Observe that since $\Zed_{2n}$ does not have one element subalgebras. $K=\emptyset$. 
We now add  to $R$ the relations 
$\eq_1, \ldots, \eq_{n-1}$ and the graphs of all partial endomorphisms,   to form a set $R'$. The 
$\MT$-Shift Strong Duality Lemma tells us 
we   still have   a strong duality.
Now we delete redundant relations, again calling on the $\MT$-Shift Strong Duality Lemma.  By  Proposition~\ref{prop:newAemr},  the set $\PEZ{2n} \cup
\{ \eq_1, \ldots, \eq_{n-1}\}$ entails all the relations in $R'$: converse, graph, partial endomorphism 
action are all admissible constructs.  
Hence this set supplies  a strongly dualising alter ego.
Finally, we may delete all partial endomorphisms except those in the generating set 
$\{f_2,\ldots, f_{n-1},g\}$ for $\PEZ{2n}$. 
\end{proof}

%%%%%%%%%%%%%%%%%%%%%%%%%%%%%%%%%%%%%%%%%%%%%%%%%%%%%%%%
\section{Test Spaces Method for Sugihara algebras: even case}\label{Sec:TSM-SugEven}

%%%%%%%%%%%%%%%%%%%%%%%%%%%%%%%%%%%%%%%%%%%%%%%%%%%%%%%%
We now apply the Test Spaces algorithm to $\SA_{2n}$.  We work with $n$ fixed and $n \geq 1$.
For Step~0, we  employ the strong duality set up in Theorem~\ref{Thm:SugDualityEven}.

\begin{prop}[{\bf Step 1}]  \label{Prop:dualZ2n}  
 The dual space $\D(\Zed_{2n})$ has universe  $\{ \id_{\Zed_{2n}}\}$.   In this structure, any  partial endomorphism acts trivially and any relation $\eq_m$
acts as  equality.   
\end{prop}

\begin{proof}  
  The universe of $\D(\M)$ is simply $\End \M$, by definition of the functor~$\D$. 
  Proposition~\ref{prop:phomEven} shows that  $\Zed_{2n}$ has no endomorphisms other than the identity.  
The lifting of any  partial endomorphism
acts  on $\id_{\Zed_{2n}}$ by composition,  and is empty unless
the composition is defined.    Since $\id_{\Zed_{2n}}$ is surjective, we deduce 
that the lifting to  $\D(\Zed_{2n})$ of any non-total partial endomorphism is the empty map.  Arguing similarly,  the 
lifting of each $\eq_m$  is the diagonal relation.  
\end{proof}

%%%%%%%%%%%%%%%%%%%%%%%%%%%%%%%%%%%%%%%%%%%%%%%%%%%%%%%%

We now put forward a candidate $(\Y,\mu, \nu)$ for the TS-configuration required in Step~2.
  The definition  
 is very similar to the one used for the odd case in Section~\ref{Sec:TSM-SugOdd},
the key difference being that elements with zero coordinates cannot  now appear.

A  TS-configuration
 for $\SA_{2n}$ will be   a set of $n$-tuples of elements of $\Zed_{2n}$
to be regarded as a substructure of $\twiddleeven 
^{\!\!n}$ and 
into which we need  to embed a copy of $\D(\Zed_{2n})$.
Define 
   \[
   \Y 
 = \{(a_1,\ldots,a_n)
\mid \exists\,  j \geq 1
 [\forall i\leq j (0 < a_i= a_j) 
\mbox{ and }\forall k\geq j (a_k < a_{k+1})]\}.
\]
Define $\nu \colon \D(\Zed_{2n}) \to \Y$ by $\nu({\id }_{\Zed_{2n}})=(1,2,\ldots,n)$.

Define $\mu$ to be the map assigning  to each  $(a_1,\ldots,a_n)
 \in \twiddleeven 
^{\!\!n}$
the unique element of $\Y$ whose set of coordinates is $\{|a_1|,\ldots,|a_{n}|\}$. Thus, 
if $\{|a_1|,\ldots,|a_n|\}=\{c_1,\ldots,c_i\}$ with $0 < c_1<c_2<\cdots<c_r$, 
then $\mu(a_1,\ldots,a_{n+1})=(c_1,\ldots,c_1,c_2,\ldots,c_r)$.

\begin{prop}[{\bf Step 2}] \label{Prop:X2n-as-TSEven}
 The triple 
$ (\Y,\mu,\nu)$, defined as above,
 is a $\text{\rm TS}$-configuration.
\end{prop}  

\begin{proof} 
The  argument is the same as  that given for  
Proposition~\ref{Prop:X2n1-as-TSOdd} (Step 2, odd case) except that we also 
 need to confirm that 
$\mu$ preserves  $\eq_1, 
\ldots, \eq_{n-1}$.  But this is clear  because each
 $\eq_m$ is a congruence and modulus is a  term function.
\end{proof}

We carry over notation from 
 the odd case with minor adaptation. 
 For each $k\leq n$, 
 there is a unique element of $\Y$ with $\{1,2,\ldots,k\}$ as its set of coordinates, \textit{viz.}  $\bvec{k}:=(1,\ldots ,1,2,\ldots , k)$, in which $1$ appears $n-k$ times.
 Given a general  $n$-tuple  $\bvec{a} = (a_1, \ldots, a_n)$ in~$\Y$,  we let $\sigma(\bvec{a})$
denote the number of different coordinates in $\bvec{a}$.   In particular $\sigma(\bvec{k}) = k$.

\begin{lem}  \label{lem:genclaimEven}
Let $\Y$ be the structure defined above.  Then
$\{\bvec{1},\ldots,\bvec{n}\}$ generates $\Y$.
\end{lem}

\begin{proof}
Let $\bvec{a}=(a_1,\ldots,a_n)\in\Y$ and denote
$\sigma (\bvec{a})$ by~$k$.   
 Corollary~\ref{cor:k-trans}(i) 
supplies $e \PEZ{2n}$ such that $e(\bvec{k}_i) = a_i$ for each~$i$.
\end{proof}

\begin{prop}[{\bf Step 3}]
  $\Y$ is join-irreducible in ${\mathcal S}_\Y$.
  \end{prop}
  \begin{proof}
 Let  $\varphi \colon \Y \to \Y$ be a morphism such that $(1,2,\ldots,n)\in \img\varphi $. 
The required result will follow if we can show that $\varphi $ acts as the identity on~$\Y$.  

Let $x\in \Y$ be such that $\varphi (x)=(1,2,\ldots,n)$. Since $(1,\ldots, n) \notin \dom e$ for any $e \in H$ and it is the only element of~$\Y$ with this property, $x=(1,2,\ldots,n)$.  
Consider now
$\bvec{a}=(a_1,\ldots,a_n)\in \Y\setminus\{(1,2,\ldots,n)\}$ and write $\varphi (\bvec{a})=(b_1,\ldots,b_n)=\bvec{b}$.   
We claim that
$\{a_1, \ldots, a_n\}  = \{b_1, \ldots, b_n\}$.  For this it will be enough to show that $\sigma(\bvec{a}) = \sigma(\bvec{b})$.

We know that
 $\bvec{a}\in \dom f_i$ if and only if 
$a_j\neq i$ for each $j$. Similarly,  $\bvec{a}\in \dom g$ if and only if $a_j\neq 1$ for each $j$.  Since $\varphi $ preserves each $e \in H$ we deduce that 
 $\{b_1,\ldots,b_n\}\subseteq \{a_1,\ldots,a_n\}$. 
Suppose  $ \sigma (\bvec{a} =k<n$.  
We now need to prove 
that
$\sigma(\bvec{b}) \geq k$.

Let  $m:= n-k+1$.   
By Corollary~\ref{cor:k-trans}  there exists $e\in\PEZ{2n}$ such that 
$
(a_1,\ldots,a_n)= e
(m,\ldots, m, m+1, \ldots,n)$.
Then $e(\bvec{a})\eq_m (1,2,\ldots,n)$.
 Since $\varphi $ preserves 
$e$ and $\eq_m$  
it follows that 
\[e(\varphi (\bvec{a}))=\varphi (e(\bvec{a})) 
\eq_m 
\,\varphi (1,2,\ldots,n)=(1,2,\ldots,n).
\]
Hence $e(\varphi (\bvec{a}))= 
(a_1,\ldots,a_m
,m+1,m+2, 
\ldots,n)$, because $\eq_m\cap \bigl( \Zed_{2n}^2 \setminus \Zed_{2m}^2\bigr)$ is the diagonal  
 relation.   
Therefore
 $\sigma(e(\varphi (\bvec{a}))) \geq k$. 
  Since~$e$ is invertible, 
$\sigma((b_1,\ldots,b_n) )=\sigma( \varphi (\bvec{a}))\geq k$. 
It follows that $|\{b_1,\ldots,b_n\}|=k$ 
since $\{b_1,\ldots,b_n\}\subseteq \{a_1,\ldots,a_n\}$.
\end{proof}

\begin{prop}[{\bf Step 5}]\label{Prop:AnEven}
Let $\alg{B}\subseteq\Zed_{2n}^{n}$ be the subalgebra whose elements $(a_1,\ldots,a_n)$ satisfy the following conditions:
\begin{itemize}
\item[{\rm (i)}]
$a_1\in\{1,-1\}$;
\item[{\rm (ii)}] there exists a value of $j$, necessarily unique,  and with $ j\leq n$ such that 
\begin{itemize}
\item[{\rm (a)}]
$|a_k| = 1 $ if $ k\leq j$  and   
{\rm (b)} $g(a_{k+1})=a_{k}$ if $n>k\geq j$.
\end{itemize}

\end{itemize}
Then $\alg{B}$ is a subalgebra of $\Zed_2\times \Zed_4\times\cdots\times \Zed_{2n}$ and it is isomorphic to $\E (\Y)$.
\end{prop}
\begin{proof}
The cases $n=1,2$ follow from straightforward calculations. 

Now fix $n>2$.

We define a map $t\colon \E(\Y)\to \B$ which we shall show is an isomorphism.  
Given  
$x\in \E(\Y)$, let
\[
t(x)=(x(\bvec{1}),x(\bvec{2}),\ldots, x(\bvec{n})).
\]

\noindent {\bf Claim 1:} $t(x)\in \B$.
\begin{proof}
If $| x(\bvec{k})|=1$ for all $k$, the claim is true.
 Assume now that there exists $k$ with $|x(\bvec{k})|\ne 1$. 
Since $\bvec{1}=(1,1,\ldots,1)\in \bigcap_{i=2}^{n-1} \dom f_i$, it follows that
 $x(\bvec{1})\in\bigcap_{i=2}^{n-1} \dom f_i =\{-1,1\}$. Therefore
%, the set 
$
K_x := \{\,k\in [n] \mid |x(\bvec{k})|=1 \mbox{ and } |x(\bvec{k+1})|\ne 1\}\ne \emptyset$.
Let  $j= \min K_x$.
Then $|x(\bvec{i})| = 1$ for $i \leq j$.
Observe that   $\bvec{j+1}\eq_2 \, h(\bvec{j})$, where $h=g^{-1}=f_2\circ\cdots\circ f_n$. Hence 
$$
x(\bvec{j+1})\eq_2 h(x(\bvec{j}))\in \{h(1),h(-1)\}= \{-2,2\}.
$$ %
So  $x(\bvec{j+1})\in\{-2,2\}$.
 Assume that $x(\bvec{j+1})=2$. For $k$ such that $j\leq k\leq n$, 
\[
 x(\bvec{k+1})\eq_{2}\, x(h(\bvec{k})) = h(x(\bvec{k})).
\]
This implies  $x(\bvec{k+1})\notin \{-1,1\}$ and $g(x(\bvec{k+1}))=g( h(x(\bvec(k)))=x(\bvec{k})$.
The case $x(\bvec{j+1})=-2$ is handled similarly.  
\end{proof}

\noindent  {\bf  Claim 2}:  $t\colon \E(\Y) \to \alg{B}$ is an injective homomorphism.
\begin{proof}
By Lemma~\ref{lem:genclaimEven}, the map~$x$ is uniquely determined by the set $\{x(\bvec{1}),\ldots,x(\bvec{n})\}$. Therefore, $t$ is  injective.
 Moreover, 
$t$ is a homomorphism because
it is given coordinatewise. 
\end{proof}

\noindent  {\bf  Claim 3}: The homomorphism $t$ maps $\E(\Y)$  onto $\alg{B}$.  
\begin{proof}   Here the even case is somewhat simpler than the odd one since we
 do not have to contend with the complications of the extra, $(n+1)^\text{1st}$,
coordinate that arose in Proposition~\ref{Prop:AnOdd};  the $n$ coordinates here follow the same pattern as the first~$n$ coordinates  in the odd case.

Let $\bvec{b}=(b_1,\ldots, b_{n})\in B$. 
Define   $x_{\bvec{b} }\colon \Y \to  \twiddleeven 
$ by
$x_{\mathbf{b}}(\mathbf{u})=
\sgn b_{\sigma({\mathbf u})}\cdot 
u_{n-j+1}$,
where $j\in\{1,\ldots,n\}$ is as in (ii).

If
$k \leq j$ then     $\sigma(\bvec{k} )= k$ and 
$
x_{\mathbf{b}}(\bvec{k})=
\sgn a_k\cdot
\bvec{k}_{n+1-j}=\sgn b_k \cdot1=b_k$.
 If $n\geq k>j$ then $\bvec{k}_{n-j+1}=k-j+1$. Also,  by (ii)(b), $|b_k|-(k-j-1)=|g^{k-j-1}(b_k)|=|b_{j+1}|=2$. Hence, $|b_k|=k-j+1$. 
 Then
 $
x_{\mathbf{b}}(\bvec{k})=
\sgn b_{k}\cdot
\bvec{k}_{n+1-j}=\sgn b_k \cdot (k-j+1)=b_k$.  Therefore  $x_{\bvec{b}}(\bvec{i})=b_{i}$ 
for $i\in\{1,\ldots, n\}$. 

It only remains to prove  that $x_{\mathbf{b}}\in \E(\Y)$. 
Let $h\in \PEZ{2n}$ and $\mathbf{u}=(u_1,\ldots,u_{n+1})\in \dom h$.
Since $h$ is injective, $\sigma({\mathbf{u}})=\sigma({h(\mathbf{u})})$. 
Then $h$ preserves $\sgn$ coordinatewise and hence 
\[
x_{\mathbf{b}}(h(\mathbf{u}))=\sgn b_{\sigma(h(\mathbf{u}))}\cdot (h(\mathbf{u}))_{n-j+1}=
\sgn  b_{\sigma({\mathbf u})}\cdot
h(u_{n-j+1})
= h(\sgn  b_{\sigma(\mathbf{u})}\cdot
 u_{n-j+1})
=h(x_{\mathbf{b}}(\mathbf{u})).
\]
Since $\eq_m$  is a congruence and
modulus is a term function, $x_{\bvec{b}}$ preserves~$\eq_m$.
\end{proof}

We have proved that $\E(\Y)$ is isomorphic to~$\B$. \end{proof} 

As a corollary to this we can identify a set of generators for~$\B$; see 
Proposition~\ref{cor:gens}.

%%%%%%%%%%%%%%%%%%%%%%%%%%%%%%%%%%%%%%%%%%%%%%%%%%%%%%%
%%%%%%%%%%%%%%%%%%%%%%%%%%%%%%%%%%%%%%%%%%%%%%%%%%%%%%%

\section{Admissibility algebras for Sugihara algebras:  overview and applications}
\label{sec:sumup}

We have achieved the goal we set in this paper: the determination of the admissibility algebra  for each quasivariety $\SA_k$.  We now go back to our discussion in Section~\ref{sec:intro}   and review our results,  from a computational and from a theoretical standpoint.  
We already argued that using admissibility algebras instead of free algebras 
reduces  the search space and 
 makes  testing  logical rules for admissibility a more tractable problem.   
 But just how big is the improvement? 
To assess this we need to compare, for a general value of  $k$, the relative sizes of $\SA_k$'s admissibility algebra and of the free algebra $\F_{\SA_k} \bigl(\left[\frac{k+1}{2}\right]\bigr)$, where $[\, \cdot \, ]$ denotes integer part.

In Propositions~\ref{Prop:AnOdd}
and~\ref{Prop:AnEven} we described the admissibility algebras for  $\SA_{2n+1}$ and $\SA_{2n}$, respectively, and with~$n$ fixed.  Now we wish to consider the quasivarieties $\SA_k$, for a general  variable,~$k$.  Accordingly, when working with $\SA_k$  
we 
now write the admissibility algebra as $\B_k$ rather than~$\B$ and write  $\Y_k$  for the structure on which the TS-configuration is based.   We write $s:= [\frac{k+1}{2}]$.  
Introduction of~$s$ highlights the similarities between Propositions~\ref{Prop:X2n1-as-TSOdd} 
and Proposition~\ref{Prop:X2n-as-TSEven}.
In both 
cases $\Y_k$ consists of $s$-tuples drawn from the non-negative elements of $\Zed_k$ and the map $\mu$ in the TS-configuration based on $\Y_k$ is defined in the same way whether $k$ is even or odd.

Free algebras can be notoriously large, and this is the case for our quasivarieties.  
 Using \tafa\ we could calculate $|\F_{\SA_3}(2)| $.  We were able not only to confirm the size of  $\F_{\SA_3}(2) $,  but also to determine the  sizes of $\F_{\SA_4}(2)$ and  $\F_{\SA_5}(3)$ exactly. 
This was done by using the natural dualities developed 
above,  knowing that $\E(\twiddle{\Zed_k}^{\!\!s})$ is an $s$-generated free algebra in $\SA_k$: we calculated the number of morphisms from  $\twiddlek 
^{\!\!s}$ into $\twiddlek$, for $k=3,4, 5$, and $6$.
We have not found a general formula for  $|\F(\SA_k(s))|$ but we do   give
   lower bounds for general~$k$, based  on  numbers of particular morphisms from  $\twiddlek 
^{\!\!s}$ into $\twiddlek$.

Proposition~\ref{Prop:AnEven} 
leads to 
a recursive specification  of (the universes of) the algebras~$\B_{2n}$:
\[
B_2 = \{-1,1\}\quad 
\text{and} \quad 
B_{2n} = \{-1,1\}\times \bigl(B_{2n-2} \cup \{(2,3,\ldots,n),(-2,-3,\ldots,n)\}\bigr) \ \ (n > 1).
\]
Hence, using the recurrence  relation $|B_{2n}|=2|B_{2n-2}|+4$
 we can prove that $|B_{2n}|=3 \cdot 2^n-4$.

Comparing the definition of $\B_{2n+1}$ with  that of $\B_{2n}$,
we  see that 
\[
B_{2n+1}= \{(b_1,\ldots,b_n,b_n)\mid (b_1,\ldots,b_n)\in B_{2n}\setminus \{-1,1\}^n\} \cup ( \{-1,1\}^n\times \{-1,0,1\}).
\]
Hence $|B_{2n+1}|=|B_{2n}|-2^{n}+3\cdot 2^n=3 \cdot 2^n-4+2\cdot 2^n=5\cdot 2^n-4$.

Our data are shown in Table~\ref{table:Sugcasestud}.

\begin{table}[ht]
\begin{center}
\begin{tabular}{ccccc}
%	\hline
  	$k$ & $s= \left[\frac{k+1}{2}\right]$ & $ \left|\F_{\SA_k}(s)
\right|\quad $ & \quad $|{\Y
}_k|$ \quad & \qquad   $|
  	 \E( \Y_k)|$ \quad  \  \\[.15cm]
	\hline &&&\\[-.3cm]
	      $3$ &  $2$ & $1\,296$ \ & $3$ & $\ 6$\ \\[.1cm]
     $4$ & $2$ & $20\,736$ \ & $3$ & $\ 8$\ \\[.1cm]
     $5$ &  $3$ & \qquad \quad $2^{44}\cdot 3^{36}\cdot (1+ 
{2^{-3}\cdot 3^{-4}})^6$   \qquad  \quad  & $7$ &  $16$ \\[.1cm]
  $6$ & $3 $ & \qquad \quad $2^{68}\cdot 3^{36}\cdot (1+ 
{2^{-7}\cdot 3^{-4}})^6$   \qquad  \quad  & $7$ &  $20$ \\[.1cm]
	 \ $2n$ \ &  $n$   & $\geq 2^{2^{n+1}}\cdot 2^ {{(n-1)^2}/{2}}\cdot 3^{2^{n}}$  & $2^{n}-1$ &$3\cdot 2^n-4$\\[.1cm]
	 \   $2n+1$ \  &   $n+1$ &   $\geq 2^{2^{n+1}}\cdot 2^ {{n^2}/{2}}\cdot 3^{2^{n+1}}$ \  & $2^{n+1}-1$ & $5\cdot 2^n-4$ 
\end{tabular}
\end{center}
\caption{Cardinalities of free  algebras, admissibility algebras and test spaces\label{table:Sugcasestud}}
\end{table}

Proposition~\ref{cor:gens} is a corollary of Propositions~\ref{Prop:AnOdd} 
and~\ref{Prop:AnEven}, for~$k$ odd and~$k$ even, respectively. 
We may regard $\Y_k$ as a substructure of $\twiddlek^{\!\!s}$.  We denote the natural inclusion map by~$\iota$.

\begin{prop}[generators for the admissibility algebra for $\SA_{k}$]\label{cor:gens}
Fix $k$, either even or odd.
\begin{enumerate}
\item[{\rm (i)}]  $t \circ \E(\iota) \colon \F_{\SA_{k}}(s) \to  \B_k$  is surjective and $\E(\iota)$ acts by restriction on each  coordinate projection~$\pi_j$.
\item[{\rm (ii)}] $\E(\mu ) \circ t^{-1} \colon \B_k \to \F_{\SA_k}(s) $ is an embedding and $\E(\mu) $ acts by composition on  the restriction to $\Y_k$ of each coordinate projection.
\end{enumerate}

For $j = 1, \ldots, s$, define
\begin{alignat*}{2}
&\text{for $k$ even:}  \hspace*{1cm} &
b_j^k&=\begin{cases}
(1,\ldots,1)& \mbox{if }j=1 ,\\
(1,\ldots,1, 2)& \mbox{if }j=2,\\
(1,\ldots,1,2,\ldots,j/2)& \mbox{if }j>2;
\end{cases}
\\[1ex] 
&\text{for $k$ odd:}  &
b_j^k &= \begin{cases}
(1,\ldots,1,0)& \mbox{if }j=1,\\ 
(1,\ldots,1, 1)& \mbox{if }j=2 ,\\ 
(1,\ldots,1,2,\ldots,(j-1)/2, (j-1)/2)& \mbox{if }j>2. 
\end{cases}
\end{alignat*} 
Then $b_j^k =  t(\pi_j{\restriction}_{\Y_k})$ and   
$\{ b^k_1, \ldots, b^k_s\}$ generates $\B_k$.
\end{prop}
\begin{proof} 
Because the duality is strong, the functor $\E$ takes surjections (embeddings) to
embeddings (surjections).  Moreover $\E(\varphi) = \varphi \circ -$ for any 
$\CX$-morphism $\varphi$ 
and it is clear that $\pi_j {\restriction}_{\Y_k} \circ \mu $ equals $\pi_j \circ \mu$
because $\img\mu = \Y_k$.
The claims in (i) and (ii) follow directly.

We 
define $b_j^k := (\pi_j{\restriction}_{\Y})$ for $j = 1,\ldots, s$. 
These elements are  the images under a surjection of the free generators for $\F_{\SA_k}(s)$, and so generate $\B_k$.
 It only remains to confirm, from the description of the elements of $\B_k$ that these elements are as given in the statement of the proposition.  This is a straightforward verification.
\end{proof}

We know from its construction that 
$\B_k\in \ope{S}(\F_{\SA_k}(s))$. 
We can say more about how $\B_k$ sits inside $\F_{\SA_k}(s)$,
\textit{alias} $\E(\twiddlek 
^{\!\!s})$.  
 We present the terms in the free variables $X_1,\ldots,X_s$ that correspond to the map $\pi_j\circ\mu$,  for $j=1,\ldots,s$. To this end
 we need to find terms $G_1,\ldots,G_s$ such that $\mu(a_1,\ldots,a_s)=(G_1(a_1,\ldots,a_s),\ldots,G_s(a_1,\ldots,a_s))$ for each element $(a_1,\ldots,a_s)\in\twiddlek 
^{\!\!s}$. 

Consider   the following Sugihara terms:   
\begin{itemize}
\item
$|x|:= x\to x$ %,
\item
$x\leftrightarrow y := (x\to y)\wedge(y\to x)$ %,
\end{itemize}
and, for  $1\leq i\leq s$,
\begin{itemize}
\item $S_i(X_1,\ldots, X_s)=\bigwedge\bigl\{\bigvee\{|X_i|\mid X_i\in S\}\mid S\subseteq \{X_1,\ldots,X_s\}\mbox{ and } |S|=i\bigr\}$ 
\item$T_i(X_1,\ldots, X_s)$,  defined by
\[
T_i(X_1,\ldots, X_s) = \begin{cases}
S_1 &\mbox{ if }i=1,\\
\neg(S_i\leftrightarrow S_{i-1})\vee S_1 &\mbox{ if }i>1
\end{cases}
\]
\item $G_j(X_1,\ldots, X_s)=S_j(T_1(X_1,\ldots, X_s),\ldots,T_s(X_1,\ldots, X_s))$
\end{itemize}

It is easy to observe that the action of $|x|$ on elements of $\Zed_k$ is precisely to send each $a$ to~$a$ if $a\geq 0$ and to $-a$ if $a<0$. 
To mimic the way $\mu$  operates we need to perform coordinate manipulations 
in order to arrange  the components of $s$-tuples  in increasing order.
This is exactly what the terms $S_j$ enable us to do. 
Finally we must replace a tuple with non-negative components  which increase non-strictly 
by one in  which  only the smallest coordinate is repeated.
To understand how the behaviour  of the term $T_i$ allows us to do this,    first  observe that if $0\leq a\leq b$ then $\neg(a\leftrightarrow b)$ is equal to $-b$ if $a=b$ and to $b$ if $a<b$. 
Hence the 
tuple  $(T_1(a_1,\ldots,a_s),\ldots,T_s(a_1,\ldots,a_s))$ is obtained from the 
tuple $(S_1(a_1,\ldots,a_s),\ldots,S_s(a_1,\ldots,a_s))$ by working iteratively to  
replace any $S_i(a_1,\ldots,a_s)$ that coincides with $S_{i-1}(a_1,\ldots,a_s)$ by $S_1(a_1,\ldots,a_s)$, \textit{viz.} $\bigwedge\{|a_1|,\ldots,|a_s|\}$.
We finally arrive at 
$\mu(a_1,\ldots,a_s) = (G_1(a_1,\ldots,a_s),\ldots, G_s(a_1,\ldots,a_s))$.
Finally, the assignment $b_j^k\mapsto G_j$ extends to an embedding from  $\B_k$ to $\F_{\SA_k}(s)$.
%%%%%%%%%%%%%%%%%%%%%%%%%%%%%%%%%%%%%%%%%%%%%%%%%%%%%%%%%%%

We indicated at the outset that the introduction of the notion of admissibility algebra stemmed from the desire to find a `small' generating algebra of $\ISP(\F_{\SA_k}(s))$ which can be used to check the admissibility of quasi-equations.  Our application of TSM delivers~$\B_k$, the smallest possible such algebra.  
Observe  that we have exhibited an explicit description of each $\B_k$, in terms of a generating set,  of the subalgebra  of $\Zed_{k}^s $ in which it sits.  This can be seen as a standalone presentation of this  algebra, 
not involving the duality methodology we used to arrive at it. 
%%%%%%%%%%%%%%%%%%%%%%% 
Denote by 
$p_s$ the restriction to~$\B_k$ of the projection map of~$\Zed_k^s$ onto the last coordinate.  It is easy to see that~$p_s$ is surjective, so that 
$\Zed_k\in \ope{H}(\B_k)$. 
From this we can prove algebraically  that $\ISP(\B_k)=\ISP(\F_{\SA_k}(s))$ (see \cite[Corollary~22]{MR13}). 
 This direct, duality-free,  argument confirms  that $\B_k$ can be employed for admissibility testing, but does not yield the stronger result that is the minimal such algebra.

%%%%%%%%%%%%%%%%%%%%%%%%%

It is fitting that we should end our paper with a brief discussion of  the problem that inspired it:  admissibility of rules in the context of the logic R-mingle. 
We shall use the following interpretations of terms into quasi-equations. To any formula $\alpha$ we assign the equation $\alpha {\eq}\, \alpha\to \alpha$. 
 (With this interpretation every logical rule becomes a quasi-equation. Moreover this is the interpretation that proves that R-mingle is algebraizable and that its equivalent algebraic semantics are Sugihara algebras;  see Dunn \cite[Section 2]{Du70}, 
Blok and Dziobak \cite{BD86}, and also Font \cite[p.~141]{AAL}.)
We  assume below that $k>4$ since $\SA_2$ is structurally complete (that is, every admissible quasi-equation in $\SA_2$ is valid on every algebra)  and for $k=3,4$ every admissible quasi-equation that is not derivable is so because its antecedents are not unifiable. 
A quasivariety with 
the latter property is called \defn{almost structurally complete}. 
In \cite[Section~4.5]{Roe13}
 $\SA_3$ was proved to be almost structurally
complete. The proof that  $\SA_4$ is almost structurally complete follows from \cite[Theorem 18]{MR13} combined with the fact that its admissibility algebra is $\B_{4}=\Zed_4\times \Zed_2$). 

In Table~\ref{table:Sugcasestud2} we  present a selection  of examples of quasi-equations and decide their admissibility using the admissibility algebras 
we have identified.

 \begin{table}[ht]
\begin{center}
\begin{tabular}{p{6.5cm}cc}
  	\qquad \qquad quasi-equation \qquad \qquad & \multicolumn{2}{c}{admissible?} \\[-0.1cm]
                      &  {even case} & {odd case} \\
\hline
$p\leftrightarrow \neg p\vdash q\leftrightarrow r$& $\checkmark$ & $\times$\\[-0.1cm]
$p, \neg p \vee q  \vdash q$ & $\checkmark$ & $\times$\\[-0.1cm]
 $p, ( p \to |q| )\to (p\to q) \vdash p\to q $ &$\checkmark $ & $\checkmark $\\[-0.1cm]
$q,p \to (q\to r) \vdash p\to r $& $\checkmark  $& $\checkmark$ \\[-0.1cm]
$\neg|p| \vee q \vdash q $& $\checkmark  $& $\times$ 
 \end{tabular}
\end{center}
\caption{Sugihara algebras: admissible rules\label{table:Sugcasestud2}}

\end{table}

Observe that $B_{2n}$ is (isomorphic to) a subalgebra of $B_{2n+1}$. Hence a quasi-equation that is admissible  in every quasivariety $\SA_{2n+1}$ must  also be admissible on every  $\B_{2n}$. The converse does not hold,  as the results in the table show.

We have already noted that the problem of finding axiomatizations for the admissible rules of R-mingle remains open.  Similarly, although Metcalfe~\cite{GM16} 
succeeded in axiomatizing the admissible rules of various
fragments of R-mingle, these did not include 
the extensions $\text{RM}_k$ for $k > 4$. This unsolved 
problem could be approached by first 
axiomatizing the quasivariety generated by the admissibility
algebra  $\B_k$ for~$\SA_k$.  We hope that the recursive specification of this family of algebras might be of assistance 
here.

\section*{References}
\bibliographystyle{elsarticle-num} 

%%%%%   INSERT \input sugihara-arxiv-1709copy.bbl

\end{document}